\documentclass{article}
\usepackage{preprint-layout} 
 
\usepackage[utf8]{inputenc}
\usepackage{authblk}
\usepackage{xcolor}
\usepackage{graphicx}
\graphicspath{{./images/}{./figures-rev1/}}
\usepackage[english]{babel}
\usepackage{lineno}
\modulolinenumbers[5]
\usepackage{amsmath}
\usepackage{amssymb}
\usepackage{amsfonts}
\usepackage{anyfontsize}
\usepackage[ruled,linesnumbered]{algorithm2e}
\usepackage{tikz}
\usetikzlibrary{math}
\usetikzlibrary{spy}
\usetikzlibrary{calc}
\usetikzlibrary{patterns}
\usepackage{subcaption}
\usepackage[perpage]{footmisc}
\usepackage[normalem]{ulem}
\usepackage{a4wide}
\usepackage{url}

\newtheorem{remark}{Remark}[section]

\newcommand{\mcThA}{\mathcal{T}_h^\mathcal{A}}
\newcommand{\mcTh}{\mathcal{T}_h}
\newcommand{\mcThcut}{\mathcal{T}_h^{\textrm{cut}}}
\newcommand{\mcThin}{\mathcal{T}_h^{\textrm{in}}}
\newcommand{\mcThag}{\mathcal{T}_h^{\mathrm{ag}}}

\newcommand{\mcFhgh}{\mathcal{F}_h^{\mathrm{gh}}}
\newcommand{\mcFh}{\mathcal{F}_h}

\newcommand{\Vh}{V_h}

\newcommand{\Vhag}{V_h^{\textrm{ag}}}
\newcommand{\Vhagmin}{V_h^{\textrm{ag},-}}
\newcommand{\tn}{|\mspace{-1mu}|\mspace{-1mu}|}

\newcommand{\Phag}{P^{\textrm{ag}}_h}

\title{Stability and conditioning of immersed finite element methods: analysis and remedies}
\author[1,2]{Frits de Prenter\footnote{Corresponding author: f.deprenter@tudelft.nl}}
\author[3]{Clemens Verhoosel}
\author[3]{Harald van Brummelen}
\author[4]{Mats Larson}
\author[5]{Santiago Badia}
\affil[1]{Faculty of Aerospace Engineering, Delft University of Technology, 2600 AA Delft, The Netherlands} 
\affil[2]{Reden -- Research Development Netherlands, 7555 RJ Hengelo (Ov.), The Netherlands} 
\affil[3]{Department of Mechanical Engineering, Eindhoven University of Technology, 5600 MB Eindhoven, The Netherlands} 
\affil[4]{Department of Mathematics and Mathematical Statistics, Ume\aa\ University, 901 87 Ume\aa, Sweden} 
\affil[5]{School of Mathematics, Monash University, Victoria 3800, Australia}
\date{17 August 2022}

\begin{document}

\maketitle

\begin{abstract}
\noindent This review paper discusses the developments in immersed or unfitted finite element methods over the past decade. The main focus is the analysis and the treatment of the adverse effects of small cut elements. We distinguish between adverse effects regarding the stability and adverse effects regarding the conditioning of the system, and we present an overview of the developed remedies. In particular, we provide a detailed explanation of Schwarz preconditioning, element aggregation, and the ghost penalty formulation. Furthermore, we outline the methodologies developed for quadrature and weak enforcement of Dirichlet conditions, and we discuss open questions and future research directions. 
\end{abstract}

\tableofcontents

\clearpage

{\Large\noindent\textbf{Nomenclature}}
\begin{table}[h!]
  \begin{center}
    \label{tab:table1}
    \begin{tabular}{l | l | l}
      Symbol & Description & Introduced\\
      \hline
      $\mathcal{A}$ & ambient or embedding domain & 2.1 \\
      $\Omega$, $\Omega_h$ & problem (of physical) domain, union of active elements & 2.1 \\
      $\partial\Omega_D$, $\partial\Omega_N$ & Dirichlet and Neumann boundary & 2.1 \\
      $n$, $\partial_n = n \cdot \nabla$ & outer normal, normal derivative & 2.1 \\
      $T$, $T_\Omega$ & element, restriction of cut element $T$ to $\Omega$ & 2.1 \\
      $h$, $h_{T_\Omega}$ & element size, generalized thickness of cut element $T$ & 2.1 \\
      $\eta_i$, $\eta$, $\eta^*$ & volume fraction (resp.\ of $T_i$, smallest, threshold) & 3.1 \\
      $\mcThA$, $\mcTh$, $\mcThcut$, $\mcThin$, $\mcThag$ & mesh (resp.\ ambient, active, cut, internal, aggregated) & 2.1 \& 4.3 \\
      $F$, $\mcFh$ & element face, set of element faces & 4.4 \\
      $\phi$, $\{\phi_i\}_{i=1}^N$ & basis function, basis & 2.1 \\
      $p$ & discretization order & 2.1 \\
      $\mathcal{P}_p(T)$ & polynomial space of order $p$ over $T$ & 4.3 \\
      $\mathcal{Q}_p(T)$ & tens.\ prod.\ space of 1D polyn.\ of order $p$ over $T$ & 4.3 \\
      $\Vh$, $\Vhag$ & approximation space, aggregated approximation space & 2.1 \& 4.3 \\
      $N$ & system size & 2.1 \\
      $u$ & analytical solution & 2.1 \\
      $u_h$ & finite element solution (gen.\ corr.\ to $\mathbf{x}$) & 2.1 \\
      $v_h$, $w_h$ & finite dimensional functions (gen.\ corr.\ to $\mathbf{y}$ and $\mathbf{z}$) & 2.1 \& 3.1 \\
      $\pi_h(\cdot)$ & interpolation operator & 4.5 \\
      $a_h(\cdot,\cdot)$, $l_h(\cdot)$ & bilinear operator \eqref{eq:bilinop}, linear operator \eqref{eq:linop} & 2.1 \\
      $s_h(\cdot,\cdot)$ & ghost-penalty operator & 4.4 \\
      $\llbracket \cdot \rrbracket$, $[\cdot]$ & jump operator \eqref{eq:jumpdef}, difference operator \eqref{eq:shpatch} & 4.4 \\
      $\beta$ & Nitsche parameter & 2.1 \\
      $\tau$ & ghost-penalty parameter & 4.4 \\
      $\| \cdot \|_{\Omega}$ & $L^2$-function norm over domain $\Omega$ & 2.1 \\
      $(\cdot,\cdot)_{\Omega}$ & $L^2$-function inner product over domain $\Omega$ & 4.4 \\
      $\| \cdot \|_{a_h}$ & $a_h(\cdot,\cdot)$ operator norm & 2.1 \\
      $\tn \cdot \tn_{\beta}$ & $\beta$-norm (referred to as energy norm in \cite{Prenter2018}) & 4.1 \\
      $\tn \cdot \tn_h$, $\tn \cdot \tn_{h,\bigstar}$ & unstabilized and stabilized energy norm & 4.5 \\
      $\ell^2(N)$ & vector space of size $N$ & 2.1 \\
      $\mathbf{x}$, $\mathbf{y}$, $\mathbf{z}$ & (coefficient) vectors & 2.1 \& 3.1 \\
      $\mathbf{A}$, $\mathbf{b}$ & system matrix, right-hand-side vector & 2.1 \\
      $\mathbf{B}$ & preconditioner & 3.2 \\
      $\| \cdot \|_2$ & Euclidean vector or matrix norm & 2.1 \\
      $\| \cdot \|_{\mathbf{A}}$ & $\mathbf{A}$-matrix vector energy norm & 2.1 \\
      $\kappa(\cdot)$ & condition number & 3.1 
    \end{tabular}
  \end{center}
\end{table}

\clearpage

\section{Introduction}

Over the past decades, the finite element method (FEM) has become an essential tool in scientific research and engineering. In its standard form, the finite element method requires the construction of a mesh that fits to the boundary of the considered geometry. For problems of practical interest, such a boundary-fitting mesh is constructed using mesh generators. Automatically generating boundary-fitting meshes can lack robustness, however, in the sense that manual intervention is required to, for example, repair distorted elements or non-matching surfaces. This is particularly the case when the geometry is very complex, such as for (possibly non-water-tight) CAD objects with many patches and trimming curves, geometries presented in the form of scan data, or settings in which frequent remeshing is required (for example in fluid-structure interactions). For such problems, immersed methods have been demonstrated to be capable of establishing a more efficient analysis pipeline, as illustrated by the examples in Figure~\ref{fig:examples}.

The pivotal idea of immersed finite element methods is to embed a complex geometry into a geometrically simple ambient domain, on which a regular mesh can be built easily. Basis functions defined on this ambient-domain mesh are then restricted to the problem geometry, after which the solution is approximated by a linear combination of these restricted basis functions. Although this discretization is conceptually straightforward, non-standard treatment of various aspects (which are discussed below) is required on account of the fact that the mesh does not fit the boundaries of the geometry.
Using a wide variety of techniques to treat these non-standard aspects, the concept of immersed finite element methods has successfully been applied in a broad range of fields, such as solid mechanics~\cite{Duester2008,Schillinger2012,Schillinger2012a,Schillinger2012b,Zander2012,Joulaian2014,Ruess2014,Rank2012,Ruess2013}; shell analysis~\cite{Rank2011,Rank2012,Schmidt2012,Ruess2013,Guo2015,Bauer2017,Guo2017}; interface problems~\cite{Dolbow2009,Schillinger2011,Annavarapu2012,Hautefeuille2012,Wadbro2013}; fluid mechanics~\cite{Massing2014,Schott2014,Hansbo2014,Xu2016,Schott2015,Schott2016,Hsu2016,Hoang2017,Massing2018,Winter2018}; fluid-structure interaction~\cite{Bazilevs2012,Burman2014,Rueberg2014,Massing2015,Kadapa2016,Kadapa2017,Wang2017,Wu2017,Kadapa2018}, in particular fluid-structure interaction for biomedical applications~\cite{Hsu2014,Kamensky2015,Hsu2015,Kamensky2017}; scan-based analysis of both man-made and biological materials~\cite{Yang2011,Ruess2012,Duester2012,Verhoosel2015,Varduhn2016,Elhaddad2017,Duczek2015pore,Wuerkner2018,Hoang2019}; shape and topology optimization~\cite{Parvizian2012,Nadal2013,Dijk2013,Bandara2016,Groen2017,Villanueva2017,Burman2018}; and many more.

\begin{figure}[h!]
  \centering
  \begin{subfigure}{.49\textwidth}
   \centering
   \includegraphics[height=4cm]{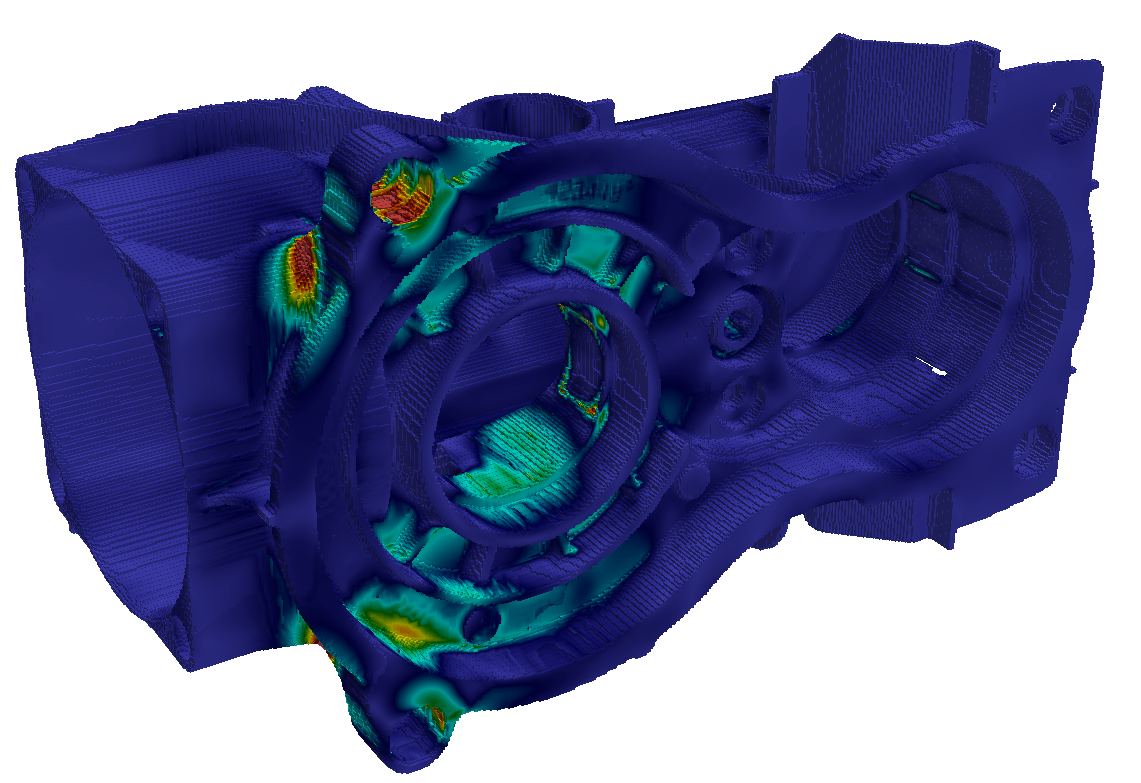}
   \caption{\label{fig:gearbox}}
  \end{subfigure}
  \hfill
  \begin{subfigure}{.49\textwidth}
   \centering
   \includegraphics[height=4.5cm]{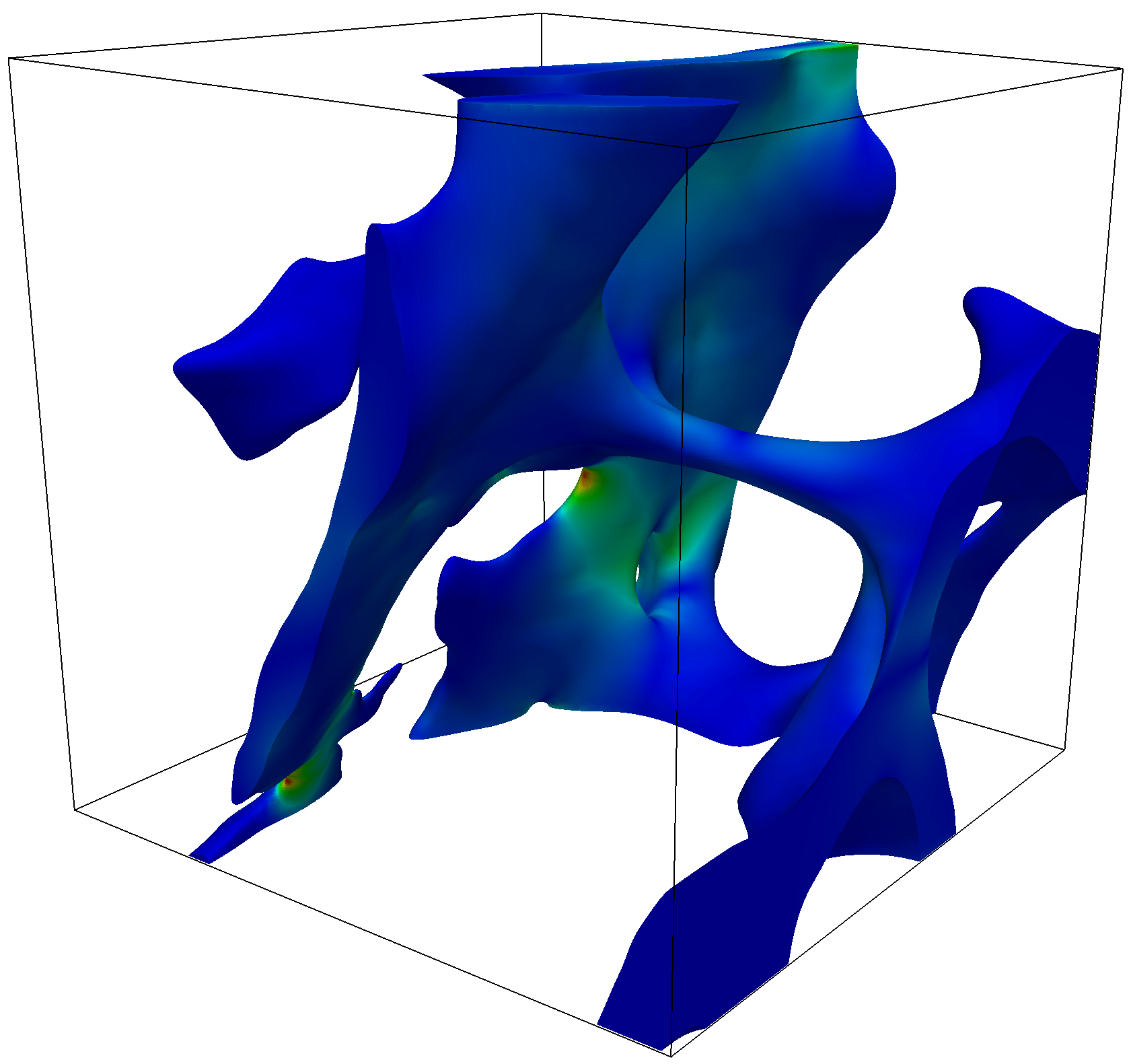}
   \caption{\label{fig:bone}}
  \end{subfigure}
 \caption{Examples of applications of immersed FEM. Figure~(a) shows the stress in an aluminum die cast gearbox housing \cite{Jomo2019}. The geometry is implicitly defined from CT data of a product to investigate stress concentrations around pores. Figure~(b) depicts the stress in a specimen of trabecular bone, rendered from CT data \cite{Prenter2019}.\label{fig:examples}}
\end{figure}

The immersed analysis concept was originally proposed in the context of the finite difference method by Peskin \cite{Peskin1972} in 1972. This immersed boundary method (IBM) and its enhancements have been employed in a wide range of applications ever since (see, \emph{e.g.}, Ref.~\cite{Mittal2005} for a contemporary review). The application of the immersed element concept in the finite element setting can be traced back to the work on the partition of unity method \cite{Melenk1996,Babuska1997}, generally referred to as the generalized or extended finite element method (GFEM \cite{Duarte2000,Strouboulis2000} or XFEM \cite{Belytschko1999,Moes1999}), where elements are cut in order to construct enrichment functions. The concept of cutting finite elements as an unfitted meshing technique was pioneered by Hansbo \cite{hansbo2002unfitted}. This work can be considered as the first instance of an immersed finite element method. The pace in the development and impact of immersed finite element methods increased significantly with the introduction of the finite cell method (FCM) \cite{Parvizian2007,Duester2008,Rank2012,SchillingerRuess2015}, which combines the cut element concept with higher-order basis functions, and CutFEM \cite{BurmanHansbo2012,Hansbo2014,Massing2014,Burman2015,Burman2018codim}, which generally employs the ghost penalty to enhance numerical stability \cite{Burman2010,BurmanHansbo2012}. Besides these prominent immersed FEM techniques, other notable examples are the aggregated finite element method (AgFEM) \cite{Badia2018,Badia2018,badia2021linking}, the Cartesian grid finite element method (cgFEM) \cite{Nadal2013,Marco2015}, weighted extended B-splines (WEB-splines) \cite{Hoellig2001,Hoellig2005} and immersed B-splines (i-splines) \cite{Sanches2011}. Discontinuous Galerkin methods can readily be used on cut meshes after aggregation of cells, since they can be posed on polytopal meshes \cite{Johansson2012}. In recent years, the immersed analysis concept has been considered in conjunction with isogeometric analysis \cite{Hughes2005,Marussig2018Review} (often referred to as iga-FCM \cite{Rank2012} or immersogeometric analysis \cite{Kamensky2015}). In this setting, immersed methods have been demonstrated to be capable of leveraging the advantageous properties of the spline basis functions used in isogeometric analysis, while enhancing the versatility of the simulation workflow for cases where boundary-fitting spline geometries are not readily available, \emph{e.g.}, in scan-based analyses.

Immersed finite element methods are typically confronted by three computational challenges in comparison to standard mesh-fitting finite elements, \emph{viz.}: \emph{i)} the numerical evaluation of integrals over cut elements; \emph{ii)} the imposition of (essential) boundary conditions over the immersed or unfitted boundaries; and \emph{iii)} the stability of the formulation in relation to small cut elements. A myriad of advanced techniques has been developed over the past decades to resolve these challenges, which has made immersed finite element techniques a competitive simulation strategy for a wide range of problems. This review focuses on the third challenge, \emph{i.e.}, the effects of small cut elements on the performance of immersed finite element methods. We restrict ourselves to a high-level consideration of the first two of these challenges, and we refer the reader to, \emph{e.g.}, Ref.~\cite{SchillingerRuess2015} for a detailed review on these topics. 

Small cut elements typically give rise to stability and conditioning problems. In this article we review three prominent methods to resolve these problems, which have been developed in recent years, \emph{viz.}: (Schwarz) preconditioning, ghost-penalty stabilization (commonly used in CutFEM), and aggregation of cut elements (commonly used in AgFEM). In the discussion of these techniques, it is important to note that small cut elements do not only affect the conditioning of the problem, but also the accuracy of the solution. More specifically, direct application of Nitsche’s method for the weak imposition of essential boundary conditions requires a Nitsche parameter that scales inversely proportional with the size (in particular the thickness) of the cut element. As observed in, \emph{e.g.}, \cite{Prenter2018}, the unboundedness of this parameter can deteriorate the accuracy of the solution. When using preconditioning techniques to resolve the small-cut-element problem, it is important to realize that these techniques do not resolve this potential issue regarding the accuracy. In contrast to preconditioning techniques, ghost penalty stabilization and aggregation ensure well-posedness with a Nitsche parameter inversely proportional to the ambient-domain mesh size, in addition to controlling the condition number. This makes the immersed finite element approximation using these stabilization techniques robust with respect to the cut element configurations, and preserves the error estimates of boundary-fitted finite element methods. It should be noted that, besides ghost penalty stabilization and aggregation, several other techniques have been developed to resolve the stability problems on small cut elements. However, in general, these do not simultaneously resolve the conditioning problem. An overview of these techniques is also presented in this review.

This review article has three objectives, \emph{viz.}: \emph{(i)} to clarify that the small-cut-element problem is multi-faceted; \emph{(ii)} to present the different techniques in an accessible theoretical framework, so that the implications of practical choices become apparent to a non-expert audience; and \emph{(iii)} to provide a comparison of the techniques for conditioning and stabilization of immersed finite element methods. With the myriad of immersed finite element techniques available, naturally comes the luxury problem of choosing which technique is most suitable in a particular situation. The rigorous theoretical underpinning of the methods as reviewed in this article is instrumental to aiding in the selection of a particular method, as it provides a fundamental understanding of the relation between small cut elements, conditioning, and stability and accuracy.

It should be noted that immersed methods are not the only techniques to create a robust workflow to deal with complex or implicitly defined geometries. One alternative is the shifted boundary method \cite{Main2018II,Main2018I,Atallah2020,Atallah2021}. The shifted boundary method aims to replace an immersed problem with a similar boundary-fitted problem on the interior element mesh, by projecting boundary conditions from the real (unfitted) boundary to the interior element boundaries. This concept was already introduced in~\cite{Bramble1972}, and is also applied 
in~\cite{Burman2018}. This method bypasses the aforementioned three computational challenges of immersed finite element methods, but instead introduces other 
challenges, such as a non-trivial treatment of boundary conditions (including the projection of the boundary data), and a non-obvious geometrical treatment. Another approach is to use hybridizable techniques on unfitted meshes~\cite{burman2018hho,badia2021conditioning}. Hybridizable methods can naturally be posed on polytopal meshes, giving additional geometrical flexibility compared to standard finite element methods. As a result, these methods can readily be used on the meshes obtained after the intersection of the boundary representation and background mesh and possibly after aggregation of elements. The impact of small cut elements and small cut faces (these schemes add unknowns on the mesh skeleton) on stability and condition numbers has only been studied very recently in the context of hybrid high-order methods, see~\cite{badia2021conditioning}. A detailed discussion of these alternative techniques is beyond the scope of this work.

This article is structured as follows. In Section~\ref{sec:immersedMethods} we introduce the basic formulation based on a model problem and discuss the developments regarding the three computational challenges associated with immersed finite element methods. In Section~\ref{sec:conditioningFull} a compact analysis of the ill-conditioning problem is presented, and Schwarz preconditioning is discussed as a natural technique to resolve this problem. Section~\ref{sec:stabilization} then considers stabilization techniques, specifically the ghost-penalty method and aggregation technique, and presents a theoretical framework required to analyze the stability and conditioning properties of these techniques. A discussion on the current state of the field and concluding remarks are finally presented in Section~\ref{sec:discussion}. Note that the results presented in this manuscript are reproduced from previous publications by the authors, which are referenced in the text or in the captions.

\vspace{1mm}

\section{Immersed finite element methods}\label{sec:immersedMethods}

\vspace{1mm}

In this section we introduce the immersed finite element framework. Section~\ref{sec:method_formulation} presents the concept of immersed finite element methods and specifies the formulations based on a model problem. In Section~\ref{sec:method_challenges} the most prominent challenges of immersed finite element methods, compared to standard boundary-fitted finite element methods, are discussed.

\vspace{1mm}

\subsection{Formulations}\label{sec:method_formulation}

\vspace{1mm}

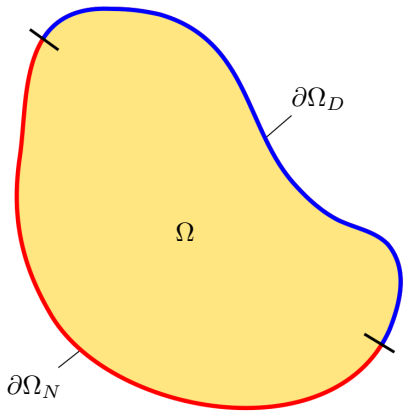
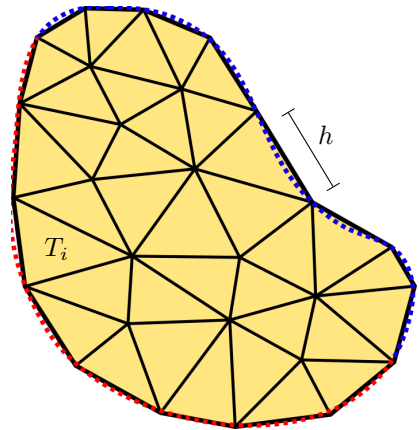
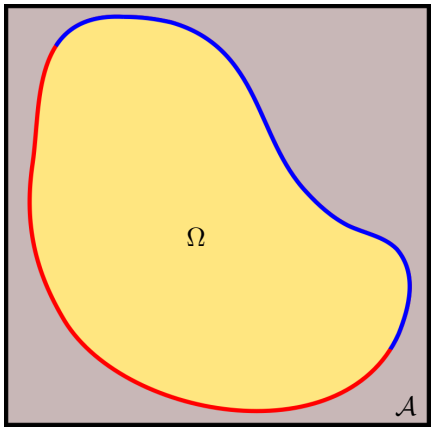
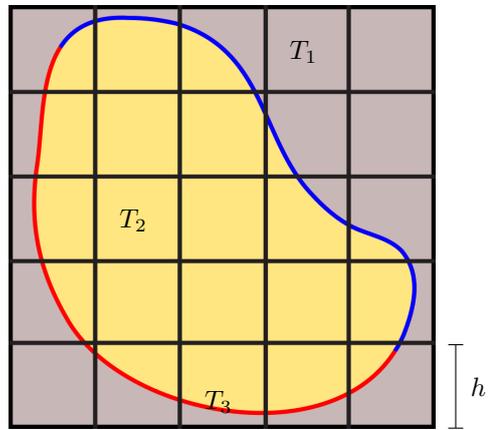
\begin{figure}
     \centering
     \begin{subfigure}[b]{.4\textwidth}
         \centering
          \begin{tikzpicture}
           \node[anchor=south west] at (0mm,0mm) {\includegraphics[height=56mm]{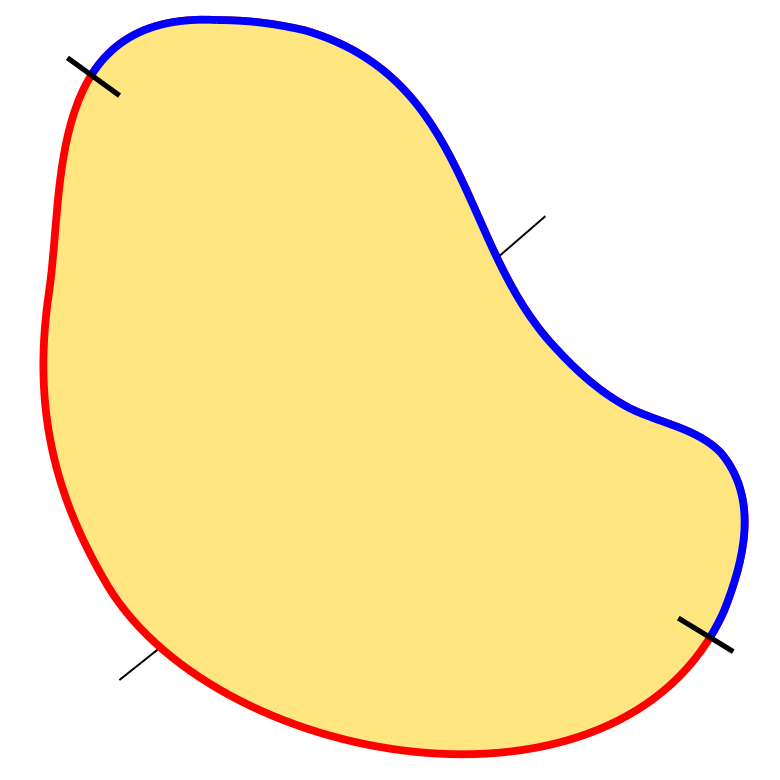}};
           \node at (6.6mm,5.8mm) {$\partial\Omega_N$};
           \node at (44mm,44mm) {$\partial\Omega_D$};
           \node at (26.4mm,26.4mm) {$\Omega$};
          \end{tikzpicture}      
         \caption{Physical domain $\Omega$ with Dirichlet boundary $\partial\Omega_D$ and Neumann boundary $\partial\Omega_N$}
         \label{fig:domainOmega}
     \end{subfigure}
     \hspace{20mm}
     \begin{subfigure}[b]{.4\textwidth}
         \centering
          \begin{tikzpicture}
           \node[anchor=south west] at (0mm,0mm) {\includegraphics[height=56mm]{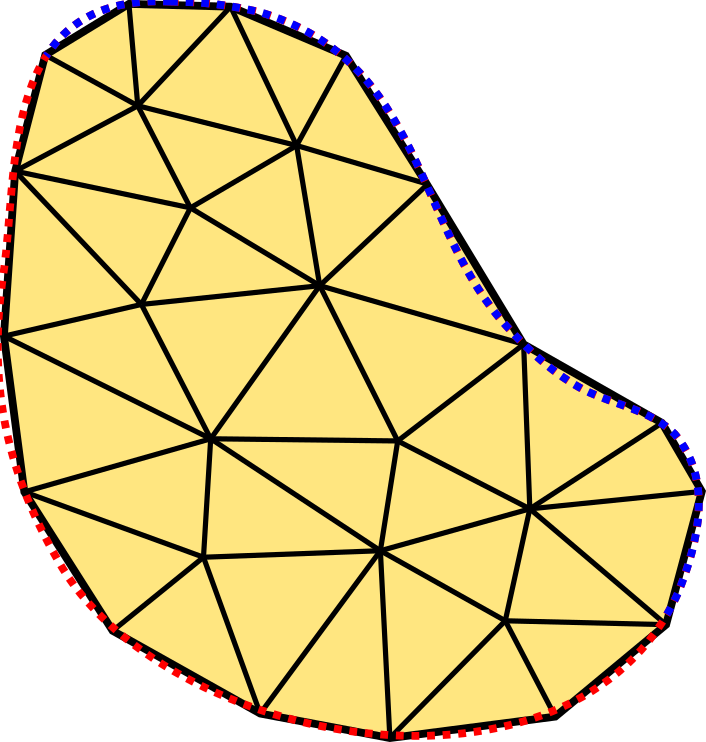}};
           \node at (7.2mm,25mm) {$T_i$};
           \draw[|-|, thin] (37.6mm,43.2mm) -- (43.6mm,33.2mm) node [black,midway,yshift=2mm,xshift=2mm,rotate=0] {$h$};
          \end{tikzpicture}      
         \caption{Boundary-fitted discretization $\mcTh$ that approximates the geometry $\Omega$}
         \label{fig:boundaryFittingDiscretization}
     \end{subfigure}
     \\
     \begin{subfigure}[b]{.4\textwidth}
         \centering
          \begin{tikzpicture}
           \node[anchor=south west] at (0mm,0mm) {\includegraphics[height=56mm]{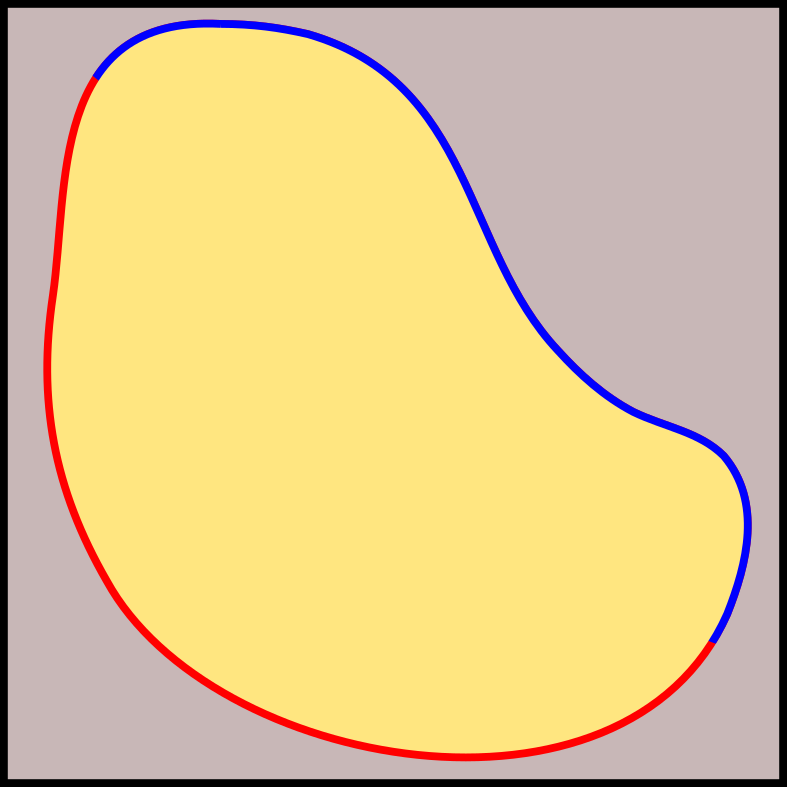}};
           \node at (26.4mm,26.4mm) {$\Omega$};
           \node at (54mm,4mm) {$\mathcal{A}$};
          \end{tikzpicture}      
         \caption{Embedding of the physical domain $\Omega$ in the ambient domain $\mathcal{A}$}
         \label{fig:embeddingAmbient}
     \end{subfigure}
     \hspace{20mm}
     \begin{subfigure}[b]{.4\textwidth}
         \centering
          \begin{tikzpicture}
           \node[anchor=south west] at (0mm,0mm) {\includegraphics[height=56mm]{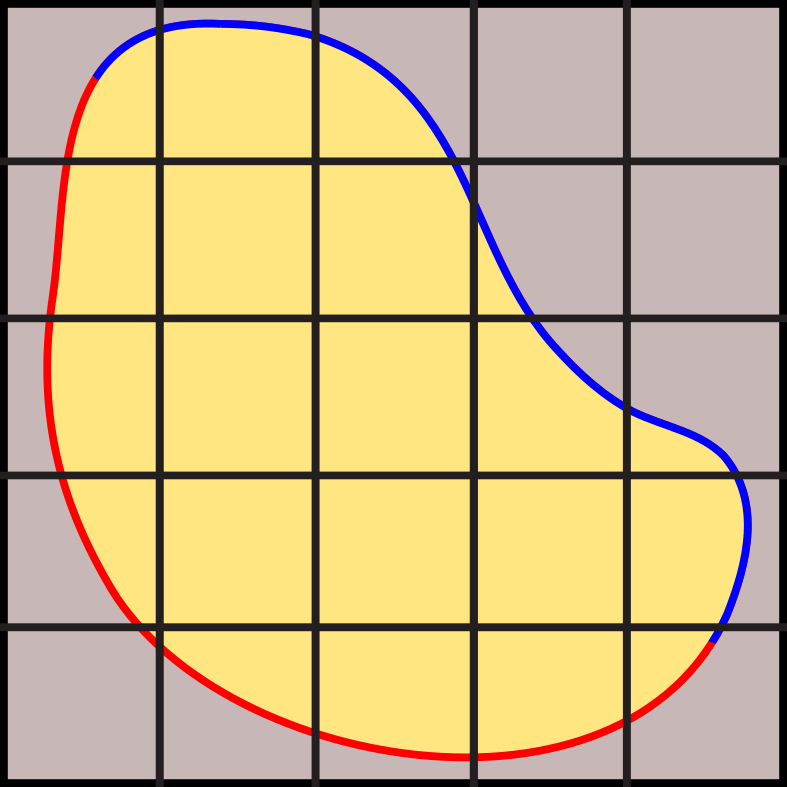}};
           \node at (40mm,51.2mm) {$T_1$};
           \node at (17.6mm,28.8mm) {$T_2$};
           \node at (28.8mm,4.8mm) {$T_3$};
           \draw[|-|, thin] (60mm,1.2mm) -- (60mm,12.4mm) node [black,midway,yshift=0mm,xshift=3mm,rotate=0] {$h$};
          \end{tikzpicture}      
         \caption{Immersed discretization with $\{T_1,T_2,T_3\}\subset\mcThA$, $\{T_2,T_3\}\subset\mcTh$, and $T_3\in\mcThcut$}
         \label{fig:immersedDiscretization}
     \end{subfigure}
\caption{Schematic representation the domain $\Omega$, a boundary-fitted discretization, the embedding in $\mathcal{A}$, and an unfitted discretization.}
\label{fig:poisson}
\end{figure}

As a model problem, we consider the Poisson equation on the domain $\Omega \subset \mathbb{R}^{d}$, with $d\in\{2,3\}$ being the number of spatial dimensions (Figure~\ref{fig:poisson}). The domain $\Omega$ has a boundary $\partial \Omega$ with outward-pointing unit normal vector $n$. The boundary consists of complementary parts $\partial\Omega_D$ and $\partial\Omega_N$ on which Dirichlet (or essential) and Neumann (or natural) conditions are prescribed with boundary data $g_D$ and $g_N$, respectively. The field variable $u:\Omega \rightarrow \mathbb{R}$ is subject to the strong formulation
\vspace{1mm}\begin{equation}
\text{(Strong)} \begin{cases}
  -\Delta u = 0 & \text{in } \Omega\\
  u = g_D & \text{on } \partial\Omega_D\\
  \partial_n u = g_N & \text{on } \partial \Omega_N
 \end{cases}
 \label{eq:poissonstrong}
\end{equation}\vspace{1mm}
where $\partial_n= n \cdot \nabla$ denotes the normal gradient operator.

In boundary-fitted finite elements, the domain $\Omega$ is subdivided into elements $T$ (generally at the expense of a geometrical error), which together comprise the mesh $\mcTh$ (Figure~\ref{fig:boundaryFittingDiscretization}). The size of element $T$ is denoted by $h_T$ and a global mesh size parameter is defined as $h = \mathrm{max}_{T\in\mcTh} h_T$. This manuscript only considers quasi-uniform discretizations, that is $h_T \approx h$ $\forall T \in \mcTh$. On the mesh, shape function are introduced, which form the basis of the approximate solution $u_h$. Integration is performed via standard quadrature rules for polynomials on simplices or hypercubes. Homogeneous Dirichlet conditions are usually imposed strongly by removing the basis functions with support on $\partial \Omega_D$. The span of the remaining basis functions is referred to as $V_{h,0} \subset H^1_0(\Omega)$. The Bubnov-Galerkin finite element method imposes inhomogeneous boundary conditions through a so-called lifting function, $\ell(g_D) \in H^1(\Omega)$, which is generally constructed using the removed basis functions. The Bubnov-Galerkin finite element formulation corresponding to the strong form \eqref{eq:poissonstrong} can then be condensed into
\begin{equation}
 \text{(Boundary-fitted FE)}\begin{cases}
  \text{find } u_h = w_h + \ell(g_D) \text{ for }  w_h \in V_{h,0}  \text{ such that:}\\
 a(w_h,v_h) = l(v_h) \qquad \forall v_h \in V_{h,0}
 \end{cases}
 \label{eq:fittingfem}
\end{equation}
where the symmetric bilinear form $a(w_h,v_h)=\int_\Omega \nabla w_h \cdot \nabla v_h \, \text{d}V$ is continuous and coercive on $H^1_0(\Omega)$, and the linear form $l(v_h)=\int_{\partial\Omega_N}  g_N v_h \, \text{d}S-\int_\Omega \nabla \ell(g_D) \cdot \nabla v_h \, \text{d}V$ is also continuous on $H^1_0(\Omega)$. These conditions are sufficient to guarantee well-posedness of the problem \eqref{eq:fittingfem} to find $w_h \in H^1_0(\Omega)$ and $u_h \in H^1(\Omega)$ for sufficiently smooth boundary data \cite{ErnGuermond,Evans}. Because of the coercivity of the bilinear form, it induces the operator norm $\| v_h \|^2_{a} = a(v_h, v_h)$, which is identical to the $H^1_0(\Omega)$-seminorm for the boundary-fitted finite element formulation of the Poisson problem.

In immersed finite element methods, the domain $\Omega$ is embedded into an ambient domain $\mathcal{A}$ (Figure~\ref{fig:embeddingAmbient}). Instead of generating a boundary-fitted partitioning of the domain $\Omega$, the ambient domain is partitioned by the background mesh $\mcThA$ (Figure~\ref{fig:immersedDiscretization}). Since the ambient domain is geometrically simple, mesh generation is trivial. But, as the element boundaries do not coincide with the boundaries of the domain $\Omega$, evaluating integrals requires dedicated procedures (see Section~\ref{sec:quadrature}). Hence, the complexity of the geometry is essentially captured by the integration procedure, instead of by the mesh generation. This geometrical simplification also has implications on the Galerkin formulation.

To specify the immersed finite element formulation, we define the active mesh $\mcTh$ as the set of elements $T$ which intersect the domain $\Omega$, \emph{i.e.},
\begin{align}
 \mcTh = \{  T \in \mcThA \mid T_\Omega \neq \emptyset \}
 \label{eq:activemesh}
\end{align}
where $T_\Omega = T \cap \Omega$.
Furthermore, the active domain is defined as the union of all the active elements $\Omega_h = \cup_{T \in \mcTh} T \supseteq \Omega$. The set of active elements that are cut by (and therefore intersect) the boundary is defined as
\begin{align}
 \mcThcut = \{ T \in \mcTh \mid  T \setminus T_\Omega \neq \emptyset \}
\end{align}
and the set of internal elements that are fully supported on the domain $\Omega$ is defined as
\begin{align}
 \mcThin = \{ T \in \mcTh \mid  T \subset \Omega \}
\end{align}

Analogous to the boundary-fitted case, shape functions are introduced on the active mesh $\mcTh$ and an approximate solution $u_h$ is formed as a linear combination of the restriction of the shape functions to $\Omega$. The space spanned by these functions is referred to as $V_h$. In contrast to the boundary-fitted finite element formulation, in the unfitted setting it is generally not feasible to create a finite dimensional subspace of $V_{h,0} \subset H^1_0(\Omega)$, which is needed to strongly impose Dirichlet conditions. Instead, Dirichlet conditions are generally imposed weakly. The most common approach for weakly enforcing Dirichlet boundary conditions is via Nitsche's method \cite{Nitsche1971}, which gives rise to the immersed finite element formulation
\begin{equation}
\text{(Immersed FE)}\begin{cases}
  \text{find } u_h \in V_{h}  \text{ such that:}\\
 a_h(u_h,v_h) = l_h(v_h) \qquad \forall v_h \in V_{h}
 \end{cases}
 \label{eq:immersedfem}
\end{equation}
where the bilinear and linear form are defined as
\begin{subequations}
\begin{align}
a_h(u_h,v_h) &= \int_{\Omega} \nabla u_h \cdot \nabla v_h {\rm d}V + \int_{\partial\Omega_D} \left(\beta u_h v_h - u_h \partial_n v_h - v_h \partial_n u_h\right) {\rm d}S\label{eq:bilinop}\\
l_h(v_h) &= \int_{\partial\Omega_N} g_N v_h {\rm d}S + \int_{\partial\Omega_D} \left(\beta g_D v_h - g_D \partial_n v_h\right) {\rm d}S\label{eq:linop}
\end{align}
\label{eq:immersedOperator}%
\end{subequations}
The parameter $\beta$, commonly referred to as the Nitsche or penalty parameter, must be chosen large enough to ensure coercivity of the bilinear form $a_h(\cdot,\cdot)$ in the discrete space $V_h$. A sufficient condition is $\beta > \max_{v_h \in V_{h}} \| \partial_n v_h \|_{\partial\Omega_D}^2/\| \nabla v_h \|_{\Omega}^2$, which can be computed by solving a generalized eigenvalue problem \cite{Embar2010}. This can result in arbitrarily high values of $\beta$ over the entire boundary $\partial\Omega_D$. It is therefore common to select a local, element-wise, Nitsche parameter satisfying $\beta|_T > \max_{v_h \in V_{h}|_T} \| \partial_n v_h \|_{T\cap\partial\Omega_D}^2/\| \nabla v_h \|_{T_\Omega}^2$ by solving a local generalized eigenvalue problem for each element. Based on dimensional considerations, one can infer that~$\beta|_T$ should be inversely proportional to a generalized thickness, $h_{T_\Omega}$, of $T_\Omega$ normal to $T\cap\partial\Omega_D$ \cite{Warburton2003,Evans2013}. We refer to $h_{T_\Omega}$ as a \emph{generalized} thickness, because $T_\Omega$ can be irregularly shaped and, consequently, a well-defined length scale cannot generally be provided. While, with an element-wise parameter, a single small cut element does not result in a high value of $\beta$ over the entire boundary, an element-wise parameter is still locally unbounded, which can lead to numerical issues \cite{Prenter2018}. With appropriate stabilization (see Section~\ref{sec:stabilization}), the properties of the immersed formulation revert to those of the boundary-fitted setting considered in the original paper by Nitsche \cite{Nitsche1971}, and a global parameter inversely proportional to the mesh width of the background mesh suffices. In the remainder of this manuscript, both element-wise and global Nitsche parameters will be indicated by $\beta$ and the notation $a_h(\cdot,\cdot)$ and $l_h(\cdot)$ will be used in both stabilized and unstabilized formulations, with the choice of the Nitsche parameter following from the context. Similar to the boundary-fitted setting, the coercive and bounded bilinear form~$a_h$ induces an equivalent operator norm $\| v_h \|^2_{a_h} = a_h(v_h, v_h)$ in the discrete space $V_h$. This operator norm coincides with the matrix energy norm of a corresponding coefficient vector, which renders it useful in the analysis of condition numbers of linear-algebraic systems emerging from immersed formulations.

A myriad of immersed finite element variants to the problem \eqref{eq:immersedfem}-\eqref{eq:immersedOperator} exist, the most prominent of which will be discussed in Section~\ref{sec:dirichlet}. A particularly noteworthy variation to \eqref{eq:immersedOperator} is the penalty method, which corresponds to the case where the terms that involve normal gradients in the above-mentioned operators are left out. This makes the method formally inconsistent with the strong form \eqref{eq:poissonstrong}, but simplifies the implementation and ensures coercivity independent of the parameter $\beta$. In this review, we do not discuss the penalty method in detail and focus on the application of Nitsche's method. In general, the use of the penalty method instead of Nitsche's method has a negligible impact on the conditioning, but does affect the accuracy of the solution.

To solve the immersed finite element formulation \eqref{eq:immersedfem} numerically, it is recast as a linear algebra problem
\begin{align}
\mathbf{A} \mathbf{x}  = \mathbf{b}
\label{eq:linearsystem}
\end{align}
where the components of the matrix $\mathbf{A} \in \mathbb{R}^{N \times N}$ and right hand side vector $\mathbf{b} \in \mathbb{R}^N$ correspond to
\begin{align}
\mathbf{A}_{ij} &= a_h(\phi_j,\phi_i)  &   b_{i} &= l_h(\phi_i)
\label{eq:matrixcomponents}
\end{align}
with $\phi_i$ the $i$-th basis function and $N$ the number of dimensions of the finite dimensional function space $V_h$. The approximate solution $u_h$ is given by
\begin{equation}
 u_h = \sum_i x_i\phi_i 
\end{equation}
with $x_i$ the components of the coefficient vector $\mathbf{x} \in \ell^2(N)$ in equation \eqref{eq:linearsystem}. In the remainder we employ two norms for this coefficient vector, \emph{viz.} the $\ell^2$ vector norm $\| \mathbf{x} \|_2^2 = \mathbf{x}^T \mathbf{x}$ and, as $\mathbf{A}$ is Symmetric Positive Definite (SPD), the matrix energy norm
\begin{align}
  \| \mathbf{x} \|_{\mathbf{A}}^2 = \mathbf{x}^T \mathbf{A} \mathbf{x}
\end{align}
Note that, on account of \eqref{eq:matrixcomponents}, this matrix energy norm is equal to the operator norm, \emph{i.e.}, $\| \mathbf{x} \|_{\mathbf{A}} = \| u_h \|_{a_h}$.

In the remainder of this manuscript we focus our presentation on the single field Poisson problem \eqref{eq:poissonstrong}, discretized by quasi-uniform meshes with $C^0$-continuous piecewise polynomial basis functions of order $p$. The presented analyses and methods naturally extend to vector-valued problems that can be expressed as a Cartesian product of scalar fields (\emph{i.e.}, one field per space dimension). Mixed formulations and discretizations with local refinements are not discussed in detail, but, unless otherwise specified, the provided insights also carry over mutatis mutandis to these cases. Curl-conforming and div-conforming problems (\emph{i.e.}, posed in $H(\text{curl};\Omega)$ or $H(\text{div};\Omega)$, respectively), suffer from similar problems with stability and conditioning on small cut elements as the grad-conforming problems (posed in $H^1(\Omega)$) considered here. With respect to conditioning, maximum continuity splines as commonly used in isogeometric analysis \cite{Hughes2005,Cottrell2009} can behave differently from $C^0$-continuous bases. Therefore, Section~\ref{sec:conditioningFull} also considers B-spline bases as a special case. 

\subsection{Challenges in immersed finite elements}\label{sec:method_challenges}
Although the immersed finite element formulation introduced above fits within the general framework of the conventional finite element method, the application of immersed finite elements involves the consideration of various specific challenges. In this section we discuss the most prominent of these, \emph{viz.}: \emph{(i)} numerical evaluation of integrals over cut elements, \emph{(ii)} imposition of Dirichlet boundary conditions, and \emph{(iii)} stability and conditioning of the formulation.

\subsubsection{Cut-element integration}\label{sec:quadrature}

In boundary-fitted FEM, integrals over the domain $\Omega$ are split into element-wise integrals. The elements generally correspond to polygons such as simplices or hypercubes, and the integrands are usually element-wise polynomials. For this reason, standard quadrature rules can readily be used. Integration in immersed methods is more involved, as the integration procedure should adequately approximate integrals on, in principle, arbitrarily shaped cut elements. This problem closely relates to the special treatment of discontinuous integrands in enriched finite element methods such as XFEM and GFEM.

In immersed finite element methods the geometry representation is independent of the mesh. In general, there are two ways to represent the geometry, \emph{viz.}\ implicit representations (\emph{e.g.}, voxel data \cite{Elhaddad2017,Korshunova2020} or a level set function \cite{Verhoosel2015}) and explicit boundary representations (\emph{e.g.}, spline surfaces \cite{Schillinger2012}, B-rep objects \cite{Rank2012}, or isogeometric analysis on trimmed CAD objects \cite{Schmidt2012,Ruess2014}). For all geometry representations, dedicated techniques are required to evaluate volume integrals over the intersection of active elements with the domain. A myriad of such integration procedures has been developed over the years in the context of immersed FEM (see \cite{Abedian2013,SchillingerRuess2015} for reviews/comparisons) and enriched FEM (see \cite{Fries2010}), an overview of which is presented below. Techniques to integrate over trimmed boundaries are strongly dependent on the geometry description. In explicit representations, the geometry description itself can sometimes be leveraged, whereas generally a boundary-reconstruction procedure is required in the case of implicit boundary representations. The reader is referred to, \emph{e.g.}, Refs.~\cite{SchillingerRuess2015} for a discussion regarding the various techniques to handle different geometry representations, and the related problem of integrating over unfitted boundaries.

\begin{figure}
  \centering
  \includegraphics[width=0.3\textwidth]{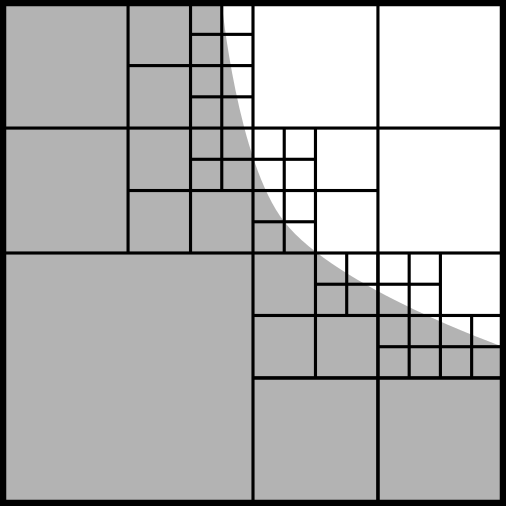}
  \vspace{3mm}
  \caption{Illustration of the octree procedure to integrate cut elements, in which integration sub-cells that intersect with the immersed boundary are recursively bisected.}
  \label{fig:octree}
  \vspace{2mm}
\end{figure}

The cut element volume integration techniques can be categorized as:
\begin{itemize}
\item \textit{Octree subdivision:} The general idea of octree (or quadtree in two dimensions) integration is to capture the geometry of a cut element by recursively bisecting sub-cells that intersect with the boundary of the domain, as illustrated in Figure~\ref{fig:octree}. At every level of this recursion, sub-cells that are completely inside the domain are preserved, while sub-cells that are completely outside of the domain are discarded. This cut element subdivision strategy was initially proposed in the context of the finite cell method (FCM) in \cite{Duester2008} and is generally appraised for its simplicity and robustness with respect to cut element configurations. Octree integration has been widely adopted in immersed FEM, see, \emph{e.g.}, \cite{SchillingerRuess2015,Duester2017,Kamensky2015,Verhoosel2015}. Various generalizations and improvements to the original octree procedure have been proposed, of which the consideration of tetrahedral cells  \cite{Varduhn2016,Stavrev2016}, the reconstruction of the unfitted boundary by tessellation at the lowest level of bisectioning \cite{Verhoosel2015}, and the consideration of variable integration schemes for the sub-cells \cite{Abedian2013}, are particularly noteworthy. Despite the various improvements to the original octree strategy, a downside of the technique remains the number of integration sub-cells (and consequently the number of integration points) that result from the procedure, especially in three dimensions and with high-order bases, where the refinement depth needs to be increased under mesh refinement to reduce the integration error with the same rate as the approximation error \cite{Divi2020}.
 
\item \textit{Cut element reparameterization:} Accurate cut element integration schemes can be obtained by modifying the geometry parameterization of cut elements in such a way that the immersed boundary is fitted. This strategy was originally developed in the context of XFEM by decomposing cut elements into various sub-cells with only one curved side and then to alter the geometry mapping related to the curved sub-cell to obtain a higher-order accurate integration scheme \cite{Fries2010}. This concept has been considered in the context of implicitly defined geometries (level sets) \cite{Fries2016,Fries2017,Omerovic2017,Fries2018}, the NURBS-enhanced finite element method (NEFEM) \cite{Sevilla2008,Sevilla2011}, the Cartesian grid finite element method (cgFEM) \cite{Nadal2013,Marco2015}, and the mesh-transformation methodology presented in~\cite{Lehrenfeld2016}. In the context of the finite cell method, the idea of cut element reparameterization has been adopted as part of the \emph{smart octree integration} strategy \cite{Kudela2015,Kudela2016,Hubrich2017}, where a boundary fitting procedure is employed at the lowest level of octree bisectioning in order to obtain higher-order integration schemes for cut elements with curved boundaries. Reparameterization procedures have the potential to yield accurate integration schemes at a significantly lower computational cost than octree procedures, but generally compromise in terms of robustness with respect to cut element configurations.

\item \textit{Polynomial integration}: Provided that one can accurately evaluate integrals over cut elements (for example using octree integration), it is possible to construct computationally efficient integration rules for specific classes of integrands. In the context of immersed finite element methods, it is of particular interest to derive efficient cut-element integration rules for polynomial functions. The two most prominent methods to integrate polynomial functions over cut elements are \emph{moment fitting techniques} \cite{Mousavi2011,Joulaian2016,Hubrich2017,Hubrich2019,Duester2020}, in which integration point weights and (possibly) positions are determined in order to yield exact quadrature rules, and \emph{equivalent polynomial methods} \cite{Ventura2006,Abedian2019}, in which a non-polynomial (\emph{e.g.}, discontinuous) integrand is represented by an equivalent polynomial which can then be treated using standard integration procedures. Such methods have been demonstrated to yield efficient quadrature rules for a range of scenarios. A downside of such techniques is the need for the evaluation of the exact integrals (using an adequate cut-element integration procedure) in order to determine the optimized integration rules. This can make the construction of such quadrature rules computationally expensive, which makes them more suitable in the context of time-dependent and non-linear problems (with fixed boundaries and interfaces), for which the construction of the integration rule is only considered as a pre-processing operation (for each cut element) and the optimized integration rule can then be used throughout the simulation. In the group of Dominik Schillinger at TU Darmstadt, work is currently being done to use neural networks for the computation of integration point weights to accelerate this process. 

\item \textit{Dimension-reduction of integrals}:
Depending on the problem under consideration, it can be possible to reformulate volumetric integrals over cut elements by equivalent lower-dimensional integrals. This approach is advantageous from a computational effort point of view, as the equivalent integrals are generally less costly to evaluate. A reformulation of volume integrals in terms of boundary integrals has been proposed in the context of XFEM in \cite{Ventura2009} and in the immersed FEM setting in \cite{Jonsson2017}. Another dimension-reduction approach pertains to high-order quadrature for embedded methods with hexahedral background meshes and implicitly defined convex surfaces \cite{Saye2015}. This technique relies on a reduction of integrals to one-dimensional integrations and provides strictly positive weights. The methodology proposed in \cite{Chin2020} (and references therein) provides closed-form formulas for the integration of monomials on convex and nonconvex polyhedra and the extension to curved domains, relying on a reduction of integrals up to vertex evaluations. A downside of dimension-reduction techniques is that they are less general than standard quadrature rules. For example, some techniques can lead to negative quadrature weights in the case of high-order FEM \cite{Longva2020}, which can potentially cause instabilities with inexact arithmetics.

\item \textit{Parameter optimization}: Various strategies have been proposed to optimize the parameters of the cut element integration techniques discussed above, most notably for octree subdivision techniques. In \cite{Abedian2013,Divi2020} algorithms are proposed to select the integration order on the different levels of sub-cells. Ref.~\cite{Longva2020} presents a methodology to reduce the number of integration points in a manner similar to moment fitting techniques. These optimization techniques have demonstrated that reducing the number of integration points does not necessarily compromise the accuracy of the simulations. This is theoretically supported by Strang's first lemma \cite{Strang1973,Strang1972,ErnGuermond}, which indicates that integration does not need to be exact in order to attain (optimal) convergence \cite{Strang1973}\footnote{As formulated by Strang and Fix in 1973 \cite{Strang1973}: \textit{``What degree of accuracy in the integration formula is required for convergence? It is not required that every polynomial which appears be integrated exactly."}}. It should be mentioned that this lemma is also considered in the context of CutFEM in~\emph{e.g.}, \cite{Burman2016,Burman2017,Divi2020}. 	
\end{itemize}
In the selection of an appropriate cut element integration scheme one balances robustness (with respect to cut element configurations), accuracy, and expense. If one requires a method that automatically treats a wide range of cut element configurations and is willing to pay the price in terms of accuracy and computational effort, octree integration is the compelling option. For moderate accuracy, quadratures that rely on exact monomial integrations and only involve vertex evaluations are appealing in terms of accuracy and robustness. If the accuracy and computational-expense requirements are more stringent and the range of configurations is suitably restricted, alternative techniques such as cut element reparameterization are attractive. In the case of implicit boundary representations, an additional consideration in the selection of the cut element integration scheme is whether or not it is required to obtain a parameterization of (or integration scheme on) the unfitted boundary. Parameter optimization procedures can aid in fine-tuning the balance between robustness, accuracy and expense. The appropriateness of the various techniques also depends on the way in which the geometry is represented (implicit \emph{vs.} explicit), as this can have a substantial impact on the implementation of a particular technique.

\subsubsection{Dirichlet boundary condition imposition}\label{sec:dirichlet}
Since the (immersed) boundary of the physical domain does not coincide with the background mesh, the imposition of boundary conditions in immersed methods requires special consideration. Given that a parameterized boundary representation for integration over the boundary exists, Neumann (or natural) boundary conditions can be imposed weakly, in the same way as done in boundary-fitted FEM. The imposition of Dirichlet (or essential) boundary conditions in immersed FEM is not as straightforward, however. Because of the disparity between the background grid and the physical domain, basis functions defined on the background mesh are not generally interpolatory on the unfitted boundary. This precludes strong imposition of Dirichlet conditions as in boundary-fitted FEM. Therefore, boundary conditions in immersed FEM are usually imposed weakly. Different techniques for the imposition of Dirichlet boundary conditions on unfitted boundaries exist, the most prominent of which are:

\begin{itemize}
\item \emph{Penalty method:} The penalty method supplements the weak form of boundary-fitted FEM with a penalty term that penalizes differences between the approximate solution and the prescribed Dirichlet data. This approach, which has been applied in the pioneering work on the finite cell method \cite{Duester2008}, is generally considered as the most straightforward technique to impose Dirichlet conditions on unfitted boundaries. The formulation omits the boundary terms that arise from the partial integration in the derivation of the weak form -- in boundary-fitted FEM these terms drop out as the test functions vanish on Dirichlet boundaries -- such that a modeling error is introduced which yields an inconsistent formulation. For appropriately selected penalty parameters the inconsistency can be acceptable \cite{Babuska1973}, making the penalty method effective for a broad class of immersed problems with complex geometries, see, \emph{e.g.}, \cite{Elhaddad2017,Korshunova2020}. Nevertheless, the choice of an appropriate penalty parameter is challenging. A too small value does not adequately enforce the prescribed boundary conditions, while a too large value exacerbates the conditioning problems~\cite{Ruess2014} and can lead to 
large, nonphysical, gradients on cut elements~\cite{Miegroet2007,Jiang2015,Prenter2018}. 
 
\item \emph{Nitsche's method:} Nitsche's method \cite{Nitsche1971} can be considered as the consistent equivalent of the penalty method, as it retains the boundary gradients (sometimes referred to as the flux terms) in the weak formulation. Through appropriate scaling of the Nitsche (or penalty) parameter, a stable formulation is obtained, see, \emph{e.g.}, \cite{FernandezMendez2004}. Nitsche's method is a widely used technique for the weak imposition of boundary conditions in immersed finite element methods. An elegant aspect of Nitsche's method is that the parameters can be computed per element \cite{Embar2010}, avoiding potential difficulties in the selection of a single global Nitsche parameter, see, \emph{e.g.}, \cite{Ruess2013}. The value of the Nitsche parameter should be inversely proportional to the thickness of the element\footnote{This holds for second-order elliptic problems like the Poisson problem \eqref{eq:poissonstrong}. With higher-order equations some parameters require higher scaling rates. \cite{Embar2010}}, and can become arbitrarily large for small cut elements \cite{Prenter2018}. This problem can be remedied by means of additional stabilization terms such as the ghost penalty, or by the aggregation of basis functions; see Section~\ref{sec:stabilization}. Also, a nonsymmetric Nitsche method can be applied to avoid the need for stabilization \cite{Burman2012,Boiveau2016,Schillinger2016}, although this does affect the linear solver. Additionally, nonsymmetric Nitsche methods are not adjoint consistent, which means that these formulations result in suboptimal approximation properties in the~$L^2(\Omega)$-norm~\cite{Burman2012}.

\item \emph{Lagrange multiplier techniques:} Dirichlet conditions on immersed boundaries can be enforced by supplementing the weak formulation with additional constraint terms, \emph{e.g.}, \cite{Babuska1973lagrange,FernandezMendez2004}. In contrast to the penalty method and Nitsche's method, in Lagrange multiplier techniques these constraint terms are associated with an auxiliary field variable which is defined over the unfitted boundary. This auxiliary field is referred to as the Lagrange multiplier field. Lagrange multiplier techniques result in a saddle point problem, which implies that the discrete Lagrange multiplier field needs to be selected in such a way that a stable system is obtained \cite{Bertsekas2982}. Examples of Lagrange multiplier type techniques for immersed FEM are presented in \cite{Burman2010Lagrange,Baiges2012,Tur2014,Kollmannsberger2015,Kadapa2016}. While an advantage of Lagrange multiplier techniques over Nitsche's method and the penalty method is that these do not require the selection of a parameter, the downsides are that additional degrees of freedom are introduced through the Lagrange multiplier field, and that a (inf-sup) stable discretization of the Lagrange multiplier field is generally non-trivial. Additionally, for many problems, the introduction of Lagrange multipliers changes the nature of the linear system from positive (semi-)definite to indefinite and breaks the diagonal dominance. This affects the applicability of iterative solvers (in particular this precludes the conjugate gradient method) and the factorization in sparse direct solvers.

\item \emph{Basis function redefinition:} An alternative class of techniques to impose Dirichlet conditions on immersed boundaries is based on the idea to redefine the basis functions in such a way that the modified (non-vanishing) basis functions are interpolatory on the unfitted boundary. This enables traditional strong imposition of boundary conditions, as is standard in boundary-fitted finite element methods. Prominent examples of methods based on this concept are WEB-splines \cite{Hoellig2001,Hoellig2005} and, more recently, i-splines \cite{Sanches2011}. The main advantage of these techniques is that they do not require modifications to the weak formulation in comparison to boundary-fitted FEM, but the algorithms to redefine the basis functions constitute an additional non-trivial component in the implementation and analysis.

\end{itemize}
In the selection of an appropriate technique for imposing Dirichlet conditions in immersed finite elements, consistency and accuracy requirements are a prominent consideration. If there are no stringent accuracy requirements, the penalty method is an attractive option on account of its simplicity. If this method does not meet the accuracy requirements, one can resort to consistent weak formulations, where particularly Nitsche's method strikes a suitable balance between accuracy and ease of implementation. Basis function redefinition strategies are an attractive alternative when there are reasons to enforce Dirichlet conditions in a strong manner, like in boundary-fitted finite element formulations.

\subsubsection{Stability and conditioning}\label{sec:introStabilityConditioning}

Immersed discretizations that make use of finite element spaces defined on a background mesh generally suffer from the so-called \emph{small-cut-element problem}. Conventional boundary-fitted finite element methods impose conditions on the shape and the size of the elements in 
the considered (family of) meshes. Such conditions can be (more-or-less) directly managed by the mesh-generation algorithm. 
In immersed finite element methods, on the other hand, one has no control over the shape and size of cut elements. Consequently, cut elements can have arbitrarily small intersections with the physical domain, which is commonly referred to as the small-cut-element problem. This loss of control in the immersed or unfitted setting can have extensive implications on the well-posedness and conditioning of the resulting discrete problem, unless the immersed method is judiciously formulated.

The challenge of stability and conditioning in immersed finite elements with respect to small cut elements is the main focus of this review. Frequently, this small-cut-element problem is considered as a single-faceted problem. In our opinion, however, there are two distinct (albeit strongly related) facets to the small-cut-element problem, \emph{viz.}\ stability and conditioning:

\begin{itemize}
\item \emph{Stability} is related to the fact that the numerical formulation itself can be ill-posed in the immersed setting. The most clear example of this is the imposition of Nitsche's method with standard unfitted finite element spaces. The Nitsche parameter required for coercivity tends to infinity under mesh refinement, which can lead to unbounded gradients on immersed or unfitted boundaries. It should be mentioned that the stability of immersed finite elements is closely related to the way in which essential boundary conditions are enforced, and, in the case of Nitsche's method, to the value of the Nitsche parameter.

\item \emph{Conditioning} is related to the linear algebraic problem of obtaining the solution of the discrete system that arises from an immersed finite element formulation. Even if the problem is properly defined from a stability perspective, \emph{e.g.}, with only Neumann conditions on unfitted boundaries, the resulting system can be arbitrarily ill-conditioned, which impedes the solution of the system. This is caused by functions that are only supported on small cut elements, for which the operator norm (that is equal to the matrix energy norm of the corresponding coefficient vector) is affected by the cut-element size, while the norm of the coefficient vector itself is not. This implies that the eigenvalues of the system matrix can be arbitrarily close to zero, depending on the cut-element configuration.

\end{itemize}
A myriad of techniques has been developed to counteract the problems associated with small cut elements, encompassing treatments for both stability and conditioning issues. The effects of cut elements on the conditioning of the linear system are discussed in detail in Section~\ref{sec:conditioningAnalysis}. Section~\ref{sec:preconditioners} provides an overview of tailored preconditioners that address this issue. Specific attention in this section is devoted to Schwarz preconditioners, that form a natural resolution to the conditioning problem. The stability of immersed discretizations is treated in detail in Section~\ref{sec:stabilization}, which considers both the stability problem itself in Section~\ref{sec:introSec4} and the methodologies that have been developed to resolve it in Section~\ref{sec:stabilityLiterature}. Two particular approaches, \emph{viz.}\ element aggregation and the ghost-penalty method, resolve the stability problems in a manner that yields enhanced coercivity in $H^1(\Omega_h)$ (\emph{i.e.}, on the union of all active elements) instead of just in $H^1(\Omega)$. Consequently, these approaches do not only guarantee stability, but simultaneously preclude conditioning problems. Element aggregation and the ghost-penalty method are discussed in detail in Sections~\ref{sec:aggregation} and \ref{sec:ghost}, respectively, and Section~\ref{sec:stabilityAnalysis} presents a unified mathematical approach to analyze these techniques.

\begin{remark}\textbf{Interpretation from the perspective of norm equivalences.} In terms of the problem definitions presented above, stability and conditioning can be distinguished as different norm equivalences in the discrete space. Stability pertains to the strength of the norm equivalence between the $H^1(\Omega)$-norm, which is a common measure for establishing the quality of a solution, and the operator norm, $\| \cdot \|_{a_h}$ (or the equivalent $\beta$-norm or energy norm, which will be defined in Section~\ref{sec:stabilization}). Conditioning, on the other hand, pertains to the strength of the norm equivalence of the matrix energy norm, $\| \cdot \|_{\mathbf{A}}$ (which is equal to the operator norm of the corresponding discrete function), and the $\ell^2$ vector norm, $\| \cdot \|_2$. In an unfitted discretization that contains small cut elements, the equivalences between these norms can become very weak, which indicates the potential of difficulties with respect to stability and/or conditioning. To ensure coercivity of the weak formulation when Nitsche's method is employed in an immersed setting without a dedicated stabilization technique, a (locally) very large Nitsche parameter is required. Such a very large Nitsche parameter, however, also causes the operator norm to be only very weakly bounded by the $H^1(\Omega)$-norm \cite{Prenter2018}. Similarly, functions that are only supported on small cut elements have a small operator norm, while the size of the cut elements does not affect the the $\ell^2$-norm of the coefficient vector. For this reason, the bound on the $\ell^2$-norm of the coefficient vector in terms of the matrix energy norm can be arbitrarily weak~\cite{Prenter2017}.\end{remark}

\begin{remark}\textbf{Stable explicit time integration.} While not considered in detail in this review, a related challenge is the stability of explicit time integrators on unfitted grids containing small cut elements. Analogous to the system matrix, the eigenvalues of the (consistent) mass matrix cannot be bounded from below in immersed formulations, such that the stable time step can become arbitrarily small. Similar to the stability of the solution on small cut elements discussed above, this can generally be resolved by the function space manipulations discussed in Section~\ref{sec:aggregation} or by the addition of the weak stabilization terms as discussed in Section~\ref{sec:ghost} to the mass matrix, see, \emph{e.g.}, \cite{burman2020explicit}. Regarding the stability of explicit time integrators, also the investigations into the spectral behavior of Nitsche's method in both boundary-fitted and immersed settings presented in \cite{Harari2018,Harari2019,Albocher2021} are of particular interest. It should be mentioned that systems with lumped mass matrices and smooth (isogeometric) discretizations form an exception to the dependence of the stable time-step size on cut elements. For such systems, the eigenvalues of the lumped mass matrix on small cut elements scale more favorably with the size of small cut elements than the eigenvalues of the stiffness matrix. Therefore, stable explicit time integration can be performed without additional stabilization and with time steps dependent on the background element size, see \cite{Leidinger2019,Leidinger2020} for details.\end{remark}

\section{Ill-conditioning and preconditioning}
\label{sec:conditioningFull}

An integral aspect of finite element methods is to find the solution of the corresponding linear system of equations. For small systems this is generally done by a direct solver, which factorizes the linear system and directly computes the solution up to machine precision. The computational cost of direct solvers scales poorly with the size of the system, however, resulting in computation time and memory requirements that become prohibitive for large systems \cite{GolubVanLoan}. For this reason, large systems are generally solved by iterative solvers, the computational cost of which generally scales better with the size of the system \cite{Saad}. The convergence of iterative solvers is strongly correlated with the conditioning of the system. Without tailored stabilization or preconditioning, systems derived from immersed finite element formulations are generally severely ill-conditioned, impeding the application of iterative solvers \cite{Prenter2017}. In Section \ref{sec:conditioningAnalysis}, an analysis of the causes of ill-conditioning in immersed finite element systems is presented. Preconditioners which effectively resolve these conditioning problems are discussed in Section~\ref{sec:preconditioners}. Section~\ref{sec:Schwarz} provides a detailed discussion about Schwarz preconditioners, which have commonly been applied to immersed finite element systems in recent years. This last subsection also includes an example that illustrates the effect of Schwarz preconditioning on an immersed finite element simulation of a scan-based geometry.

\subsection{Conditioning analysis}\label{sec:conditioningAnalysis}

\subsubsection*{Condition number}

The conditioning of a linear system is often used as an indication of the complexity of solving that linear system by means of an iterative solver. An important property of a system matrix to indicate the conditioning is the (Euclidean) condition number, $\kappa\left(\mathbf{A}\right)$, which is defined as the product of the norm of the system matrix and the norm of its inverse
\begin{equation}\label{eq:conditionnumber}
	\kappa \left( \mathbf{A} \right) = \| \mathbf{A} \|_2 \| \mathbf{A}^{-1} \|_2 \geq 1
\end{equation}
While the condition number is formally defined as the sensitivity of the solution to perturbations in the right hand side, it also relates to the convergence of iterative solvers. It is noted that the convergence of these solvers is dependent on more factors than simply the condition number, such as, \emph{e.g.}, the grouping of eigenvalues and the orthogonality of eigenvectors \cite{Saad,Greenbaum}.

Based on the equality between the matrix energy norm $\| \mathbf{y} \|_{\mathbf{A}}$ of a vector $\mathbf{y}$ in the vector space $\ell^2(N)$ and the operator norm $\| v_h \|_{a_h}$ of the corresponding function $v_h = \sum_i y_i \phi_i$ in the isomorphic function space $V_h$, the norm of the (symmetric) system matrix $\mathbf{A}$ can be written as
\begin{equation}
\label{eq:local_id_HvB}
	\| \mathbf{A} \|_2 = \max_{\mathbf{y}} \frac{\| \mathbf{A}\mathbf{y} \|_2}{\| \mathbf{y} \|_2} = \max_{\mathbf{z},\mathbf{y}} \frac{a_h(w_h,v_h)}{\| \mathbf{z} \|_2\| \mathbf{y} \|_2} = \max_{\mathbf{y}} \frac{a_h(v_h,v_h)}{\|\mathbf{y}\|_2^2} = \max_{\mathbf{y}} \frac{\| v_h \|_{a_h}^2}{\| \mathbf{y} \|_2^2}
\end{equation}
with $w_h = \sum_i z_i \phi_i \in V_h$ and $\mathbf{z} \in \ell^2(N)$. Similarly, for the norm of the inverse of $\mathbf{A}$ it can be written that
\begin{equation}\label{eq:normAinv}
	\| \mathbf{A}^{-1} \|_2 = \max_{\mathbf{y}} \frac{\| \mathbf{y} \|_2^2}{\| v_h \|_{a_h}^2}
\end{equation}
The identities in~(\ref{eq:local_id_HvB}) and (\ref{eq:normAinv}) convey that the condition number of a linear system emanating from an (immersed) finite element formulation is determined by the tightness of the equivalence between, on one side, the vector norm $\|\mathbf{y}\|_2$ of the coefficient vector, and, on the other side, the operator norm $\| v_h \|_{a_h}$ of the corresponding function, on the isomorphic spaces $\ell^2(N)$ and $V_h$. The weakness of this norm equivalence for systems that contain small cut elements is the root cause of ill-conditioning in immersed finite element methods.

\begin{figure}
\centering
    \begin{tikzpicture}
        \node[anchor=south west] at (0mm,0mm) {\includegraphics[height=70mm]{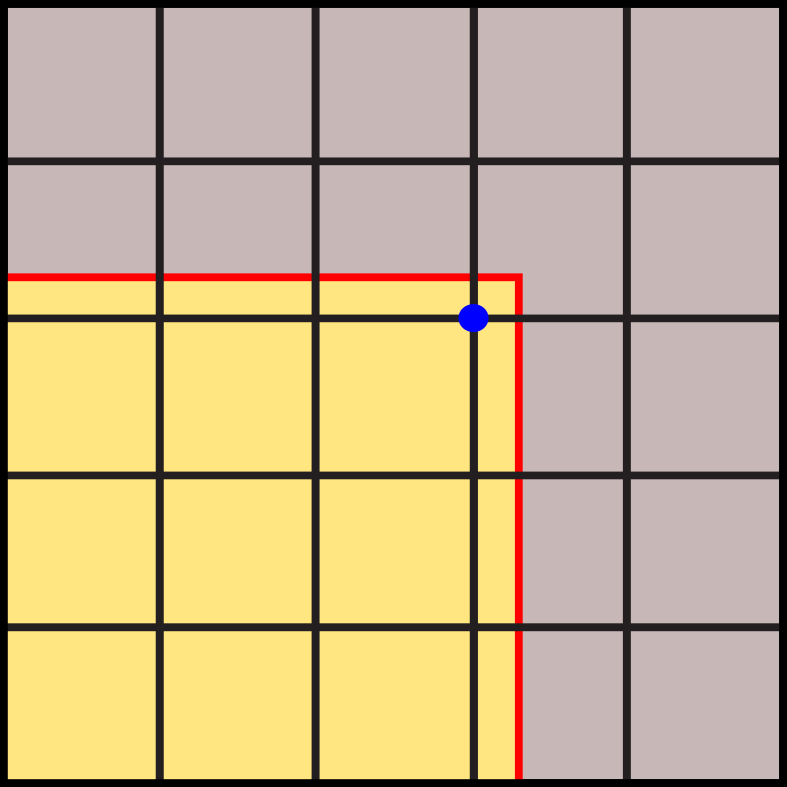}};
        \node at (50mm,50mm) {$T_i$};
        \node at (52mm,22mm) {$\partial\Omega_N$};
        \node at (22mm,22mm) {$\Omega$};
        \draw[|-|, thin] (73mm,43mm) -- (73mm,57mm) node [black,midway,yshift=0mm,xshift=3mm,rotate=0] {$h$};
        \draw[|-|, thin] (43.5mm,73mm) -- (47mm,73mm) node [black,midway,yshift=3mm,xshift=0mm,rotate=0] {$\sqrt{\eta_i}h$};
    \end{tikzpicture}      
\caption{Exemplary cut element geometry to illustrate the derivation of the condition number estimate. The vertex with $\boldsymbol{\xi} = (0,0)$ is indicated in blue.}
\label{fig:scalingRelationGeom}
\end{figure}
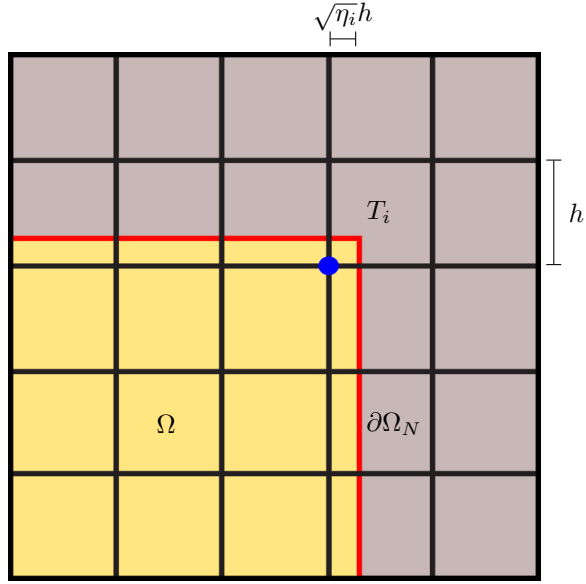

The norms defined above can be employed to provide an estimate, or lower bound, for the condition number of a system with a small cut element. To derive this estimate, we consider a two-dimensional cut-element scenario as indicated in Figure~\ref{fig:scalingRelationGeom}. First, we define the volume fraction $\eta_i$ of the cut element $T_i\in\mcThcut$ as the ratio between the volume of the intersection of the element with the domain $\Omega$ and the volume of the full element in the background grid
\begin{equation}\label{eq:ratio}
	\eta_i = \frac{| T_i \cap \Omega |}{| T_i |} = \frac{| T_{i,\Omega} |}{| T_i |}
\end{equation}
Second, we note that in a polynomial basis of order $p$ it is generally possible to construct a function $v_h \in V_h$
\begin{equation}
	v_h = \begin{cases} \frac{\xi_1^p \xi_2^p}{h^{2p}} & x \in T_i \\ 0 & x \notin T_i \end{cases}
\label{eq:smallFunc}
\end{equation}
with $\boldsymbol{\xi}=(\xi_1,\xi_2)$ a local coordinate that has its origin $\boldsymbol{\xi} = \left( 0,0 \right)$ at the (not-cut) vertex of $T_i$, as indicated in Figure~\ref{fig:scalingRelationGeom},
and where the corresponding coefficient vector $\mathbf{y}$ in the vector space $\ell^2(N)$ has a norm that is approximately $\| \mathbf{y} \|_2 \approx 1$ depending on the employed basis (\emph{e.g.}, B-splines, Lagrange, or integrated Legendre). Under the assumption that element $T_i$ is cut by a natural boundary such that the weak form does not contain boundary terms, the operator norm of $v_h$ is given by
\begin{equation}\begin{aligned}
	\| v_h \|_{a_h}^2 &  = \frac{1}{h^{4p}} \int_{\xi_1=0}^{h \eta_i^{\frac{1}{2}}}\int_{\xi_2=0}^{h \eta_i^{\frac{1}{2}}} \left(\left( p \xi_1^{p-1}\xi_2^{p}\right)^2  + \left(p \xi_1^{p}\xi_2^{p-1}\right)^2\right) \text{d} \xi_1 \text{d} \xi_2 \\
	& = \frac{2p^2}{\left(2p-1\right)\left(2p+1\right)} \eta_i^{2p}
\end{aligned}\end{equation}
This shows that while the $\ell^2$ vector norm $\| \mathbf{y} \|_2$ is not affected by the volume fraction $\eta_i$, the operator norm~$\| v_h \|_{a_h}$ and, by equivalence, the matrix energy norm scales with the volume fraction $\eta_i$. Therefore it follows that
\begin{equation}
	\|\mathbf{A}^{-1}\|_2 = \max_{\mathbf{y}} \frac{\|\mathbf{y}\|_2^2}{\| \mathbf{y} \|_{\mathbf{A}}^2} = \max_{\mathbf{y}} \frac{\|\mathbf{y}\|_2^2}{\| v_h \|_{a_h}^2} \gtrsim \frac{1}{\eta_i^{2p}} 
\end{equation}
with $\gtrsim$ denoting an inequality involving a constant (\emph{i.e.}, $a \gtrsim b$ indicates $a > C b$ for some constant $C$). As the norm $\|\mathbf{A}\|_2$ does not depend on the volume fraction $\eta_i$, this results in a condition number $\kappa(\mathbf{A}) \gtrsim \eta_i^{-2p}$. In \cite{Prenter2017} this derivation is performed in an abstract setting for more general cut scenarios (under certain shape assumptions), for different numbers of dimensions, and with different boundary conditions. This results in the estimate of the condition number of linear systems derived from immersed finite element formulations of second order problems
\begin{equation}
	\kappa(\mathbf{A}) \gtrsim \eta^{-(2p+1-2/d)}
	\label{eq:scaling}
\end{equation}
with $\eta$ denoting the smallest volume fraction in the system, \emph{i.e.}, $\eta = \min_{i} \eta_i$. The scaling relation \eqref{eq:scaling} is numerically verified in \cite{Prenter2017} for different grid sizes, discretization orders, and for both $C^0$-continuous and maximum continuity spline bases.

\clearpage

It can be noticed that the shape of the cut element is not included in the condition number estimate. That is because within certain shape-regularity assumptions (see~\cite{Prenter2017} for details), the shape of the cut element is not a dominant factor in the order of magnitude of the smallest eigenvalue of the system matrix. When the system matrix is (diagonally) scaled, as will be discussed below, the shape of the cut element does play a role. It should be mentioned that, under the same shape-regularity assumptions, the derivation in \cite{Prenter2017} also establishes that when a local, element-wise, Nitsche parameter is employed, the condition number estimate is not affected by the type of boundary condition imposed on the (cut) boundary of the smallest cut element. When a global Nitsche parameter is employed, the value of this parameter is determined by the smallest cut element. While this does not affect the smallest eigenvalue of the system, this makes the largest eigenvalue dependent on the smallest cut element as well, resulting in even larger condition numbers.

\subsubsection*{Effect of diagonal scaling, smoothness and the cut-element shape}

The previous paragraph considered the condition number of the system matrix $\mathbf{A}$ as is, without any treatment. It is noted, however, that linear algebra solvers typically perform basic rescaling procedures, most notably diagonal scaling operations such as Jacobi, Gauss-Seidel or SSOR preconditioning \cite{Barrett}. Consequently, it is important to not only consider the condition number independently, but also to have a closer look at the coefficient vector $\mathbf{y}$ corresponding to a function with a very small operator norm $\| v_h \|_{a_h}$. As indicated in \cite{Prenter2017}, there exist two mechanisms by which the operator norm of a basis function can be much smaller than the norm of the corresponding coefficient vector. First, it is possible that a basis function in itself is very small. In this case, simply the unit vector corresponding to that specific basis function will already cause a large condition number through \eqref{eq:normAinv} (\emph{i.e.}, with this unit vector taken as $\mathbf{y}$, the quotient in this equation will be very large). Second, on small cut elements it is possible that the dependence of certain basis functions on a specific parametric coordinate or on higher-order terms diminishes. This essentially reduces the dimension of the space spanned by the basis functions on the small cut element, such that these basis functions become almost linearly dependent (see Figure~\ref{fig:good_bad_cut}). In this case, the small function $v_h$ as in \eqref{eq:smallFunc} cannot be represented by a unit vector, and the coefficient vector $\mathbf{y}$ corresponding to the small function requires multiple nonzero entries. This is an important nuance in relation to the scaling of the system, as small eigenvalues that correspond to an (almost) unit vector are resolved by diagonal scaling, while small eigenvalues caused by almost linear dependencies are not.

Because of the above-discussed distinction between the classes of small eigenvalues in immersed finite element systems, there is a significant difference in conditioning between, on one hand, higher-order ($p \geq 2$) $C^0$-polynomials, and, on the other hand, maximal continuity splines and linears (note that linears are a subclass of maximal continuity splines). This is because in the former case, almost linear dependencies will be formed on all small cut elements. For such systems, Jacobi preconditioning lowers the condition number, but the resulting systems are still ill-conditioned and the condition number still shows a dependence on the smallest cut element. For maximum continuity splines, the shape of small cut elements starts to play a role. On elements in which a vertex is contained in $T_{\Omega}$, as indicated by element $T_{1,\Omega}$ in Figure~\ref{fig:good_bad_cut}, a discretization with maximum continuity splines will only contain a single basis function of which the support is restricted to the small cut element. Therefore, a Jacobi preconditioner suffices to repair the conditioning with regard to that element. On cut elements where $T_{\Omega}$ does not contain a vertex, as indicated by element $T_{2,\Omega}$ in Figure~\ref{fig:good_bad_cut}, the diminished dependence on the (in this case horizontal) parametric coordinate causes almost linear dependencies. Consequently, such elements still cause ill-conditioning even after diagonal rescaling of immersed systems with maximum continuity splines. Because of this dependence on the continuity of the basis, a scaled linear system derived from an unfitted discretization with maximum continuity splines will contain far fewer problematically small eigenvalues than a scaled linear system derived from an unfitted discretization with $C^0$-continuous basis functions of the same degree $p \geq 2$. 
Depending on the size of the system and the smoothness of the boundary, the number of problematic eigenvalues in an immersed isogeometric discretization in conjunction with diagonal scaling can even be small enough to render direct application of a Krylov-subspace based iterative solver feasible. Because only cut elements of a specific shape cause small eigenvalues for systems based on maximum continuity splines in conjunction with diagonal scaling, the overall smallest volume fraction and the scaled condition number are in general only weakly correlated in such systems~\cite{Prenter2017}.

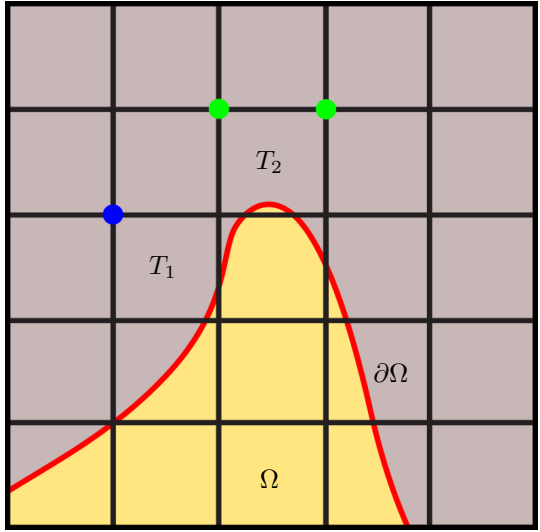
\begin{figure}
     \centering
    \begin{tikzpicture}
        \node[anchor=south west] at (0mm,0mm) {\includegraphics[height=70mm]{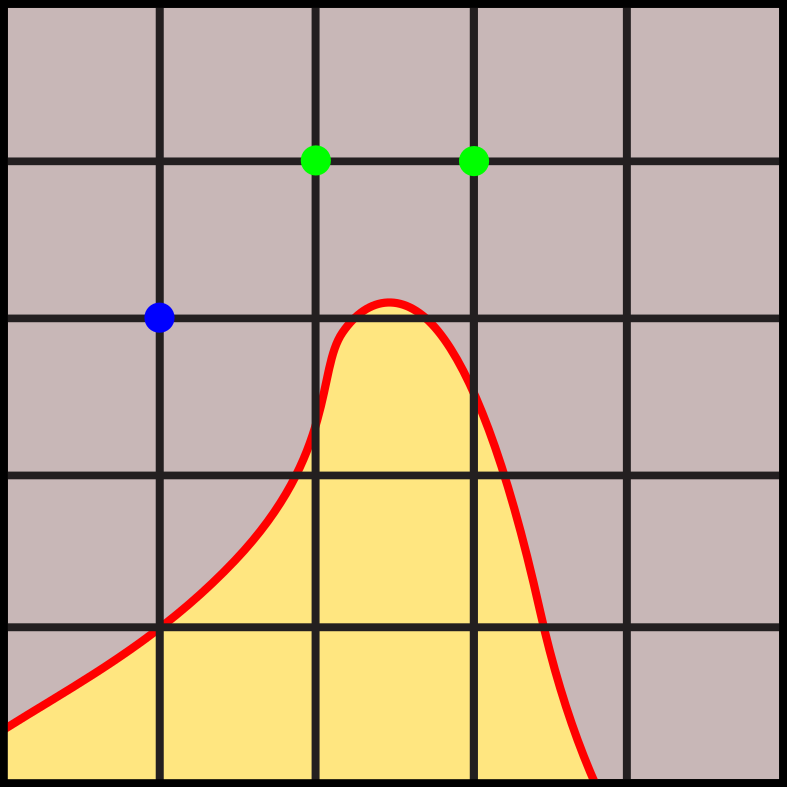}};
        \node at (22mm,36mm) {$T_1$};
        \node at (36mm,50mm) {$T_2$};
        \node at (52mm,22mm) {$\partial\Omega$};
        \node at (36mm,8mm) {$\Omega$};
    \end{tikzpicture}      
\caption{Illustration of an unfitted boundary with different types of cut elements. Element $T_{1,\Omega}$ contains a vertex. With maximal continuity splines, only one basis function is supported on this element only (for a linear basis, the node corresponding to this function is indicated by the blue dot). Rescaling the basis functions will resolve the conditioning problems for this type of cut element. Element $T_{2,\Omega}$ does not contain a vertex, and multiple basis functions are only supported on this element (with a linear basis, these are the basis functions corresponding to the nodes indicated with the green dots). If the volume fraction of this element is very small, the dependence of the basis functions on the horizontal coordinate (relative to the dependence on the vertical coordinate) diminishes, such that the basis functions describe essentially the same degree of freedom and become linearly dependent. Therefore, rescaling the basis functions will not resolve the conditioning problems on this type of cut element. In a similar manner, the relative contribution of higher-order terms diminishes relative to the contribution of linear terms on small cut elements. As a result, with higher-order discretizations that are not of maximal continuity, almost linear dependencies generally occur on small cut elements of any shape, while these only occur on specifically cut elements with discretizations of maximal continuity.}
\label{fig:good_bad_cut}
\end{figure}

\subsection{Preconditioning and literature overview}\label{sec:preconditioners}
Based on the conditioning analysis presented above, ill-conditioning of immersed FEM systems can effectively be negated by dedicated preconditioning techniques. The general idea of preconditioning is to construct a preconditioning matrix, $\mathbf{B}$, which is an approximation of the inverse of the system matrix $\mathbf{A}$, and then to solve the preconditioned system
\begin{align}
  \mathbf{B} \mathbf{A} \mathbf{x} = \mathbf{B} \mathbf{b}
\end{align}
Although the original system matrix $\mathbf{A}$ in equation~\eqref{eq:linearsystem} can be ill-conditioned, \emph{i.e.}, $\kappa(\mathbf{A}) \gg 1$, a properly formed preconditioner results in a well-conditioned preconditioned system matrix, \emph{i.e.}, $\kappa( \mathbf{B} \mathbf{A} )\approx \kappa( \mathbf{A}^{-1} \mathbf{A} ) = \kappa( \mathbf{I} ) = 1$. In constructing the preconditioner $\mathbf{B}$, one balances the computational effort required to compute and apply the preconditioner with the extent to which the inverse of the original system matrix is approximated. 

Various dedicated preconditioning techniques to resolve ill-conditioning problems caused by small cut elements have been developed in the context of GFEM and XFEM, such as: a preconditioner based on local Cholesky decompositions \cite{Bechet2005}, a FETI-type preconditioner tailored to XFEM \cite{Menk2011}, an algebraic multigrid preconditioner that is based on the Schur complement of the enriched basis functions \cite{Hiriyur2012}, and domain decomposition preconditioners based on additive Schwarz \cite{Berger-Vergiat2012,Waisman2013}. 

In recent years, dedicated preconditioners have been developed for immersed FEM. It is demonstrated in \cite{Lang2014} that systems with linear bases can effectively be treated by a diagonal preconditioner in combination with the removal of very small basis functions. Under certain restrictions on the cut-element geometry, it is derived in \cite{Lehrenfeld2017} that a scalable preconditioner for linear bases is obtained by combining a Jacobi preconditioner for basis functions on cut elements with a standard multigrid preconditioner for interior basis functions that do not intersect the boundary. In \cite{Prenter2017}, a preconditioner is developed that combines diagonal scaling with local Cholesky factorizations. This technique is motivated by the analysis of the conditioning problems of immersed methods in the previous section, and can be interpreted as a local change of basis on small cut elements, where the Cholesky factorizations correspond to local orthonormalization procedures for almost linearly dependent basis functions, which are identified by a tailored algorithm. The resulting preconditioner effectively resolves ill-conditioning for immersed methods discretized with higher-order (continuous) bases. 

Recently, Schwarz preconditioners -- which were already discussed above in the context of GFEM and XFEM~\cite{Berger-Vergiat2012,Waisman2013} -- have also gained momentum in immersed finite elements. The concept of Schwarz preconditioning to overcome the problem of linear dependencies on small cut elements was considered in~\cite{Prenter2019,Ludescher2020}. This concept was generalized to multilevel $hp$-finite element bases in \cite{Jomo2019}, where it was also demonstrated to be effective in parallel computing frameworks. 
In~\cite{Prenter2020}, Schwarz preconditioning of unfitted systems was applied as a smoother in a multigrid solver, making it suitable for large scale computations. The methodology was applied in high-performance parallel-computing settings in \cite{Saberi2020,Jomo2020}, with~\cite{Jomo2020} demonstrating excellent scalability for problems with multi-billion degrees of freedom 
distributed over close to~$10^5$ cores. A Balancing Domain Decomposition by Constraints (BDDC) scalable preconditioner, tailored to immersed FEM by choosing appropriate weighting coefficients for cut basis functions, has been proposed in~\cite{Badia2017}. BDDC methods are multilevel additive domain decomposition algorithms that can scale up to millions of cores/subdomains~\cite{Badia_2016-3}. The preconditioner in~\cite{Badia2017} results in an effective preconditioner for linear basis functions and exhibits the same parallelism potential as standard BDDC, thus being well-suited for large-scale systems on distributed-memory machines.

It is worth mentioning that dedicated preconditioners have also been developed and investigated for ghost-penalty stabilized discretizations. While, as will be demonstrated in Section~\ref{sec:stabilization}, such stabilized systems do not suffer from the typical conditioning problems related to small cut elements discussed in Section~\ref{sec:conditioningAnalysis}, the different setting of the problem and the involvement of additional terms does warrant the investigation into the applicability of efficient solvers, in particular multigrid techniques. Dedicated multigrid routines for Nitsche-based ghost-penalty stabilized methods are presented in \cite{Kothari2021,Nuessing2018,Kothari2020}. It is notable that \cite{Kothari2020} additionally presents multigrid routines for Lagrange multiplier based methods, see also \cite{KothariEccomas,Kothari2022}. Furthermore, in \cite{Gross2021} and \cite{Gross2022} a similar approach as in \cite{Lehrenfeld2017} (discussed above) is followed. In these references it is demonstrated that a stable splitting of the degrees of freedom exists, and that a scalable solver is obtained by combining a standard multigrid technique for a well-defined set of internal degrees of freedom with a diagonal preconditioner for the set of degrees of freedom along the boundary. While this technique is only applicable to ghost-penalty stabilized systems, in contrast to \cite{Lehrenfeld2017} it can also be applied to higher-order discretizations. A final noteworthy contribution is \cite{AyusoDeDios2020}, which presents a preconditioner for ghost-penalty stabilized immersed interface problems of high contrast. In this reference a different splitting of the degrees of freedom is applied to define a Schur complement, based on which preconditioners are presented that are robust to high contrast ratios.

The remainder of this section focuses on Schwarz preconditioning and the considerations regarding its application to systems derived from immersed finite element methods.

\subsection{Schwarz preconditioning}\label{sec:Schwarz}

\vspace{.8mm}

\subsubsection*{Concept of Schwarz preconditioning}

\vspace{.8mm}

The concept of Schwarz preconditioning is to invert (restrictions of) local blocks of the system matrix $\mathbf{A}\in\mathbb{R}^{N \times N}$ and then sum these contributions to form the preconditioner. To provide a definition, we consider a set of (potentially overlapping) index blocks, where each index block contains $M_i \geq 1$ indices. The additive Schwarz preconditioner is then defined as
\vspace{.8mm}
\begin{equation}
 \mathbf{B}_{\textsc{AS}} = \sum_i \mathbf{B}_i = \sum_i \mathbf{P}_i \underbrace{\left(\mathbf{P}_i^T\mathbf{A}\mathbf{P}_i \right)^{-1}}_{\mathbf{A}_i^{-1}}\mathbf{P}_i^T
 \label{eq:ASconstruct}
\end{equation}
\vspace{.3mm}
and the multiplicative Schwarz preconditioner as
\vspace{.8mm}
\begin{equation}
 \mathbf{B}_{\textsc{MS}} = \sum_i \mathbf{B}_i \prod_{j=1}^{i-1} \big(\mathbf{I} - \mathbf{A} \mathbf{B}_j \big)
 \label{eq:MSconstruct}
\end{equation}
The prolongation operator $\mathbf{P}_i \in \mathbb{R}^{N \times M_i}$ consists of the unit vectors corresponding to the $M_i$ indices in the $i$-th index block. Pre- and post-multiplying the system matrix $\mathbf{A}\in\mathbb{R}^{N \times N}$ with the (transpose of) this prolongation operator restricts it to the submatrix $\mathbf{A}_i\in\mathbb{R}^{M_i \times M_i}$ consisting of only the indices in the $i$-th index block. Similarly, the opposite pre- and post-multiplication with these operators injects the local inverse $\mathbf{A}_i^{-1}\in\mathbb{R}^{M_i \times M_i}$ into the matrix $\mathbf{B}_i\in\mathbb{R}^{N \times N}$. It is to be noted that the index blocks may overlap. In additive Schwarz these contributions are then added in $\mathbf{B}_{\textsc{AS}}$ and treated simultaneously. In multiplicative Schwarz, repetitions of the same index are treated sequentially. Application of the multiplicative Schwarz preconditioner can be expressed by recursive relations; see, \emph{e.g.},~\cite{Prenter2020}. It is noted that index blocks consisting of a single index simply reduce to Jacobi preconditioning in additive Schwarz, and to Gauss-Seidel preconditioning in multiplicative Schwarz. The formal definition and details about the construction of Schwarz preconditioners are presented~\cite{Prenter2019} and~\cite{Prenter2020}.

\subsubsection*{Application to immersed systems}

\vspace{.8mm}

Schwarz preconditioners can be conceived of as locally orthonormalizing the basis functions corresponding to the indices in a block. In order to effectively employ the concept of Schwarz preconditioning as a tailored preconditioner for unfitted systems, it is therefore essential that \emph{for every set of almost linearly dependent functions, there is an index block containing all these functions}. Since almost linear dependencies occur between basis functions that are supported on a small cut element, the index blocks are generally chosen by selection procedures based on the overlapping support of basis functions, either for all active elements in $\mcTh$ or only for the cut elements in $\mcThcut$. A discussion regarding considerations in the index blocks is provided later in this section.

A further interpretation of Schwarz preconditioning in relation to immersed finite element methods can be obtained from the additive Schwarz lemma. This lemma states that for a Symmetric Positive Definite (SPD) matrix $\mathbf{A}$, the $\mathbf{B}_{\textsc{AS}}^{-1}$-inner product of an arbitrary vector $\mathbf{y}$ is equal to \cite{Matsokin1985,Lions1988} (see \cite{Smith1996,Toselli2005} for this specific form)
\vspace{.8mm}
\begin{equation}
 \mathbf{y}^T \mathbf{B}_{\textsc{AS}}^{-1} \mathbf{y} = \min_{\mathbf{y} = \sum\limits_j \mathbf{P}_j \tilde{\mathbf{y}}_j} \hspace{.2cm} \sum\limits_i \tilde{\mathbf{y}}_i^T \mathbf{A}_i \tilde{\mathbf{y}}_i = \min_{\mathbf{y} = \sum\limits_j \mathbf{P}_j \tilde{\mathbf{y}}_j} \hspace{.2cm} \sum\limits_i \left( \mathbf{P}_i \tilde{\mathbf{y}}_i \right)^T \mathbf{A} \left( \mathbf{P}_i \tilde{\mathbf{y}}_i \right)
 \label{eq:ASlemma}
\end{equation}
\vspace{.3mm}
In these identities, $\tilde{\mathbf{y}}_i\in\mathbb{R}^{M_i}$ denotes a block vector corresponding to the $i$-th index block. The statement $\mathbf{y} = \sum_j \mathbf{P}_j \tilde{\mathbf{y}}_j$ indicates that the sum of the prolongations of these block vectors form a partition of the vector $\mathbf{y}\in\mathbb{R}^{N}$. A set of block vectors with this property exists, provided that every index is contained in (at least) one index block. In the case that the blocks do not overlap, this set of block vectors is unique. In the case that the index blocks do overlap, multiple sets of block vectors have the partition property. Accordingly, the lemma states that the $\mathbf{B}_{\textsc{AS}}^{-1}$-inner product of $\mathbf{y}$ is equal to the minimum of the sum of the $\mathbf{A}_i$-inner products of the block vectors $\tilde{\mathbf{y}}_i$ over all sets of block vectors with the partition property.

To relate this to the specific conditioning problems in immersed finite element methods, recall that $\mathbf{B}_{\textsc{AS}}$ can be considered as a sparse approximation of $\mathbf{A}^{-1}$. Efficient preconditioning requires $\mathbf{B}_{\textsc{AS}}$ to have similar properties as $\mathbf{A}^{-1}$ or, similarly, $\mathbf{B}_{\textsc{AS}}^{-1}$ to have similar properties as~$\mathbf{A}$. The analysis of the conditioning problems associated with immersed finite element methods in Section~\ref{sec:conditioningAnalysis} conveys that the principal cause of 
these problem is that, potentially, $\| \mathbf{y} \|_2 \gg \| \mathbf{y} \|_{\mathbf{A}}$ in case $\mathbf{y}$ corresponds to a function comprised of \emph{i)} a very small basis function or \emph{ii)} almost linearly dependent basis functions. For the first case, it follows from~\eqref{eq:ASlemma} that~$\mathbf{y}$ must also have a small~$\mathbf{B}_{\textsc{AS}}^{-1}$-inner product, such that this property is captured by the additive Schwarz preconditioner, independent of the index blocks. For the second case, assume that indeed \emph{for every set of almost linearly dependent functions, there is an index block containing all these functions}, in accordance with the previously stated postulate. Then, in case $\mathbf{y}$ corresponds to a function comprised of almost linearly dependent basis functions, $\mathbf{y}$ can be written as the prolongation of a single block vector. Consequently, also in the second case, it follows from \eqref{eq:ASlemma} that $\mathbf{y}$ will have a small $\mathbf{B}_{\textsc{AS}}^{-1}$-inner product, such that this property is captured in the additive Schwarz preconditioner. As a result, with a proper choice of the index blocks, for both causes of very small eigenvalues of the matrix $\mathbf{A}$, the corresponding modes will also be present in $\mathbf{B}_{\textsc{AS}}^{-1}$. The Schwarz preconditioner thus specifically targets the problematic aspects of small eigenvalues due to small cut elements, and thereby effectively resolves the ill-conditioning in systems derived from immersed formulations.

\subsubsection*{Effectivity and multigrid preconditioning}

While the effectivity of Schwarz preconditioning for immersed problems can be explained by the additive Schwarz lemma in combination with a particular selection of the blocks, a formal mathematical bound on the eigenvalues of an unfitted system treated by Schwarz preconditioning has not yet been formulated. Numerical results, however, consistently show that the resulting systems behave the same as boundary-fitted systems, in the sense that (for second order problems as the Poisson equation) the condition number scales as~$h^{-2}$, and the number of iterations that is required to solve the linear system up to a prescribed tolerance is proportional to~$h^{-1}$ \cite{Prenter2019,Prenter2020}.

To resolve the remaining grid dependence after Schwarz preconditioning, the Schwarz preconditioner can be applied as a smoother in a geometric multigrid framework. In~\cite{Prenter2020,Jomo2020,Saberi2020} this is demonstrated to result in a methodology that is robust to cut elements and which solves linear systems with quasi-optimal complexity, \emph{i.e.}, at a computational cost that is linear with the number of degrees of freedom (DOFs). A delicate consideration in a multigrid framework is the choice between additive and multiplicative Schwarz. As demonstrated in \cite{Prenter2020}, the stability of a multigrid solver with an additive Schwarz smoother requires a considerable amount of relaxation, and for this reason \cite{Prenter2020} employs a multiplicative implementation. As discussed later in this section, multiplicative Schwarz is less suited for parallelization, such that parallel implementations employ either additive Schwarz \cite{Jomo2020} or a hybrid variant \cite{Saberi2020}. Another important aspect of multigrid solvers with Schwarz smoothers is the dependence of their effectivity on the discretization order, which is investigated and discussed in \cite{Prenter2020}. 

\subsubsection*{Block selection}

As previously mentioned, for Schwarz preconditioning to be effective, it is essential that every combination of basis functions that can become almost linearly dependent is contained in an index block, which is generally achieved by selecting these blocks based on the overlap in the supports of basis functions. In \cite{Prenter2019}, an index block is devised for every cut element, containing all basis functions supported on it. For basis functions that are not supported on any cut element, a simple diagonal scaling is performed.
A tailored block selection procedure for multilevel $hp$-adaptive discretizations is developed in \cite{Jomo2019}. In this procedure, the set of basis functions supported on a refined (leaf) element is further restricted to only the necessary DOFs. Additionally, the procedure is optimized by only devising blocks for cut elements with a volume fraction that is smaller than a certain threshold $\eta^* \in [0,1]$. For locally refined discretizations based on truncated hierarchical B-splines, a block selection strategy is presented in \cite{Prenter2020}.
In \cite{Prenter2019} it is noted that, for vector-valued problems, degrees of freedom describing the solution in a certain geometrical dimension can generally not form a linear dependency with degrees of freedom describing the solution in another geometrical dimension. Hence, it can be beneficial to generate separate index blocks for each geometrical dimension.

The most important consideration in the selection of index blocks is the size of the blocks. Small blocks are computationally inexpensive, but can miss almost linear dependencies on pathologically cut elements. Large blocks are more robust, but are computationally more expensive. With the Schwarz preconditioner directly employed in an iterative solver, the number of iterations is generally large enough for small blocks to be more efficient. When the Schwarz procedure is employed as a smoother in a multigrid method, the small number of iterations generally renders larger blocks more appropriate. For this reason, \cite{Prenter2020} and \cite{Jomo2020} consider block selection procedures based on multiple elements. In \cite{Prenter2020} an index block is created for every basis function, containing all the basis functions with a support that is encapsulated by the support of the function for which the block is created. Ref.~\cite{Jomo2020} employs index blocks containing all basis functions supported on a cluster of $2^d$ elements, with $d$ the number of dimensions.

Regarding the lack of consensus on the block selection procedures in different contributions, it can be concluded that an optimal choice of index blocks is still an unresolved question.

\subsubsection*{Extension to other problems}

The additive Schwarz lemma in \eqref{eq:ASlemma} only pertains to Symmetric Positive Definite (SPD) systems. Such systems cover a large number of applications, encompassing many problems in structural mechanics, but do not comprise all problems commonly solved by immersed finite element methods. In particular, incompressible-flow problems in which the pressure takes the form of a Lagrange multiplier in the weak formulation are indefinite, and problems involving convection or formulations based on the nonsymmetric Nitsche method are not symmetric. Applications of the Schwarz framework to immersed finite element approximations of such problems are considered in~\cite{Prenter2019}, in which for all considered problems it is observed that the solver is independent of the cut elements and that the number of iterations is approximately inversely proportional to the mesh size.

\subsubsection*{Implementation}

A specific operation in the construction of a Schwarz preconditioner is the computation of stable inverses of the submatrices~$\mathbf{A}_i$. With very small eigenvalues, numerical round-off errors can cause detrimental errors in the inversion of submatrices with eigenvalues that are too close to machine precision. In the worst case, this can lead to negative eigenvalues in the preconditioned system, resulting in failure of the iterative solver. For that reason, the inverses are generally not computed by a simple inversion operation, but via an eigenvalue decomposition. After the decomposition, eigenvalues that are smaller than a prescribed threshold are discarded. Details about this procedure are described in~\cite{Prenter2019,Jomo2019}.

The Schwarz preconditioner is suitable for parallel implementations~\cite{Jomo2019}. A particular facet of the parallelization of a Schwarz preconditioner is that each submatrix $\mathbf{A}_i$ has to be available in a single subprocess for inversion. As this is not generally the case in parallel (boundary-fitted) finite element codes, this calls for special care in the implementation, \emph{e.g.}, by applying ghost elements \cite{Jomo2019}. As previously mentioned, a specific consideration pertaining to parallel multigrid implementations is the choice between additive and multiplicative Schwarz. While multiplicative Schwarz does not require relaxation, it does require extensive communication and synchronization between parallel processes, specifically in distributed memory systems. For that reason the parallel multigrid implementation 
in~\cite{Jomo2020} employs an additive Schwarz smoother, while in \cite{Saberi2020} a hybrid variant is applied with an additive approach for DOFs that are shared between processes and a multiplicative treatment of DOFs that belong to a single process.

\subsubsection*{Preconditioning example}\label{sec:trabecular}

To illustrate the effectivity of Schwarz preconditioning incorporated in a multigrid cycle in providing a robust solution procedure that is not affected by either the cut elements or by the size of the system, we consider the $\mu$CT-scanned trabecular bone specimen displayed in Figure~\ref{fig:TrabecularGeometry}. This geometry was first studied in \cite{Verhoosel2015} and the presented results were previously published in \cite{Prenter2020}. The specimen is compressed with an average uniaxial strain of $1\%$ in the horizontal direction, resulting in the (Frobenius norm) stresses indicated by the colors. Three different meshes are considered, consisting of $32^3$, $64^3$ and $128^3$ elements in the ambient domain. Figure~\ref{fig:TrabecularGeometry} shows the active elements of the coarsest mesh on the right half of the geometry. With quadratic basis functions, these different meshes result in, respectively, $182\cdot10^3$, $1.03\cdot10^6$ and $6.65\cdot10^6$ active degrees of freedom. Figure~\ref{fig:TrabecularResults} shows the convergence behavior of the conjugate gradient solver preconditioned by a multigrid cycle with a tailored multiplicative Schwarz smoother for the different discretizations. It can be observed that the convergence is virtually independent of both the cut elements and the size of the system. For details regarding this example the reader is referred to~\cite{Prenter2020}. 

\begin{figure}
 \centering
 \vspace{-1mm}
 \begin{subfigure}{.49\textwidth}
  \begin{center} 
  \begin{tikzpicture}\node[anchor=south west] at (0mm,0mm) {\includegraphics[height=50mm]{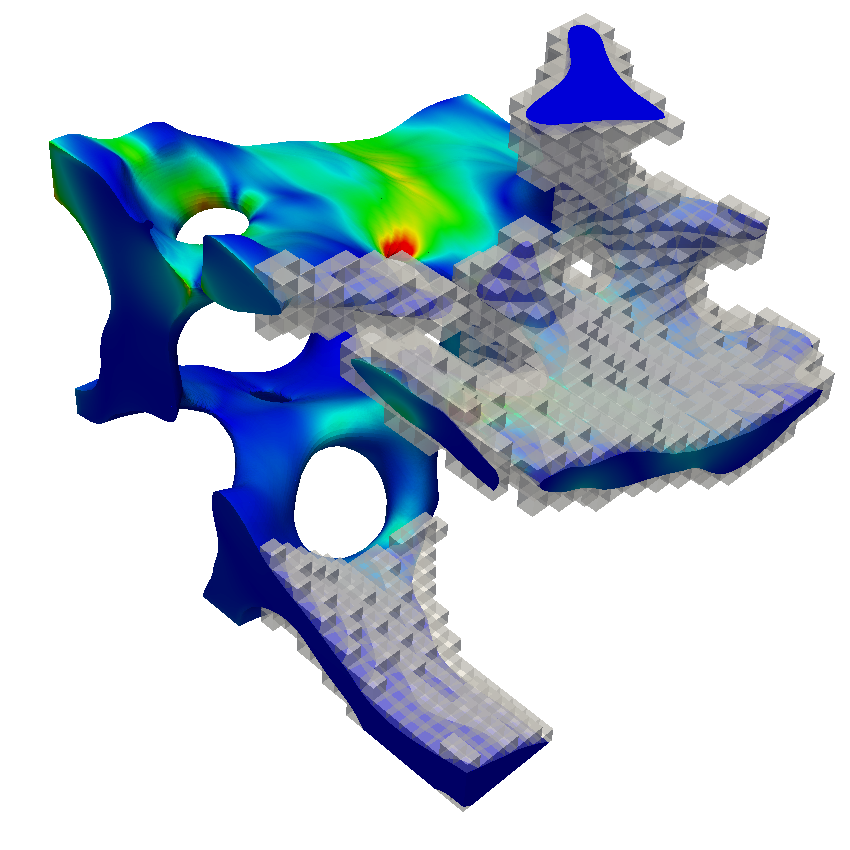}};
  \draw[|-|, thin] (53mm,3mm) -- (56mm,50mm) node [black,midway,yshift=.2mm,xshift=2mm,rotate=86.3] {1.28 mm};
  \end{tikzpicture} \\ \includegraphics[width=5.5cm]{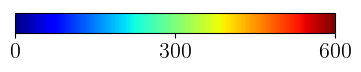}
  \end{center}
  \caption{Frobenius norm of the stress tensor [MPa] and active elements on right part of geometry\label{fig:TrabecularGeometry}}
 \end{subfigure}
 \hfill
 \begin{subfigure}{.49\textwidth}
  \centering
  \vspace{8mm}
  \includegraphics[width=.9\textwidth]{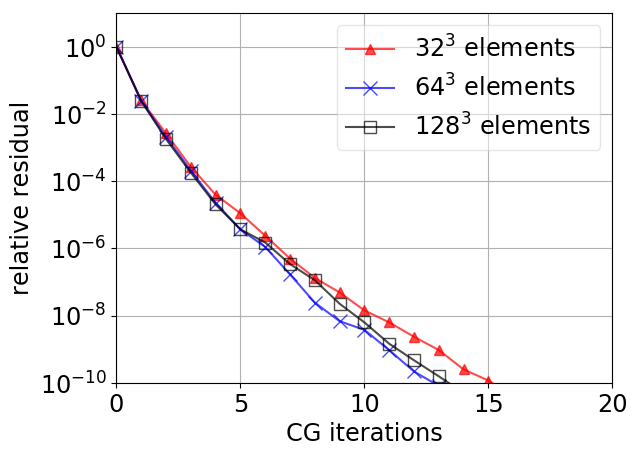} \\ 
  \vspace{9mm}
  \caption{Convergence of the systems with different grid sizes\label{fig:TrabecularResults}}
 \end{subfigure}
 \caption{Illustration of Schwarz preconditioning applied to the immersed elasticity analysis of a trabecular bone specimen. In (a) the geometry, the (right half of the) active elements of the coarsest grid, and the Frobenius norm of the resulting stresses are visible. In (b) the convergence behavior of the preconditioned iterative solver is presented. Data previously published in \cite{Prenter2020}.\label{fig:TrabecularExample}}
\end{figure}

\section{Stability}\label{sec:stabilization}

\subsection{Stability analysis}\label{sec:introSec4}

As indicated in Section~\ref{sec:introStabilityConditioning}, small cut elements do not only adversely affect the conditioning of the resulting linear system, but can also disturb the quality of the solution as the normal gradient is not adequately controlled. When Nitsche's method is applied on a boundary-fitted discretization, the trace inverse inequality 
$\| \partial_n u_h \|_{T \cap \partial \Omega_D} \lesssim  h_T^{-1/2} \| \nabla u_h \|_{T}$ 
holds on every element $T$ adjacent to the Dirichlet boundary $\partial\Omega_D$; see, \emph{e.g.}, \cite{Warburton2003}. By virtue of this trace inequality, the bilinear operator $a_h(\cdot,\cdot)$ defined in \eqref{eq:immersedOperator} is coercive with a global Nitsche parameter~$\beta$ that scales with~$h^{-1}$, and optimal error bounds with respect to the $H^1(\Omega)$-norm hold. With an immersed discretization, however, an element $T$ that intersects $\partial \Omega_D$ only partially intersects the physical domain $\Omega$. Therefore, the aforementioned trace inverse inequality assumes the form 
$\| \partial_n u_h \|_{T \cap \partial \Omega_D} \lesssim h_{T_\Omega}^{-1/2} \| \nabla u_h \|_{T_\Omega}$,
where $h_{T_\Omega}$ indicates a generalized thickness of the fragment $T_\Omega = T \cap \Omega$ (the intersection of element $T$ and the physical domain) normal to $T \cap \partial\Omega_D$ (the intersection of element $T$ and the Dirichlet boundary); see Section~\ref{sec:method_formulation}. Consequently, coercivity of the bilinear form $a_h(\cdot,\cdot)$ requires an element-wise Nitsche parameter $\beta$ that is 
not inversely proportional to the characteristic mesh width of the background element, $h_T$, but instead is inversely proportional to the generalized thickness of the intersection, \emph{i.e.}, $\beta \sim h_{T_\Omega}^{-1} \gtrsim h_T^{-1}$. Due to the fact that without special precautions, in immersed finite element methods one has no control over~$h_{T_\Omega}$, the Nitsche parameter can in principle become arbitrarily large. 

In \cite{Prenter2018} it is shown that the error of immersed Nitsche-based formulations admits a natural analysis in the norm (within this manuscript referred to as the $\beta$-norm)
\begin{equation}
\label{eq:immersedenergynorm}
  \tn u \tn_\beta^2 = \| \nabla u \|_{\Omega}^2 + \| \beta^{\frac{1}{2}} u \|^2_{\partial \Omega_D} + \| \beta^{-\frac{1}{2}} \partial_n u \|^2_{\partial \Omega_D} \qquad u \in V_h \oplus H^2(\Omega)
\end{equation}
provided that $\beta$ is large enough for the bilinear form to be coercive. Note that this norm is referred to differently in \cite{Prenter2018}, and that it differs from the energy norm that will be introduced in Section~\ref{sec:stabilityAnalysis}. On account of the last term, this norm cannot be applied to the full space $H^1(\Omega)$. For this reason, it is assumed that $u \in H^2(\Omega)$ in the analysis, and the setting is restricted to the composite space $V_h \oplus H^2(\Omega)$. In \cite{Prenter2018} it is demonstrated that the Nitsche-based approximation possesses a best-approximation property with respect to the $\beta$-norm, but that the error of the best approximation in this norm can still be arbitrarily large. In view of this and the fact that the equivalence between the $\beta$-norm in~\eqref{eq:immersedenergynorm} and the $H^1(\Omega)$-norm can be arbitrarily weak, error bounds in the $H^1(\Omega)$-norm cannot be provided. Ref.~\cite{Prenter2018} provides examples of computations where direct application of Nitsche's method fails for unfitted problems and provides references in which nonphysical stress patterns on small cut elements are reported.

\subsection{Remedies and literature overview}\label{sec:stabilityLiterature}

Multiple stability-enhancing techniques have been developed to overcome the aforementioned problem. A list of these is presented below. Two techniques in particular are discussed in detail in the subsequent subsections, \emph{viz.}\ element aggregation in Section~\ref{sec:aggregation} and ghost-penalty stabilization in Section~\ref{sec:ghost}. Essentially, these approaches provide control over the gradients on cut elements in terms of the gradients on interior elements, in a manner that is consistent with the original problem. Consequently, extended coercivity holds with respect to $H^1(\Omega_h)$ (\emph{n.b.}\ on the entire active domain), and the bilinear form is coercive with a well-behaved Nitsche parameter $\beta \sim h_T^{-1} (\lesssim h_{T_\Omega}^{-1})$. With element aggregation this is achieved through a modification of the approximation space, while with the ghost penalty a consistent stiffness is added to the bilinear form. As a result, these techniques provide optimal approximation properties, and condition number estimates analogous to those for boundary-fitted finite element formulations. Therefore, both aspects of the small-cut-element problem are resolved simultaneously. A unified analysis of the two methods is presented in Section~\ref{sec:stabilityAnalysis}.

\begin{itemize}
\item By employing a \emph{different formulation} to enforce boundary conditions, the stability problems caused by Nitsche's method are circumvented. Such alternative methodologies to enforce boundary conditions are discussed in Section~\ref{sec:dirichlet} (\emph{e.g.}, the nonsymmetric Nitsche method \cite{Burman2012,Boiveau2016,Schillinger2016} or the shifted boundary method \cite{Main2018II,Main2018I,Atallah2020,Atallah2021}). For some applications, however, these approaches lead to other complications, such as loss of consistency, symmetry, or adjoint consistency. 

\item A \emph{fictitious domain stiffness} is commonly employed with high-order discretizations in the finite cell method \cite{Parvizian2007,Duester2008} (see \cite{SchillingerRuess2015,Duester2017} for reviews). This approach is mathematically analyzed in \cite{Dauge2015}. With a fictitious domain stiffness, the volume integrals of the bilinear operator are extended into the fictitious domain multiplied by a small parameter (in practical applications this only holds for the fictitious parts of cut elements, \emph{i.e.}, $\Omega_h\setminus\Omega$). This provides a stiffness on the fictitious part of cut elements, providing some control over the approximate solution on these. This approach has been applied to many real-world applications, notable examples of which are implant-vertebra models \cite{Elhaddad2017} and additively manufactured structures \cite{Korshunova2020}. The fictitious domain stiffness is not consistent with the original problem, however. For this reason, a drawback of this approach is that it can be challenging (or for certain applications impossible) to set the parameter low enough to keep the consistency error at an acceptable level, but high enough to obtain sufficient control over the fictitious parts of elements. Besides the stability, the fictitious domain stiffness also improves the conditioning, but generally not to a level that permits the application of iterative solvers.

\item A \emph{minimal stabilization procedure} is developed in~\cite{Buffa2020}. In this procedure, the elements are first classified as good (resp. bad) elements, \emph{i.e.}, elements with a sufficiently (resp. insufficiently) large overlap with the physical domain. Next, for each bad element, the normal derivative of the discrete function on the Dirichlet boundary that appears in the boundary integral of the bilinear form, is replaced by the normal derivative of an extension of the function on a nearby good element. This gives rise to a similar trace theorem as in boundary-fitted discretizations and, accordingly, the weak formulation is coercive with a Nitsche parameter that is inversely proportional to the (untrimmed) element size of the background grid such that optimal error estimates can be derived. This technique has mainly been applied to trimmed patches in isogeometric analysis, and can also be employed for coupling conditions on unfitted interfaces in overlapping multi-patch discretizations \cite{Antolin2021}. Additionally, in \cite{Puppi2020} it is shown that this methodology can be used to obtain inf-sup stable mixed formulations. It is to be noted that this approach does not provide control of the gradients in the (fictitious part of) cut elements themselves, and therefore only provides coercivity with respect to $H^1(\Omega)$, not extended coercivity with respect to $H^1(\Omega_h)$. As such, this approach yields stability and provides a well-posed imposition of boundary conditions, but it does not solve conditioning problems.

\item A \emph{recovery-based stabilization technique} is commonly applied in the Cartesian grid finite element method (cgFEM, \cite{Nadal2013,Marco2015}), and is similar to the ghost penalty with patch control~\cite{Navarro-Jimenez2020}. In this stabilization procedure, the solution on a small cut element is weakly constrained by adding an $L^2$-projection term over the entire background element to the bilinear form. Instead of constraining the solution on the small cut elements to a smooth extension of the solution on neighboring elements, Richardson iteration is applied to constrain the solution to a smooth extension of the approximate solution in the previous iteration. As a benefit of this technique, \cite{Navarro-Jimenez2020} mentions that it does not require operators that are not generally available in (immersed) finite element codes. Furthermore, \cite{Navarro-Jimenez2020} establishes the convergence of the Richardson iteration and a bound on the condition number similar to that of boundary-fitted finite element methods, and also illustrates these aspects numerically.

\item A \emph{least squares stabilization} term can be applied when the employed function space is (at least) $C^1$-continuous. In \cite{Elfverson2019} and \cite{Larsson2022} it is demonstrated that a stable and coercive formulation can also be achieved by applying a least squares finite element term in the vicinity of the Dirichlet boundary. Because in $H^2(\Omega)$ the normal gradient on the boundary can be controlled by volumetric terms, coercivity is even achieved in the full space instead of only in the discrete space. A drawback of this approach is that it only applies to $C^1$-continuous bases. The least squares stabilization does not repair the conditioning, such that in \cite{Elfverson2019} this approach is combined with the removal of basis functions with very small supports in the physical domain and in \cite{Larsson2022} it is combined with a fictitious domain stiffness.

\end{itemize}

\subsection{Aggregated finite element methods}\label{sec:aggregation}

In this section, we introduce a methodology to solve both the stability and conditioning issues described in Section \ref{sec:introStabilityConditioning}. The method was originally proposed in \cite{Badia2018} and was coined the \emph{aggregated} finite element method (AgFEM). As the name points out, the method relies on an aggregation (also called agglomeration) of elements. This element aggregation is used to define a discrete extension operator from degrees of freedom on well-posed (interior) elements to ill-posed (cut) elements. The AgFEM space is the image of this operator. Thus, as indicated in Section~\ref{sec:stabilityLiterature}, AgFEM solves the issues described above by modifying the finite element space, without altering or perturbing the bilinear form. We present this idea for (nodal or Lagrangian) $C^0$-continuous finite element spaces. It can also be applied to discontinuous Galerkin methods, even though this case is trivial; discontinuous Galerkin schemes can readily be applied to the polytopal meshes obtained after aggregation. 

\subsubsection*{Geometrical construction} 

Let us recall from Section~\ref{sec:method_formulation} the problem domain $\Omega \subset \mathbb{R}^d$ that is embedded in the ambient domain $\mathcal{A}\supset\Omega$. The mesh $\mcThA$ is a quasi-uniform and shape-regular discretization of $\mathcal{A}$, $\mcTh$ is the set of all active elements $T\in\mcThA$ that intersect $\Omega$, and the union of all active elements is defined as $\Omega_h = \cup_{T\in\mcTh} T \supset \Omega$. The set of cut elements that intersect the boundary $\partial\Omega$ of the problem domain is denoted as $\mcThcut \subset \mcTh$ and additionally we define the set of interior elements as $\mcThin = \mcTh \setminus \mcThcut$, see Figure~\ref{fig:agg-1}. As previously mentioned, immersed finite element formulations can lead to ill-conditioned discrete systems and unbounded gradients on cut elements. The interior elements in $\mcThin$ do not play a role in these problems, but it is the cut elements in $\mcThcut$ that can lead to the so-called small-cut-element problem. 

\begin{figure}[htpb]
    \begin{subfigure}[b]{0.49\textwidth}
    \centering
    \includegraphics[width=0.7\textwidth]{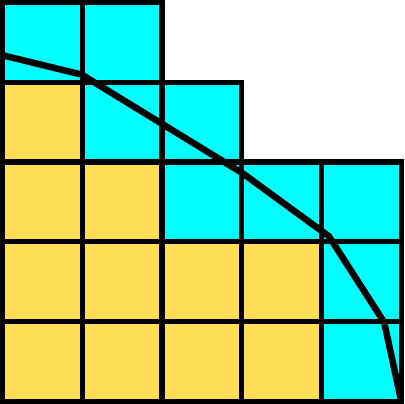}
    \caption{Interior and cut elements}
    \label{fig:agg-1}
    \end{subfigure}
    \begin{subfigure}[b]{0.49\textwidth}
        \centering
        \includegraphics[width=0.7\textwidth]{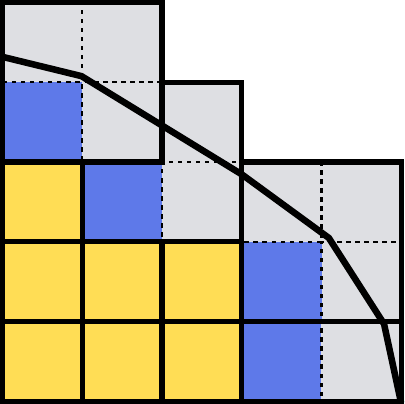}
        \caption{Aggregates and root elements}
        \label{fig:agg-1b}
    \end{subfigure}
    \caption{On the left, we show the active portion of a  background mesh $\mcTh$, with interior elements in $\mcThin$ (yellow) and cut elements in $\mcThcut$ (light blue). On the right, we create aggregates (in gray) composed of one root element (dark blue) and some cut elements. The curved black line represents the boundary $\partial \Omega$.}
\end{figure}

The definition of finite element spaces that are robust to cut-element configurations makes use of an element aggregation strategy. The idea is to aggregate cut elements to interior elements. An aggregate contains one interior element (the root element) and one or more cut elements. We do not make any assumption on the shape of aggregates. As a result of this finite element space definition, aggregation is only active on the boundary. Interior elements that are not in touch with cut elements are not affected. We represent the aggregated mesh with $\mcThag$. In Figure~\ref{fig:agg-1b}, cut elements have been aggregated to interior elements, creating aggregates. The root element is highlighted with a different color in each aggregate. The aggregation algorithm is straightforward and can be found, \emph{e.g.}, in \cite{Badia2018}. The algorithm proposed in this reference is iterative. At each iteration, cut elements are aggregated to an interior element or to a previously aggregated element (after the first iteration). The aggregation should minimize the aggregate size, since this will affect the final accuracy of the solution.

\begin{remark}
  In this discussion, we considered for simplicity that cut elements are ill-posed and must be aggregated to interior elements. However, this naive definition of ill-posed elements is not a requirement of the method and more subtle definitions can be considered. For instance, one can define a threshold volume fraction $\eta^* \in [0,1]$ (see \eqref{eq:ratio}) and mark any element $T_i \in \mathcal{T}_h^{\mathrm{cut}}$ with $\eta_i \le \eta^*$  as ill-posed. The naive case is recovered for $\eta^* = 1$ and the standard Galerkin method with $\eta^* = 0$. One can verify that the condition number and stability bounds hold for this relaxed definition with constants that depend on $\eta^*$. $\eta^* = 1$ is an excellent choice as soon as one wants to {have enough resolution to capture geometrical details}, which is the situation in most cases. For under-resolved situations, \emph{e.g.}, for thin-walled structures in which the thickness scale $t_\epsilon$ is not captured by the background mesh, a more clever choice of $\eta^*$, \emph{e.g.}, $\eta^* \sim \frac{t_\epsilon}{h}$, is required. In practice, the choice of $\eta^*$  is a trade-off between accuracy (aggregate size) and well-posedness (condition number bound) of the resulting linear system. We note that the definition and implementation of AgFEM is independent of $\eta^*$, since this parameter only affects the geometrical algorithm that aggregates elements.
\end{remark}

\subsubsection*{Discontinuous spaces}

Let us consider first the case of discontinuous Galerkin finite element spaces. We denote by $V_h^-$ a discontinuous Galerkin space of a given order. The main reason why standard finite element methods on unfitted meshes fail is that one has no control over $\|u_h\|_{\Omega_h}$. Control over this quantity implies the stability of the method regardless of the cut location. Let us consider a discontinuous Galerkin space on top of $\Omega_h$ (see Figure~\ref{fig:agg-2}). The support of a shape function on a cut element $T \in \mcThcut$ is $T_\Omega$. This intersection depends on the cut location, and its measure can be arbitrarily small. In the limit of vanishing measure, the value of the degree of freedom associated with this shape function does not affect the linear system. Thus, the problem is singular.
\begin{figure}[htpb]
    \begin{subfigure}[b]{0.49\textwidth}
    \centering
    \includegraphics[width=0.7\textwidth]{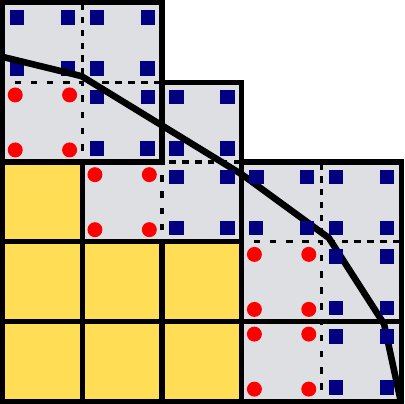}
    \caption{Discontinuous Galerkin constraints}
    \label{fig:agg-2}
    \end{subfigure}
    \begin{subfigure}[b]{0.49\textwidth}
    \centering
    \includegraphics[width=0.7\textwidth]{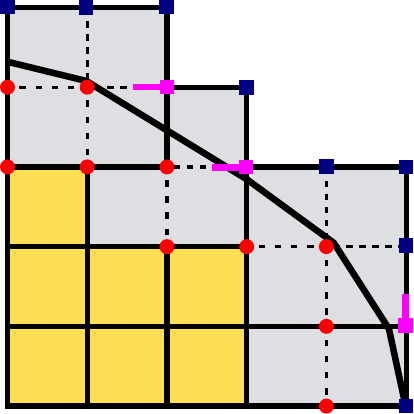}
    \caption{Continuous Galerkin constraints}
    \label{fig:agg-3}
    \end{subfigure}
    \caption{On the left, we depict the degrees of freedom of a piecewise linear discontinuous Galerkin space $\mathcal{Q}^{-}_1(\Omega_h)$. The nodes are deliberately placed in the element interiors to make it evident that they do belong to the element. Red circles represent well-posed degrees of freedom and blue squares ill-posed ones. In this case, for each ill-posed degree of freedom, its constraining degrees of freedom are the well-posed ones in the same aggregate using the expression in (\ref{eq:agg-constraints}). On the right, we glue together degrees of freedom between elements to enforce ${C}^0$ continuity. In order to define the constraints, we must provide the concept of the aggregate owner of ill-posed degrees of freedom. This is obvious for aggregate-interior nodes (there is only one option). For the ill-posed degrees of freedom on aggregate interfaces, lines point to the owner aggregate. With this ownership information, one can now constrain the ill-posed degrees of freedom by the well-posed ones in the aggregate that owns them using the same expression (\ref{eq:agg-constraints}) as above.}
\end{figure}

Given a mesh $\mcTh$, we can define the discontinuous element-wise spaces
\[
    \mathcal{P}_p^{-}(\mcTh) = \prod_{T \in \mcTh} \mathcal{P}_p(T), \qquad \mathcal{Q}_p^{-}(\mcTh) = \prod_{T \in \mcTh} \mathcal{Q}_p(T)
\] 
depending on the local finite element space, namely polynomials up to order $p$ or the tensor product of univariate polynomials of order $p$. The discontinuous Galerkin space $V_{h}^{-}$ on the active mesh can either be $\mathcal{P}_p^{-}(\Omega_h)$ or $\mathcal{Q}_p^{-}(\Omega_h)$. Other spaces, like serendipity finite elements, could readily be considered. The space in Figure~\ref{fig:agg-2} corresponds to $\mathcal{Q}_1^{-}(\Omega_h)$. 

Based on the previous discussion, it is straightforward to check that degrees of freedom on interior elements are not problematic. The corresponding shape function has support on a whole element. Only the degrees of freedom on cut elements are problematic. Discontinuous Galerkin methods are geometrically flexible, \emph{i.e.}, the functional space and the topology of the elements are not connected. Thus, we can simply define the aggregated discontinuous Galerkin space as $\mathcal{P}_p^{-}(\mcThag)$ or $\mathcal{Q}_p^{-}(\mcThag)$. All these shape functions are polynomials, and their support is at least one full interior element. Thus, it can be checked that 
\begin{equation}\label{eq:agg-l2-stab}
\|v_h \|_{\Omega_h} \lesssim \|v_h \|_{\Omega} \qquad v_h \in  V_h^{\mathrm{ag},-}
\end{equation}\vspace{.5mm}
Extended stability of gradients is obtained using the same argument. Gradients of functions in $\Vhagmin$ are polynomials of degree $p-1$ in each aggregate, and the norm on the whole aggregate can be bounded by the one on the root element, getting
\begin{equation}\label{eq:agg-h1-stab}
\|\nabla v_h \|_{\Omega_h} \lesssim \| \nabla v_h \|_{\Omega} \qquad v_h \in  \Vhagmin
\end{equation}\vspace{.5mm}
These relations are also employed in the mathematical analysis of stabilized methods in Section~\ref{sec:stabilityAnalysis}. The constants in the extended stability bounds \eqref{eq:agg-l2-stab} and \eqref{eq:agg-h1-stab} depend on the order of approximation, but are independent of the background mesh size $h$ and the cut location. Thus, the small-cut-element problem is solved when using discontinuous Galerkin methods on aggregated meshes. Similarly, cell merging strategies in the finite volume setting have been proposed, \emph{e.g.}, in \cite{hunt2004adaptive}. Its application to the discontinuous Galerkin method can be found in \cite{Johansson2012,Mller2013}.

The aggregation method preserves the order of accuracy if the aggregate size is proportional to the background mesh size. The ratio between the aggregate size and root element size is related to the geometry of the domain boundary (see \cite{Badia2018} for more details) and can be improved with refinement (for a given definition of the domain boundary $\partial \Omega$). For a fixed mesh representation of the domain boundary (the standard situation), the assumption that the aggregate size is proportional to the background mesh size holds for $h$ small enough.

\subsubsection*{Constraining discontinuous spaces}

\vspace{1mm}

It is straightforward to see that $\Vhagmin \subset V_h^-$. Even though it is not really needed for the discontinuous Galerkin method, one can build $\Vhagmin$ by constraining $V_{h}^-$. The idea is to define an extension operator $\mathcal{E}_{h}^{\mathrm{ag},-}: V_{h}^{\mathrm{in},-} \rightarrow \Vhagmin$, with $V_{h}^{\mathrm{in,-}}$ the restriction of $V_{h}^-$ to $\mcThin$. In the discontinuous Galerkin case, this extension operator can be defined aggregate-wise. The idea is to extend the root-element shape functions to the whole aggregate. This is equivalent to constraining degrees of freedom on cut elements to the degrees of freedom on the respective root elements.   

Let us recall Ciarlet's definition of finite elements; see, \emph{e.g.}, \cite{ErnGuermond}. For nodal finite element methods (continuous or discontinuous), one can define a set of points (so-called nodes) and an associated set of degrees of freedom, corresponding to pointwise evaluation at the nodes. Nodes are chosen such that the degrees of freedom uniquely define functions in $V_{h}^-$ and thus define a basis for its dual space. The dual basis of the degrees of freedom are the so-called shape functions. Thus, there is a correspondence between nodes, degrees of freedom and shape functions. We use the following notation: given a node $\boldsymbol{\alpha}_i$, we represent its corresponding shape function and degree of freedom by $\phi_i$ and $\sigma_i(\cdot)$, respectively. The duality between the degrees of freedom and the shape functions implies 
$\sigma_i(\phi_j)=\delta_{ij}$ with $\delta_{ij}$ the Kronecker delta, and the relation between the nodes and the degrees of freedom in turn implies $\phi_i(\boldsymbol{\alpha}_j)=\delta_{ij}$. By virtue of the dual relation between degrees of freedom and shape functions, the following correspondence holds for all $v_h^-\in{}V_h^-$: $y_i=\sigma_i(v_h^-)$ if and only if $v_h^-=\sum_i y_i \phi_i$.

We can classify the degrees of freedom as \emph{well-posed} (the ones of the root element) and \emph{ ill-posed} (the ones on the cut elements). We denote the set of well-posed (resp., ill-posed) degrees of freedom by $\mathcal{O}_{h}^{\mathrm{wpd}}$ (resp., $\mathcal{O}_{h}^{\mathrm{ipd}}$). Next, we define a map $\mathcal{O}_{h}^{\mathrm{ipd} \to \mathrm{wpd}}(\cdot)$ that for an ill-posed degree of freedom returns the well-posed ones. For discontinuous Galerkin methods, this map is defined aggregate-wise. At each aggregate $T \in \mcThag$, the ill-posed degrees of freedom (blue squares in each aggregate of Figure~\ref{fig:agg-2}) are constrained by the values of its corresponding well-posed degrees of freedom (the ones in the corresponding root element, red circles in Figure~\ref{fig:agg-2}). More specifically, an ill-posed degree of freedom $\sigma_i(\cdot)$ is constrained as
\vspace{1.0mm}
\begin{equation}\label{eq:agg-constraints}
\sigma_i(\cdot) = \sum_{\sigma_j \in \mathcal{O}_{h}^{\mathrm{ipd}\to \mathrm{wpd}}(\sigma_i)} \sigma_i(\phi_j) \sigma_j(\cdot) = \sum_{\sigma_j \in \mathcal{O}_{h}^{\mathrm{ipd}\to \mathrm{wpd}}(\sigma_i)} \phi_j(\boldsymbol{\alpha_i}) \sigma_j(\cdot)
\end{equation}
\vspace{1.0mm}
where the basis functions $\phi_j$ of well-posed degrees of freedom are evaluated outside of the elements on which these are defined in $\sigma_i(\phi_j)$ and $\phi_j(\boldsymbol{\alpha_i})$. Equation \eqref{eq:agg-constraints} introduces constraints on ill-posed values, stating that the values on ill-posed nodes are not free but are determined from the well-posed ones. More specifically, the ill-posed value is equal to the nodal values on the root element multiplied by the value of the corresponding shape function evaluated at the ill-posed node. Doing this, we are constraining the ill-posed degrees of freedom to be the extension of the values in the interior. This extension applied to functions in $V_{h}^-$ defines $\mathcal{E}_{h}^{\mathrm{ag},-}$, whose image is $\Vhagmin$. 

\vspace{1.0mm}
 
\subsubsection*{Aggregation for continuous spaces}

\vspace{1.0mm}

Now, we would like to generalize the extension-based definition of the aggregated discontinuous finite element space to ${C}^0$ Lagrangian spaces. First, we discuss how the active and interior ${C}^0$ spaces can be obtained from the discontinuous space by enforcing ${C}^0$ continuity of piecewise polynomials using a local-to-global map. 
We also discuss why this approach cannot be easily applied to the aggregated meshes. Instead, we introduce the concept of ownership of ill-posed degrees of freedom on aggregates, and define an extension operator from well-posed to ill-posed degrees of freedom that preserves ${C}^0$ continuity. 

We can readily define the ${C}^0$ Lagrangian finite element space as a subspace of the discontinuous space as $V_h = V_h^- \cap {C}^0(\Omega_h)$, and its restriction to interior elements is represented with $V_{h}^{\mathrm{in}}$. In standard finite element methods, one can enforce ${C}^0$ continuity by using the previously introduced one-to-one relation between nodes, degrees of freedom, and shape functions. A shape function is different from zero only on its corresponding node, and zero on all the other nodes, such that the degree of freedom is just the evaluation on its corresponding node. ${C}^0$ continuity is enforced by the following equivalence class: two degrees of freedom must have the same value (irrespective of the element) if their corresponding nodes are in the same spatial point. Doing this, we glue together degrees of freedom from different elements. Figure~\ref{fig:agg-3} illustrates the result after gluing together the element-wise degrees of freedom in Figure~\ref{fig:agg-2}. This enforcement of continuity is amenable for implementation: one defines an element-wise matrix assembly and a local-to-global index map for degrees of freedom.

Unlike discontinuous Galerkin methods, this construction only holds for particular element-wise polynomial spaces and element topologies. First, it is not straightforward to implement such continuity in aggregated meshes. Second, it can destroy the approximation properties of $\Vhag$. Let us take a look at an ill-posed degree of freedom on the interface between two aggregates (purple nodes in Figure~\ref{fig:agg-3}). In a straightforward extension of the constraints in \eqref{eq:agg-constraints}, such a degree of freedom would be constrained by the multiple aggregates that contain it. 
This, together with the enforcement of ${C}^0$-continuity, would couple the degrees of freedom between separate aggregates, and thereby could ruin approximation properties. In extreme cases, it can lead to over-constrained systems. This \emph{locking\/} phenomenon has been analyzed in detail in \cite{badia2021linking}.

The aggregated finite element method in \cite{Badia2018} was designed to solve this locking phenomenon and provide an accurate ${C}^0$ Lagrangian finite element space on unfitted meshes that is easy to implement. The idea is to introduce the concept of \emph{ownership\/} for ill-posed degrees of freedom. For each ill-posed global degree of freedom in $\Vhag$, one assigns one of the aggregates in $\mcThag$ containing this degree of freedom as the owner. The choice of the owner is arbitrary and does not affect the main properties of the method. We can define a map $\mathcal{O}_{h}^{\mathrm{ipd} \to \mathrm{wpd}}$ using the fact that each aggregate contains only one interior element. Given an ill-posed degree of freedom, its corresponding well-posed constraining degrees of freedom are those of the root element of the aggregate owner. We illustrate this construction in Figure~\ref{fig:agg-3}. The purple elements are the ones that belong to more than one aggregate (the other ones are straightforward). For these degrees of freedom, the small rectangle points to the aggregate owner of the degree of freedom. With this construction, we can now constrain global ill-posed degrees of freedom using the same constraints as for discontinuous Galerkin methods (\ref{eq:agg-constraints}). 

These constraints are local and define an extension $\mathcal{E}_{h}^{\mathrm{ag}}: V_{h}^{\mathrm{in}} \rightarrow V_{h}^{\mathrm{ag}}$. In fact, the aggregated finite element space is the image of this extension operator. The method satisfies the desired approximation properties and the extended $L^2(\Omega_h)$ and $H^1(\Omega_h)$ stability in (\ref{eq:agg-l2-stab}) and (\ref{eq:agg-h1-stab}); see \cite{Badia2018} for more details. We also refer to \cite{badia-robust-high-order} for an alternative definition of the discrete extension operator that relies on interpolation for high-order polynomial bases and to \cite{Marussig2017,Marussig2018,Burman2022Splines} for the extension of higher-order continuous (isogeometric or spline-based) discretizations. The definition of a space-time discrete extension operator for moving domains and interfaces can be found in \cite{badia-space-time}.

\subsubsection*{Implementation aspects}

The aggregation algorithm can readily be implemented in a finite element code that supports mesh adaptivity. First, local element matrices are computed as usual. Next, the constraints for the aggregated finite element method are direct constraints, \emph{i.e.}, constraints which are enforced in the local-to-global assembly by expressing the ill-posed degrees of freedom in terms of well-posed ones to eliminate them from the system. 

The distributed-memory implementation of the aggregated finite element method has been considered in \cite{Verdugo2019}. The proposed implementation starts with a standard distributed implementation of the finite element method that relies on a sub-mesh partition (with gluing info, \emph{e.g.}, via one layer of neighbor elements). First, one must parallelize the aggregation algorithm. The parallel implementation of this step is straightforward; one needs to perform nearest neighbor communications at the end of each iteration of the aggregation algorithm to propagate information. The result is identical to the serial version (see \cite[\S 3.2]{Verdugo2019}).

In general, the aggregation strategy is such that aggregates can have support on multiple processors. Thus, constraints in (\ref{eq:agg-constraints}) cannot be computed locally. To compute these constraints in a parallel environment, one requires information from the root element. The cut element to root element map (and its processor Id) is a by-product of the parallel aggregation step. This way, it is easy to define the information that each processor must receive from other processors to compute constraints. However, the inverse map that provides the cut elements constrained by a given root element is not straightforward. This map is required to prepare the data that needs to be sent to other processors. The parallel inverse path reconstruction algorithm in~\cite[\S 3.5]{Verdugo2019} generates this information.

\subsubsection*{Adaptive meshes}

The size of the aggregates on the boundary can harm the accuracy of the method. One way to limit the size of aggregates is via mesh refinement. However, local mesh refinement in continuous finite elements is not straightforward. In the resulting mesh so-called hanging nodes appear, which must be constrained by regular nodes to keep ${C}^0$ continuity. The nature of this constraint is analogous to the ones defined in (\ref{eq:agg-constraints}) for aggregated finite elements.

The combination of mesh adaptivity and aggregation constraints is not straightforward. First, the definition of ill-posed versus well-posed degrees of freedom is more complex. A degree of freedom that only belongs to cut elements can be well-posed because it constrains well-posed hanging degrees of freedom. The right intertwining of these two sets of constraints has been analyzed in \cite{Badia_2021-1}. This work proposed a two-step algorithm to construct the discrete extension operator that carefully mixes aggregation constraints of problematic degrees of freedom and standard hanging degree of freedom constraints. Following this approach, the aggregated finite element space is defined as a discrete extension with well-defined linear constraints.

\subsubsection*{Extension to other problems}

The aggregated spaces can be used to discretize elliptic and parabolic problems (combined with some time-stepping scheme). These schemes have, for example, been applied to nonlinear solid mechanics in \cite{Badia2021}. The extension to interface problems is possible by defining independent aggregated spaces on both sides of the interface \cite{Neiva_2021}. In this work continuity of traces is weakly enforced (as Dirichlet data), which leads to robust schemes that can handle high-contrast problems.

The extension of aggregated finite elements to indefinite systems has been analyzed in \cite{Badia2018b}. We note that its application to stabilized finite element formulations that transform the indefinite system into a definite one is straightforward~\cite{codina2018variational}. The stability of indefinite problems relies on so-called inf-sup conditions. These conditions are stringent and are only valid for very specific \emph{i.e.}, finite element spaces. For instance, the velocity and pressure finite element spaces in the (Navier\nobreakdash--)Stokes problem have to match to satisfy this condition. Aggregation on these spaces does not generally preserve inf-sup stability. The aggregated mixed finite element method for the Stokes problem in~\cite{Badia2018b} starts with the inf-sup stable pair $\boldsymbol{Q}_p(\mcThin) \times \mathcal{P}_{p-1}^{-}(\mcThin)$, \emph{i.e.}, the velocity resides component-wise in~$\mathcal{Q}_p(\mcThin)\cap{C}^0(\mcThin)$ and the pressure is a discontinuous function 
in~$\mathcal{P}_{p-1}(\mcThin)$. The extension defined above for continuous and discontinuous spaces leads to a pair that does not satisfy the inf-sup condition. Instead, a modified extension of the velocity components is proposed in~\cite{Badia2018b}, which relies on serendipity finite elements. In some specific situations, additional pressure jump stabilization terms must also be added.

Aggregation has also been applied to enable stable explicit time stepping, such as in, \emph{e.g.}, \cite{burman2020explicit} for the wave equation.

\subsubsection*{Solving the linear system} 

One of the main motivations of aggregated finite element methods is to produce a well-posed discrete system. For second-order elliptic problems, one can prove the same condition number bounds as for boundary-fitted formulations. More specifically, the condition number bound is of the order of $h^{-2}$; see~\cite[Corollary~5.9]{Badia2018}. Thus, the condition number is independent of the mesh location. As a result, one can apply standard direct or iterative solvers and preconditioning techniques. However, the sparsity pattern of the matrix is different from standard finite elements. As in adaptive finite elements, the sparsity pattern is affected by the constraints.

The linear solver step in aggregated finite element methods has been studied in \cite{Verdugo2019}. The authors propose a row-wise linear algebra distribution layout for the resulting matrix in \cite[\S 3.6]{Verdugo2019} and minimize the amount of inter-processor communications. The authors use the parallel implementation of aggregated finite elements in $\mathtt{FEMPAR}$ \cite{Badia2018-fempar}. They apply a standard algebraic multigrid preconditioner (called $\mathtt{GAMG}$) from the $\mathtt{PETSc}$ library \cite{petsc-user-ref} using a standard configuration. The results in \cite[\S 4]{Verdugo2019} show excellent scalability and optimality properties for this standard solver applied to the aggregated finite element method. The aggregated results are very similar to the results obtained with boundary-fitted meshes. The largest problems that are considered exceed $10^9$ elements on 16,464 processors. From these numerical experiments, it is observed that the additional coupling between boundary degrees of freedom does not affect a standard algebraic multigrid preconditioner. 

\subsubsection*{A numerical example}

This section is a summary of the numerical experiments in \cite{Verdugo2019}, in which the aggregated finite element method is considered. Herein, we are particularly interested in the linear solver step when using the aggregated finite element method. We refer the interested reader to \cite{Verdugo2019}, in which one can find a more thorough numerical experimentation, including, \emph{e.g.}, the time spent in all stages of the simulation and their scalability or the effect of some algorithmic parameters on the results.

\def\FEMPAR{{\texttt{FEMPAR}}}
\def\dealii{{\texttt{deal.II}}}
\def\p4est{{\texttt{p4est}}}
\def\t8code{{\texttt{t8code}}}
\def\fenics{{\texttt{FEniCS}}}
\def\petsc{{\texttt{PETSc}}}
\def\trilinos{{\texttt{TRILINOS}}}

The motivation for this test is to show that the aggregated finite element method leads to a well-posed discrete system for which one can readily use standard iterative solvers and preconditioners. To this end, we consider the widely available linear solvers available in the \petsc~library \cite{petsc-user-ref}. We use a conjugate gradient method from the \texttt{KSP} module of \petsc, preconditioned with the smoothed-aggregation algebraic multigrid preconditioner \texttt{GAMG} \cite{May2016}. We set up the linear solver by relying as much as possible on the default configuration given by \texttt{GAMG} to effectively show that the aggregated finite element spaces lead to linear systems that can be efficiently solved using standard multigrid tools. The software used for these experiments is the MPI-parallel implementation of the AgFEM  available at \FEMPAR{}~\cite{Badia2018-fempar} linked with \p4est{} v2.0~\cite{burstedde_p4est_2011} for octree mesh handling and $\mathtt{PETSc}$ v3.9.0~\cite{petsc-user-ref} for distributed-memory linear algebra data structures and solvers.

We consider the Poisson equation with Dirichlet boundary conditions weakly imposed using Nitsche's method, as introduced in Section~\ref{sec:method_formulation}. The Nitsche parameter is taken as $\beta=10\,h^{-1}$. When using the standard finite element space without aggregation, this expression for the Nitsche parameter does not necessarily lead to a coercive bilinear form, as explained in Section~\ref{sec:method_formulation}. Therefore, in the results with standard finite element spaces, an element-wise Nitsche parameter that follows from a local generalized eigenvalue problem is employed. Note that, as mentioned previously, this local value of the parameter is not bounded due to the lack of stability of the method.

We  consider hexahedral elements with continuous piecewise trilinear shape functions. The geometry in the experiments is the complex body depicted in Figure~\ref{fig:geom}, which is often considered in the literature of immersed finite element methods; see \cite{Burman2015}. The bounding box used to define the background mesh is the unit cube $[0, 1]^3$, which is also depicted in Figure~\ref{fig:geom}. We consider background meshes of hexahedral elements generated and partitioned in parallel with the $\mathtt{p4est}$ library. The test is performed for three different local loads (referred to as ``load 1", ``load 2" and ``load 3") with $20^3$, $30^3$, and $40^3$ elements per parallel subdomain, respectively. We generate the meshes by recursive subdivision of the unit cube. The finest mesh considered in the examples has $1,073,741,824$ elements (corresponding to refinement level 10), and it is partitioned into $16,777$ subdomains which are mapped to the same number of MPI tasks.

\begin{figure}[ht!]
  \centering
    \includegraphics[width=0.45\textwidth]{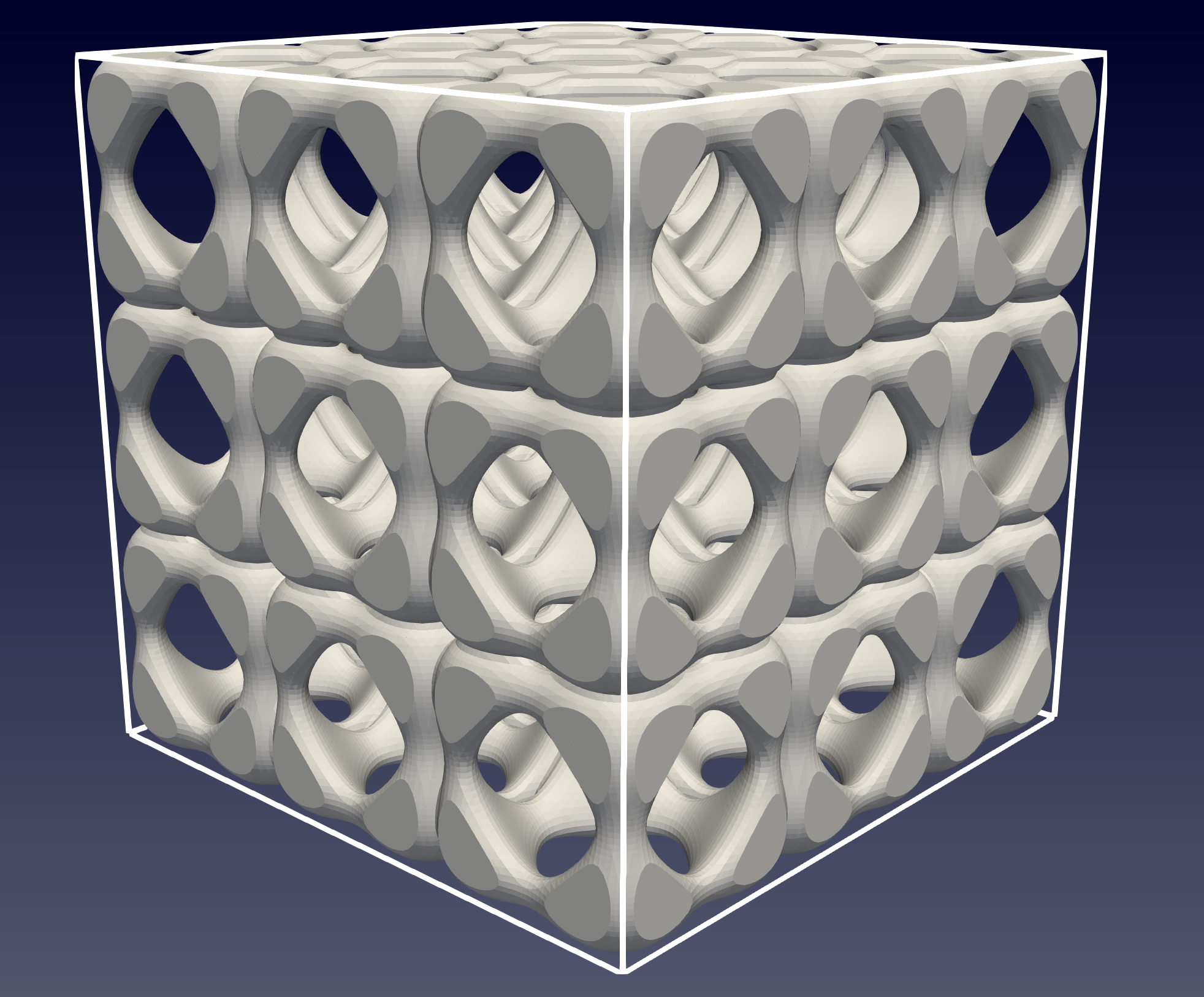}
\caption{View of the geometry and the bounding box considered in the numerical example. Figure reproduced from \cite{Verdugo2019}.}
\label{fig:geom}
\end{figure}

First, we analyze the impact of using aggregated finite-element or standard finite-element spaces on the number of solver iterations; see Figure~\ref{fig:cg-iter-agg-vs-std-c}. This figure conveys that the use of aggregated finite element spaces is beneficial 
(and generally indeed essential) to achieve good performance of the linear solver. For the standard finite element spaces, some 
results are missing in Figure~\ref{fig:cg-iter-agg-vs-std-c}. These correspond to cases in which the linear solver was not able to provide a converged solution. This shows that standard finite element spaces are not reliable in immersed discretization methods. In contrast, aggregated finite element spaces allow one to effectively solve the underlying linear systems. The results obtained with aggregated finite element spaces (blue lines) are very close to the expected optimal performance of multigrid methods (\emph{i.e.}, number of linear solver iterations is asymptotically independent of the problem size).
Figure~\ref{fig:time-solver} reports the wall clock time spent in the linear solver step, separated into time spent in the solver setup and in the solver run. 
The solver wall clock time increases only moderately as the problem size increases.

\begin{figure}[ht!]
  \centering
  \begin{subfigure}{0.45\textwidth}
    \includegraphics[scale=1]{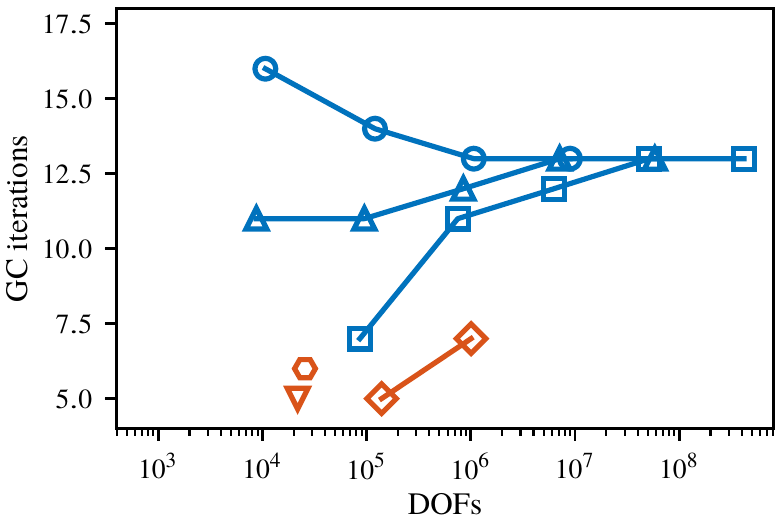}\caption{Solver iterations}
    \label{fig:cg-iter-agg-vs-std-c}
  \end{subfigure}
    \hspace{3em} 
  \begin{subfigure}{0.45\textwidth}
    \includegraphics[scale=1]{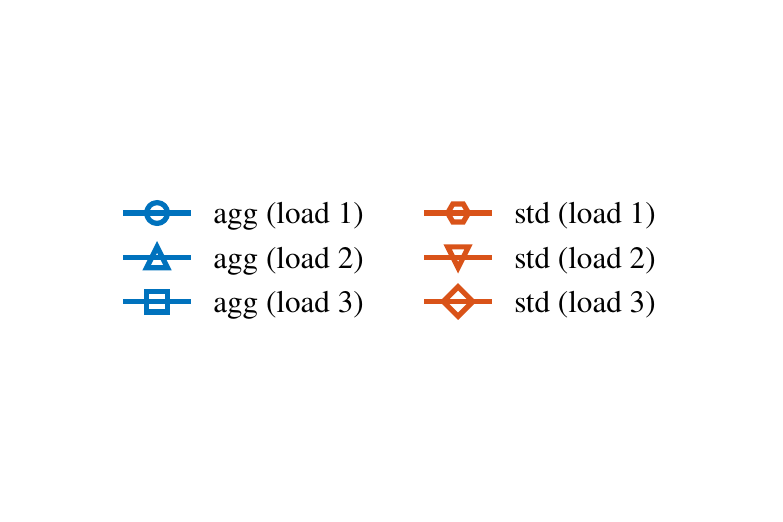}
  \end{subfigure}

    \begin{subfigure}{0.45\textwidth}
    \includegraphics[scale=1]{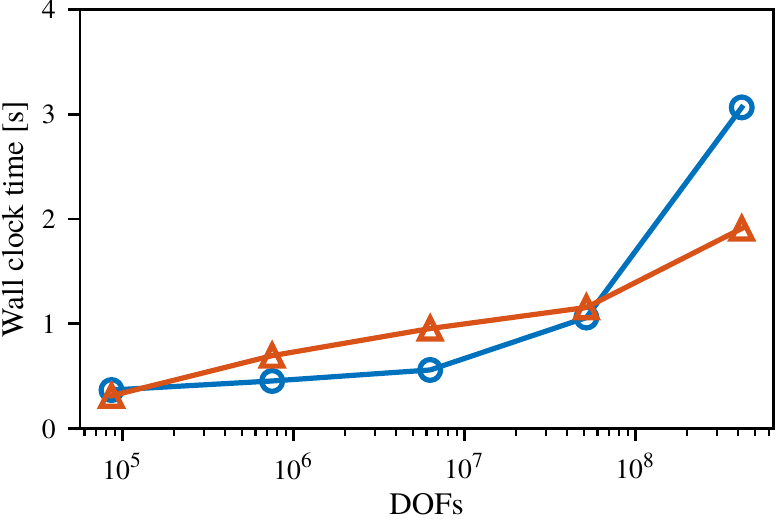}
    \caption{Wall clock time}
    \label{fig:time-solver}
  \end{subfigure}
      \hspace{3em}
      \begin{subfigure}{0.45\textwidth}
    \includegraphics[scale=1]{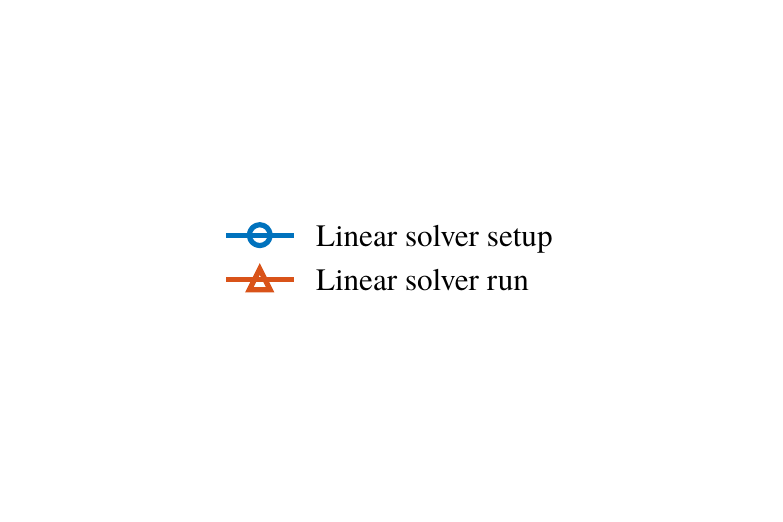}
  \end{subfigure}    

\caption{Influence of standard or aggregated finite element spaces on the linear solver behavior. The figure on the top shows the conjugate gradient (CG) iterations versus problem size for three different loads per processor. At the bottom, we show the wall clock time for the two main linear solver phases (setup and run) versus problem size for aggregated finite element spaces and load 3. Figures reproduced from \cite{Verdugo2019}.}
\label{fig:cg-iter-agg-vs-std}
\end{figure}

\subsection{Ghost-penalty stabilization}\label{sec:ghost}

In this section, we present common techniques for the weak stabilization of immersed finite element methods, often referred to as the \emph{ghost penalty} and generally associated with the \emph{Cut Finite Element Method} (CutFEM) \cite{Burman2010,BurmanHansbo2012,Burman2015}. In the ghost-penalty approach, additional terms are added to the bilinear form which provide control over the solution on the elements that intersect the boundary, while the original approximation space $\Vh$ is unmodified. These terms are denoted by $s_h(v_h,w_h)$ and provide a contribution to the operator norm in the form of $\| v_h \|^2_{s_h} = s_h(v_h,v_h)$. The underlying idea is to control the variation of the approximation function across neighboring elements, such that the solution on small cut elements is controlled by the solution on interior elements. A standard way to achieve this is by controlling the jump in the normal derivatives across faces, by adding appropriately scaled stiffnesses to these jumps. Another approach is to add a volumetric penalization to the difference between a function itself and the extension of the function on a neighboring element. In the subsequent sections, both approaches are discussed.

\clearpage

\begin{figure}[htpb]
    \centering
    \includegraphics[width=0.35\textwidth]{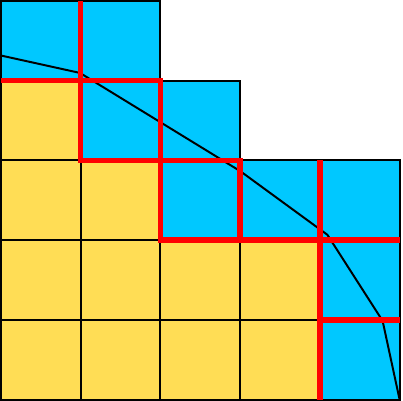}
    \caption{Set of ghost-penalty faces $\mcFhgh$ marked in red. These faces are interior faces of the active mesh $\mcTh$ (yellow and blue elements) and belong to at least one cut element in $\mcThcut$ (blue elements).}
    \label{fig:agg-4}
\end{figure}

\subsubsection*{Face-based stabilization} 
The most common stabilization term for cut elements is the face-based ghost penalty~\cite{BurmanHansbo2012,Burman2015}, defined as
\begin{align}\label{eq:shface}
s_h(w_h,v_h) = \sum_{F \in \mcFhgh}  \sum_{j=1}^{p} \tau_j h^{2j - 1}  (\llbracket\partial_n^j w_h \rrbracket,\llbracket\partial_n^j v_h \rrbracket)_F
\end{align}
where $\tau_j$ are positive constants, $\mcFhgh$ is the set of all interior faces that are shared by an element in $\mcThcut$ (see Figure~\ref{fig:agg-4}), $(\cdot,\cdot)_F$ indicates the inner product over $F$, and 
\begin{equation}
\llbracket\partial_n v_h \rrbracket = \partial_{n_1} v_{h,1} + \partial_{n_2} v_{h,2}
\label{eq:jumpdef}\end{equation}
is the jump in the normal gradient across the face $F$. In this definition of the jump on face $F$, $n_1$ and $v_{h,1}$, and $n_2$ and $v_{h,2}$, correspond to the exterior unit normal and the function $v_h$ on the elements $T_1$ and $T_2$ that share face $F$, respectively. Note that, therefore, $n_1 = -n_2$. With this penalty on the jumps in the normal derivatives between two neighboring elements, the norm of a function $v_h \in \Vh$ on one of the elements can be controlled by the norm on the other. Accordingly, gradients on cut elements are controlled by gradients on full elements in the interior.

The parameters $\tau_j$ in the stabilization terms must be chosen in an appropriate way. With too small parameters, the effect is negligible, while with too large parameters, degrees of freedom can be constrained too strongly, potentially degenerating the accuracy of the approximation \cite{Dettmer2016}. 
On Dirichlet boundaries, the required value to provide coercivity of the bilinear form correlates inversely with the Nitsche parameter.
On Neumann boundaries, coercivity of the bilinear form does not depend on the stabilization term.
Nevertheless, stabilization can still be applied to the Neumann boundary in order to repair the conditioning problems discussed in Section~\ref{sec:conditioningAnalysis}, for which a scaling with $h^{2j + 1}$ instead of $h^{2j - 1}$ in \eqref{eq:shface} suffices \cite{Hansbo2017}.

To elaborate the ghost-penalty approach, let us for simplicity consider the case of linear elements, \emph{i.e.}, $p=1$. It then holds that
\begin{align}\label{eq:face-pair}
\| \nabla v_h \|^2_{T_1} \lesssim \| \nabla v_h \|^2_{T_2} + h \| \llbracket \partial_n v_h \rrbracket \|^2_F, \qquad
\| v_h \|^2_{T_1} \lesssim \| v_h \|^2_{T_2} + h^3 \| \llbracket \partial_n v_h \rrbracket \|^2_F, \qquad v_h \in \Vh
\end{align}
To derive these estimates, we consider the element-wise restrictions $v_{h,i}=v_h|_{T_i}$ ($i\in\{1,2\}$), both extended onto $T_1 \cup T_2$; see Figure~\ref{fig:ghost-elem}. The second estimate in~(\ref{eq:face-pair}) follows from
\begin{equation}\begin{aligned}
\| v_{h,1} \|^2_{T_1} &\lesssim \| v_{h,1} - v_{h,2} \|^2_{T_1} + \| v_{h,2} \|^2_{T_1} 
\lesssim \| (\boldsymbol{\xi} - \boldsymbol{\xi}_F) \cdot \nabla (v_{h,1}- v_{h,2}) \|^2_{T_1} + \| v_{h,2} \|^2_{T_2} 
\\
& \lesssim h^3 \|\partial_{n_i} (v_{h,1}- v_{h,2})\|^2_F+ \| v_{h,2} \|^2_{T_2}  
= h^3 \|\llbracket\partial_n v\rrbracket \|^2_F+ \| v_{h,2} \|^2_{T_2}  
\end{aligned}\label{eq:face_derivation}\end{equation}
with $\partial_{n_i}$ indicating either of the normal derivatives. Here we first added and subtracted $v_{h,2}$ and used the triangle inequality. We then used the Taylor expansion $v_{h,1} - v_{h,2} = (\boldsymbol{\xi} - \boldsymbol{\xi}_F) \cdot \nabla (v_{h,1}- v_{h,2})$ at $\boldsymbol{\xi}$ with point $\boldsymbol{\xi}_F \in F$ the projection of $\boldsymbol{\xi}$ along the normal direction, which holds since $v_{h,1} - v_{h,2} = 0$ on~$F$ by virtue of the continuity of~$v_h$. Additionally, we used the estimate $\|v_{h,2} \|_{T_1} \lesssim \| v_{h,2} \|_{T_2}$ which holds since $v_{h,2}$ is polynomial. The first estimate follows in the same way, but is simpler since the gradients are piecewise constant, leading to an $h$ scaling factor instead of the $h^3$ factor. 

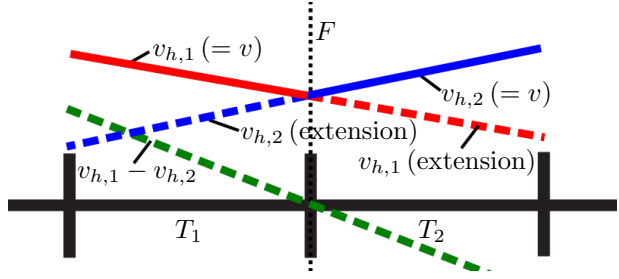
\begin{figure}
     \centering
    \begin{tikzpicture}
        \node[anchor=south west] at (0mm,0mm) {\includegraphics[width=80mm]{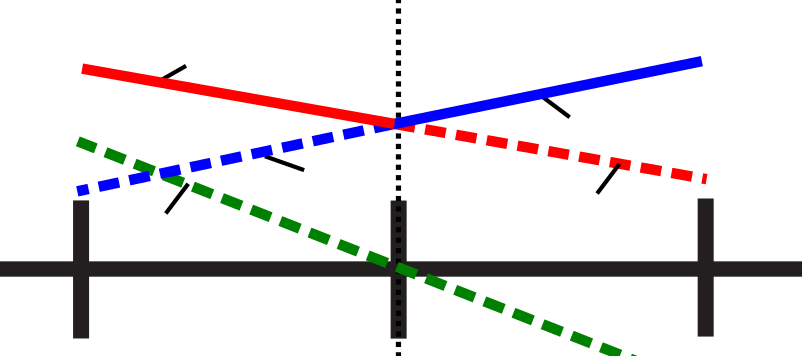}};
        \node at (25mm,6.5mm) {$T_1$};
        \node at (57mm,6.5mm) {$T_2$};
        \node at (27.5mm,30.5mm) {$v_{h,1} \, (=v)$};
        \node at (18.2mm,14mm) {$v_{h,1}-v_{h,2}$};
        \node at (43.5mm,19.5mm) {$v_{h,2} \, (\text{extension})$};
        \node at (65.5mm,24.5mm) {$v_{h,2} \, (=v)$};
        \node at (59mm,15.5mm) {$v_{h,1} \, (\text{extension})$};
        \node at (43mm,33mm) {$F$};
    \end{tikzpicture}      
\caption{Illustration of the extension and difference of $v_{h,1}$ and $v_{h,2}$ in \eqref{eq:face_derivation}.}
\label{fig:ghost-elem}
\end{figure}

Using the pairwise bounds in~(\ref{eq:face-pair}) -- either directly or through a chain of neighboring elements -- the degrees of freedom on cut elements in $\mcThcut$ can be controlled by the degrees of freedom on interior elements $\mcThin$ via the ghost penalty
\begin{equation}
\label{eq:localbnds_qz1}
\| \nabla v_h \|^2_{\Omega_h} \lesssim \| \nabla v_h \|^2_{\Omega} + s_h(v_h,v_h), \qquad
\| v_h \|^2_{\Omega_h} \lesssim \| v_h \|^2_{\Omega} + h^2 s_h(v_h,v_h), \qquad v_h \in \Vh
\end{equation}
To prove the first bound we proceed as
\begin{equation}\begin{aligned}
\| \nabla v_h \|^2_{\Omega_h} & = \sum_{T \in \mcTh} \| \nabla v_h \|^2_{T} \\ 
& = \sum_{T \in \mcThin} \| \nabla v_h \|^2_{T} + \sum_{T \in \mcThcut} \| \nabla v_h \|^2_{T} 
\\ 
& \lesssim \sum_{T \in \mcThin} \| \nabla v_h \|^2_{T} + \Big( \sum_{T \in \mcThin} \| \nabla v_h \|^2_{T} + s_h(v_h,v_h) \Big)
\\ 
& \lesssim \| \nabla v_h \|^2_\Omega + s_h(v_h,v_h) \label{eq:stab-energy-control}
\end{aligned}\end{equation}
The second bound in~(\ref{eq:localbnds_qz1}) can be established similarly. Note that the constants in these estimates will depend on the lengths of the paths connecting cut elements to interior elements, which is related to the properties and the resolution of the boundary, as well as the properties of the mesh. For smooth boundaries and locally quasi-uniform meshes with sufficiently small mesh sizes, one can show that the length of these paths is uniformly bounded.

Conversely, the stabilization term can be bounded in terms of the $L^2$\nobreakdash-norm of the gradient of the approximation on the active domain~$\Omega_h$
\begin{align}\label{eq:ghost-inverse}
s_h(v_h,v_h) \lesssim \| \nabla v_h \|_{\Omega_h} \qquad v_h \in \Vh
\end{align}
which follows from
\begin{align}
h^{2 j - 1} \| \llbracket \partial^j_n v_h \rrbracket \|^2_F 
\lesssim 
\sum_{i=1}^2
h^{2 j - 1} \| \partial^j_n v_{h,i} \|^2_F 
\lesssim 
\sum_{i=1}^2
h^{2 j - 2} \| \nabla^j v_{h,i} \|^2_{T_i}
\lesssim 
\sum_{i=1}^2
 \| \nabla v_{h,i} \|^2_{T_i} = \| \nabla v_h \|_{T_1 \cup T_2}^2
\end{align}
where we first employed the triangle inequality, then a trace inverse estimate to pass from the face $F$ to the neighboring elements $T_1$ and $T_2$ \cite{Warburton2003}, and finally another inverse estimate to arrive at a bound in terms of~$\nabla v_h$~\cite{ErnGuermond}. Using the same inverse estimate it follows that 
\begin{equation}\label{eq:ghost-inverse-l2}
h^2 \| v_h \|^2_{s_h} \lesssim h^2 \| \nabla v_h \|^2_{\Omega_h} \lesssim \| v_h \|^2_{\Omega_h}
\end{equation}

The estimates~\eqref{eq:localbnds_qz1}, \eqref{eq:ghost-inverse}, and \eqref{eq:ghost-inverse-l2} are central to the coercivity and continuity of the (stabilized) bilinear form with a (global) Nitsche parameter that scales inversely with the mesh size of the \emph{background} grid. The estimates are also essential for the condition number bounds derived in Section~\ref{sec:stabilityAnalysis}. The principal notion of the ghost-penalty method is encapsulated in the bounds in~\eqref{eq:localbnds_qz1}, which convey that the ghost penalty provides control over the solution on the entire active domain including the fictitious part. The bounds~\eqref{eq:ghost-inverse} and \eqref{eq:ghost-inverse-l2} in turn impart that the ghost-penalty term is suitably well behaved.


\subsubsection*{Element- and patch-based stabilization forms}


An alternative to the face-based penalty is a penalty that is based on volumetric integrals. This is convenient for higher-order polynomials, since it avoids the computation of higher-order derivatives at element interfaces, which is a non-standard operation in many finite element codes. An overview and analysis of different types of such ghost-penalty formulations is presented in \cite{Burman2022Weak}.

A stabilization term based on element integrals can be defined according to any of the two forms
\begin{align}\label{eq:shelem}
s_{h,0}^{\mathrm{elem}}(w_h,v_h) = \sum_{F \in \mcFhgh} \tau_0 h^{-2} ([ w_h ],[ v_h ])_{\mcTh(F)}, \quad \textrm{or} \quad
s_{h,1}^{\mathrm{elem}}(w_h,v_h) = \sum_{F \in \mcFhgh} \tau_1 ([\nabla w_h ],[\nabla v_h ])_{\mcTh(F)}
\end{align}
where for each face $F\in \mcFhgh$, $\mcTh(F)$ is the pair of elements that share $F$ and
\begin{align}\label{eq:shpatch}
[ v_h ] = v_{h,1} - v_{h,2}, \qquad [\nabla v_h ] = \nabla v_{h,1} - \nabla v_{h,2}
\end{align}
on $T \in \mcTh(F) = \{ T_1, T_2\}$ with polynomials $v_{h,1}$ and $v_{h,2}$ extended in the canonical way; see Figure~\ref{fig:ghost-elem}. Note that, opposed to the existing volumetric integrals in the weak formulation introduced in Section~\ref{sec:method_formulation}, the integrals in the penalization are taken over the entire background elements including the fictitious parts and are not restricted to the physical domain $\Omega$. The form in \eqref{eq:shelem} is presented in \cite{Preuss2018,Lehrenfeld2019} and is obtained by example formulation 2 in \cite[\S 2.4]{Burman2022Weak} in case the element pairs are chosen based on the faces in $\mcFhgh$.

We may also consider stabilization based on patches or macro elements, which is the form that was presented with the introduction of the ghost penalty in \cite{Burman2010} and, for a certain choice of patches, is also considered in \cite{Lehrenfeld2019}. With this form of stabilization, each cut element is associated with a patch of neighboring elements that contains at least one element in the interior. We may then penalize the difference between the finite element function and a global polynomial on the patch. More precisely, let $M_T$ be a patch of elements containing $T$, then we can define
\begin{equation}\begin{aligned}
s_{h,0}^{\mathrm{patch}}(w_h,v_h) & = \sum_{T\in \mcThcut} \tau_0 h^{-2}( w_h - P_p w_h, v_h - P_p v_h )_{M_T} \\
s_{h,1}^{\mathrm{patch}}(w_h,v_h) & = \sum_{T\in \mcThcut} \tau_1 (\nabla (w_h - P_p w_h), \nabla (v_h - P_p v_h) )_{M_T}
\end{aligned}\label{eq:patchbased}\end{equation}
where $P_p: V_h|_{M_T} \rightarrow \mathcal{P}_p(M_T)$ is the $L^2$-projector or the Ritz projector, respectively.
Note that by strictly following this formulation, patches containing multiple cut elements will be contained in the summation multiple times. In practice, every patch that contains cut elements is included only once, or the boundary is simply covered by a set of non-overlapping patches such as in \cite{Larson2021}.

Ghost-penalty stabilization can also be combined with the concept of aggregation to design weakly aggregated schemes, as proposed in \cite{Burman2022Weak,badia2021linking}. In these schemes, one makes use of the standard finite element space $\Vh$ on the active mesh, and instead of strongly imposing constraints via an extension operator of interior degrees of freedom, the difference between $\Vh$ and $\Vhag$ is weakly penalized. Recall the discrete extension operator $\mathcal{E}_{h}^{\mathrm{ag}}: V_{h}^{\mathrm{in}} \rightarrow V_{h}^{\mathrm{ag}} \in V_h$ defined in Section~\ref{sec:aggregation} to define a mapping $\Phag:\Vh \ni v_h \mapsto \mathcal{E}_{h}^{\mathrm{ag}} (v_h|_{\mcThin}) \in \Vhag$. The difference $v_h - \Phag v_h$ can then be penalized similarly as in~\eqref{eq:patchbased}

\clearpage

\begin{equation}\begin{aligned}
s_h^{\mathrm{ag}}(w_h,v_h) &= \sum_{\mcTh}  \tau_0 h^{-2} \left( w_h - \Phag w_h  , v_h - \Phag v_h \right)_T\\
s_h^{\mathrm{ag}}(w_h,v_h) &= \sum_{\mcTh}  \tau_1 \left( \nabla (w_h - \Phag w_h)  ,\nabla( v_h - \Phag v_h )\right)_T
\end{aligned}\end{equation}%
These terms can be computed element-wise (in a similar manner as described in Section~\ref{sec:aggregation}) and are similar to local-projection stabilization techniques; see, \emph{e.g.}, \cite{Becker2001Dec,Badia2012Nov}. We note that the (strong) aggregated finite element method of Section~\ref{sec:aggregation}
is formally recovered in the limit $\tau_j \to \infty$. Thus, the stabilization method is also accurate for large values of $\tau_j$. This last statement does not hold for other forms of the ghost penalty, in which a large parameter can constrain too many degrees of freedom, which inhibits the accuracy \cite{Dettmer2016} and is referred to as locking in \cite{Burman2022Weak}. Detailed analyses and experimentation with this type of stabilization can be found in \cite{badia2021linking} and \cite{Burman2022Weak}.

Stabilization techniques based on volumetric integrals have also been developed for multimesh discretizations \cite{Johansson2020,Johansson2019,Johansson2020CMAME} and for problems on manifolds \cite{Burman2018codim}. Additionally, it is noteworthy that with the $H^1$-product and with the projector $P_p = 0$ or $\Phag = 0$, the fictitious domain stiffness that is commonly employed in the finite cell method is recovered~\cite{Nguyen2017}. Fictitious domain stiffness requires $\tau_1 \ll 1$, however, as this stabilization term is not consistent with the original problem. Therefore, with a fictitious domain stiffness, stability and conditioning need to be balanced with consistency.

\subsubsection*{Implementation aspects}
To implement any of the aforementioned stabilization forms, one must be able to find face neighbors of given elements and assemble contributions from pairs of elements sharing a face. For higher-order elements, the face-based ghost penalty requires computation of higher-order derivatives. The element-based formulation may then be preferable. All the stabilization terms lead to additional coupling and, as a consequence, fill-in in the stiffness matrix \cite{Hoang2018,Hoang2019}. Systems stabilized by the ghost penalty have similar condition number bounds as boundary-fitted formulations, and therefore larger systems (\emph{i.e.}, finer grids) yield larger condition numbers and require more iterations in iterative solvers. This can be overcome by combining the ghost penalty with a multigrid solver; see the discussion on this in Section~\ref{sec:preconditioners}.

\subsubsection*{Extension to other problems}

The ghost penalty can also be applied to stabilize time-dependent problems, for which generally the mass form needs to be stabilized as well. In that case the scaling is chosen in such a way that the stabilization terms scale with the mesh parameter in the same way as the mass matrix; see \eqref{eq:ghost-inverse-l2}. Similar to Schwarz preconditioning and aggregation, the ghost penalty can also be applied to provide inf-sup stable discretization of mixed formulations such as that of the Stokes problem; see, \emph{e.g.}, \cite{Massing2014,Hoang2017,Guzman2018Sep,Hoang2019}.
Furthermore, the ghost penalty can also be applied to stabilize discretizations of codimensional manifolds,
such as partial differential equations posed on non-planar surfaces discretized with three-dimensional meshes; see, \emph{e.g.}, \cite{Hansbo2015,Burman2015Beltrami,Burman2016Beltrami,Burman2018codim}.

\subsection{A unified analysis of aggregation and stabilization}\label{sec:stabilityAnalysis}

In this section, we present a framework for the analysis of stable immersed finite element methods. We consider aggregated and ghost-penalty-stabilized techniques at the same time, to identify the common parts and the main differences.
It is recalled from Section~\ref{sec:introSec4} that when Nitsche's method is applied without any stability-enhancing technique, the bilinear form is only coercive with an \emph{a priori} unbounded Nitsche parameter $\beta$. In that case, most of the properties and estimates presented in this section do not apply, and error bounds cannot be provided.
Remark~\ref{rem:comparison} at the end of this section discusses which specific analyses still hold and which no longer hold without any form of stabilization.

\clearpage

\subsubsection*{Methods and formulations}
We recall from~\eqref{eq:immersedfem} and \eqref{eq:immersedOperator} that the standard Nitsche finite element method takes the form
\begin{align}\label{eq:standardnitsche}
\begin{cases}
\text{find } u_h \in \Vh \text{ such that:}\\
a_h(u_h, v_h) = l_h(v_h) \qquad \forall v_h \in \Vh
\end{cases}
\end{align}
with bilinear form $a_h(\cdot,\cdot)$ according to
\begin{align}
a_h(w_h,v_h) = (\nabla w_h,\nabla v_h)_\Omega - (w_h, \partial_n v_h)_{\partial \Omega_D} - (\partial_n w_h, v_h)_{\partial \Omega_D} 
+ (\beta w_h,v_h)_{\partial \Omega_D}
\end{align}
We note that this is consistent with the strong formulation in \eqref{eq:poissonstrong}, \emph{i.e.}, the exact solution $u$ to 
the corresponding continuous problem satisfies \begin{align}\label{eq:consistency}
a_h(u,v_h) = l_h(v_h) \qquad v_h \in \Vh
\end{align}
provided that~$u$ is regular enough, for instance if $u \in H^2(\Omega)$.
To accommodate both approaches in the same analysis we consider the formulation
\begin{align}\label{eq:stab-nitsche}
\begin{cases}
\text{find } u_h \in V_{h,\bigstar} \text{ such that:}\\
a_{h,\bigstar}(u_h, v_h) = l_h(v_h) \qquad \forall v_h \in V_{h,\bigstar}
\end{cases}
\end{align}
where the finite dimensional function space is defined by
\begin{align}
V_{h,\bigstar} = 
\begin{cases}
\Vhag \subset \Vh & \text{(aggregation)}
\\
\Vh & \text{(ghost-penalty stabilization)}
\end{cases} 
\end{align}
and the bilinear form is defined by
\begin{equation}
a_{h,\bigstar}(w_h,v_h) = a_\Omega(w_h,v_h) - (w_h,\partial_n v_h)_{\partial \Omega_D} - (\partial_n w_h,v_h)_{\partial \Omega_D} + (\beta w_h,v_h)_{\partial \Omega_D}
\end{equation}
with
\begin{align}
a_\Omega(w_h,v_h) = 
\begin{cases}
 (\nabla w_h, \nabla v_h)_\Omega   & \text{(aggregation)}
\\
(\nabla w_h, \nabla v_h)_\Omega  + s_h(w_h,v_h) & \text{(ghost-penalty stabilization)}
\end{cases}
\end{align}
For the ghost-penalty method, the stabilization form is included in the definition of $a_\Omega(\cdot,\cdot)$ because the stabilization form $s_h(\cdot,\cdot)$ is naturally paired with the bulk or volumetric term $(\nabla \cdot,\nabla \cdot)_\Omega$ to enable bounds of the form \eqref{eq:face-pair} and \eqref{eq:localbnds_qz1}.

\subsubsection*{Properties}
For the analysis, we start from the following intrinsic properties of aggregation methods and ghost-penalty formulations:
\begin{itemize}
\item There is a constant (hidden in the binary operator~$\lesssim$) such that 
\begin{align}\label{eq:stab}
\| \nabla v_h \|_{\Omega_h} \lesssim a_\Omega(v_h,v_h) = \| v_h \|_{a_\Omega} \qquad v_h \in V_{h,\bigstar}
\end{align}
where as before $\Omega_h = \cup_{T \in \mcTh} T$ is the union of all the active elements, such that $\Omega_h \supset \Omega$. 
See, respectively,  \eqref{eq:agg-h1-stab} and \eqref{eq:stab-energy-control} for these estimates in the context of aggregated and weakly stabilized methods. This property forms the basis of the effectiveness of weakly stabilized and aggregated methods. Essentially, it indicates that the solution on the extended domain $\Omega_h$, including the fictitious parts outside $\Omega$ of active elements, is controlled by $a_\Omega(\cdot,\cdot)$. Consequently, this precludes the very small eigenvalues on small cut elements discussed in Section~\ref{sec:conditioningAnalysis}, and it provides control of the normal gradients which enables bounding the Nitsche parameter $\beta$ from above as discussed in Section~\ref{sec:introSec4} and \ref{sec:stabilityLiterature}. With aggregation this is achieved by directly precluding basis functions with a very small support in the physical domain as discussed in Section~\ref{sec:aggregation}, while with the ghost penalty such functions are given a contribution in $a_\Omega(\cdot,\cdot)$ by $s_h(\cdot,\cdot)$ as discussed in Section~\ref{sec:ghost}.

\item There is an interpolation operator $\pi_h: H^s(\Omega) \rightarrow V_{h,\bigstar}$, $s\geq 2$, such that  
\begin{align}\label{eq:approx}
\| \nabla^m (v - \pi_h v) \|_{\Omega_h} \lesssim h^{s-m} \| v \|_{H^s(\Omega)} \qquad 0\leq m \leq s \leq p+1
\end{align}
we use a continuous extension operator $\mathcal{E}: v \in H^s(\Omega) \rightarrow H^s(\mathbb{R}^d)$ to extend $v$ outside of $\Omega$. For ghost-penalty-stabilized methods, the interpolation operator is generally constructed by first extending the function $v$ outside of $\Omega$ and then applying a standard interpolation operator, for instance the Cl\'{e}ment operator \cite{Clement1975}. For the aggregated finite element space, a specific construction of the extension is used to establish the interpolation estimate \cite{Burman2022Extension,Burman2022Splines}.

\item The stabilization form $s_h(\cdot,\cdot)$ satisfies
\begin{align}\label{eq:approxsh}
\| \pi_h v \|_{s_h}^2 = s_h(\pi_h v,\pi_h v) \lesssim h^{s-1} \| v \|_{H^s(\Omega)} \qquad 1\leq s \leq p+1
\end{align}
This is a consistency assumption on the stabilization form $s_h(\cdot,\cdot)$, formulated in such a way that the stabilization form only acts directly on functions in the finite element space, and acts on other functions~$v$ indirectly via their projection~$\pi_h{}v$. Note that the stabilization form~$s_h(\cdot,\cdot)$ may for instance involve higher-order derivatives, which are not generally well defined for the exact solution or the solution to a dual problem used in an error analysis. This complication is avoided by formulating the consistency condition~\eqref{eq:approxsh} in terms of projections onto the finite element space.

To provide an example we verify this consistency assumption \eqref{eq:approxsh} for the face-based 
stabilization form \eqref{eq:shface} in the case of linear elements. Consider a face $F \in \mcFh$  shared by the two elements $T_1$ and $T_2$. Then for any linear polynomial
$w \in \mathcal{P}_1(T_1 \cup T_2)$ we have
\begin{equation}
h \| [ \partial_n \pi_h v  ] \|^2_F 
= h \| [ \partial_n (\pi_h v - w)  ] \|^2_F
\lesssim \sum_{i=1}^2 \|  \nabla (\pi_h v - w)  \|^2_{T_i} 
\end{equation}
where we used a trace inverse estimate to pass from the face to the elements \cite{Warburton2003}. Using the Bramble-Hilbert lemma (see, \emph{e.g.}, \cite{BrennerScott2008}) we can choose $w$ in such a way that 
\begin{align}
\sum_{i=1}^2  \|  \nabla (\pi_h v - w)  \|^2_{T_i} 
 \leq 
 \sum_{i=1}^2 
 \|  \nabla (\pi_h v - v)  \|^2_{T_i} + \|  \nabla (v - w)  \|^2_{T_i} 
 \lesssim
\sum_{i=1}^2 h^2 \| v \|^2_{H^2(S(T_i))}
\end{align}
where $S(T_i)$ is the union of the set of elements in $\mcTh$ that share a node with $T$.

\end{itemize}
Element aggregation and ghost-penalty stabilization are conceptually different in the manner in which they approach the conditions~\eqref{eq:stab}\nobreakdash--\eqref{eq:approxsh}. Aggregation methods aim to satisfy condition~\eqref{eq:stab} by restricting the approximation to a suitable subspace $\Vhag\subset V_h$. On the other hand, condition~\eqref{eq:approx} demands optimal approximation properties of the (sequence of) aggregated approximation spaces~$\Vhag$ and, hence, the aggregated approximation spaces cannot be too small. Because the stabilization operator $s_h(\cdot,\cdot)$ vanishes in aggregation methods, condition~\eqref{eq:approxsh} is trivially satisfied.
Ghost-penalty stabilization requires that the stabilization term $s_h(\cdot,\cdot)$ is strong enough to comply with condition~\eqref{eq:stab}. On the other hand, condition~\eqref{eq:approxsh} insists that the stabilization cannot be too strong, as this will affect the solution. Because ghost-penalty stabilization imposes no auxiliary conditions on the approximation spaces, condition~\eqref{eq:approx} is trivially satisfied. 

\subsubsection*{Coercivity and Continuity}

We define the following energy norms
\begin{align}\label{eq:energynorm}
\tn v \tn_h^2 &= \| \nabla v \|^2_{\Omega} +  \| h^{-\frac{1}{2}} v \|^2_{\partial \Omega_D} + \| h^{\frac{1}{2}} \partial_n v \|^2_{\partial \Omega_D} 
\\
\tn v_h \tn_{h,\bigstar}^2 &= a_\Omega(v_h,v_h) + \| h^{-\frac{1}{2}} v_h \|^2_{\partial \Omega_D} + \| h^{\frac{1}{2}} \partial_n v_h \|^2_{\partial \Omega_D} 
\end{align}
and note that
\begin{align}
\tn v_h \tn^2_{h,\bigstar} =
\begin{cases}
\tn v_h \tn^2_h & \text{aggregation}
\\
\tn v_h \tn^2_h + \| v_h \|_{s_h}^2 & \text{weak stabilization}
\end{cases}
\end{align}
Note that these norms are different from the $\beta$-norm that was applied in \cite{Prenter2018} and presented in \eqref{eq:immersedenergynorm} in Section~\ref{sec:introSec4}. First, $\tn \cdot \tn_{h,\bigstar}$ employs $h^{-1}$, which is uniformly bounded, instead of $\beta$, which is not bounded in case no stabilization technique is applied. Second, in case of ghost-penalty stabilization, $\tn \cdot \tn_{h,\bigstar}$ also includes a ghost-penalty contribution. Consequently, whereas the $\beta$-norm admits a natural analysis of unstabilized forms, the $\tn v_h \tn_{h,\bigstar}$-norm naturally accompanies weakly stabilized and aggregated formulations. 

\begin{lem}
The form $a_{h,\bigstar}(\cdot,\cdot)$ is continuous on $V_{h,\bigstar} \oplus H^2(\Omega)$
\begin{align}\label{eq:cont}
a_{h,\bigstar}(w,v) \leq \tn w  \tn_{h,\bigstar}  \tn v \tn_{h,\bigstar}  \qquad w,v \in V_{h,\bigstar} \oplus H^2(\Omega)
\end{align}
and, if $\beta$ is chosen large enough, $a_{h,\bigstar}(\cdot,\cdot)$ is coercive on $V_{h,\bigstar}$
\begin{align}\label{eq:coer}
\tn v_h \tn_{h,\bigstar}^2 \lesssim a_{h,\bigstar}( v_h,v_h) \qquad v_h \in V_{h,\bigstar}
\end{align}
\end{lem}
\begin{proof} Under the notion that with aggregation or weak stabilization $\beta \lesssim h^{-1}$, the continuity follows directly from the Cauchy-Schwarz inequality together with the observation that $\partial_n v \in L^2(\partial \Omega_D)$ for $v \in V_{h,\bigstar} \oplus H^2(\Omega)$ and therefore $\| \partial_n v \|_{\partial \Omega_D}$ is well defined. For the coercivity we use the element-wise trace inverse inequality \cite{Warburton2003}
\begin{equation}
h \| \partial_n v \|^2_{T \cap \partial \Omega_D} \lesssim \| \nabla v \|^2_T \qquad v \in \mathcal{P}_p(T)
\end{equation}
and \eqref{eq:stab} to conclude that 
\begin{align}\label{eq:coer-b}
\| h^{\frac{1}{2}} \partial_n v_h \|^2_{\partial \Omega_D} 
\lesssim
\sum_{T \in \mcThcut} \| \nabla v_h \|^2_T  \lesssim 
\sum_{T \in \mcTh} \| \nabla v_h \|^2_T  \lesssim 
\| \nabla v_h \|^2_{\Omega_h} \lesssim \| v_h \|^2_{a_\Omega} \qquad v_h \in V_{h,\bigstar}
\end{align}
where $\mcThcut$ is the set of elements that intersect the boundary.
Next, for any $\delta>0$, we have
\begin{align}
2(\partial_n v, v)_{\partial \Omega_D} 
\leq 2 \| \partial_n v \|_{\partial \Omega_D} \| v \|_{\partial \Omega_D} 
\leq \| (h\delta)^{\frac{1}{2}} \partial_n v \|^2_{\partial \Omega_D} + \| (h\delta)^{-\frac{1}{2}} v \|^2_{\partial \Omega_D} 
\end{align}
which leads to
\begin{equation}
\begin{aligned}
a_{h,\bigstar}(v_h,v_h) &= \| v_h \|^2_{a_\Omega} - 2 (\partial_n v_h, v_h)_{\partial \Omega_D} + \| \beta^{\frac{1}{2}} v_h \|^2_{\partial \Omega_D}
\\
&\geq  \| v_h \|^2_{a_\Omega} - \| (h\delta)^{\frac{1}{2}} \partial_n v_h\|^2_{\partial \Omega_D}  
- \| (h\delta)^{-\frac{1}{2}} v_h \|^2 _{\partial \Omega_D} + \| \beta^{\frac{1}{2}} v \|^2_{\partial \Omega_D}
\\
&\geq  ( 1 - \delta C )  \| v_h \|^2_{a_\Omega} 
+ \| (\beta - (h\delta)^{-1})^{\frac{1}{2}} v_h \|^2_{\partial \Omega_D}  \qquad v_h \in V_{h,\bigstar}
\end{aligned}
\end{equation} 
where $C$ is the hidden constant between the first and final expressions in~\eqref{eq:coer-b}. Taking $\delta$ small enough to guarantee that $\delta C \lesssim 1/2$ and $\beta$ large enough such that $\beta - (h\delta)^{-1} > h^{-1}/2$ on each element in $\mcThcut$, the coercivity follows.
\end{proof}

Assuming that $l_h(v_h)$ is continuous, \emph{i.e.}, $l_h(v_h) \lesssim \tn v_h \tn_{h,\bigstar}$ for $v_h \in V_{h,\bigstar}$, we may now apply the Lax--Milgram lemma to show that there is a unique and stable solution $u_h \in V_{h,\bigstar}$ to~\eqref{eq:stab-nitsche}.

\subsubsection*{Error Estimates} 
Conditions~\eqref{eq:stab}--\eqref{eq:approxsh} ensure that the corresponding immersed finite-element approximation according to~\eqref{eq:standardnitsche} has optimal approximation properties in the energy norm.
\begin{thm}
Consider the immersed finite-element approximation~\eqref{eq:standardnitsche}. Assume that conditions~\eqref{eq:stab}--\eqref{eq:approxsh} hold. The approximate solution~$u_h$ satisfies the energy-norm error estimate
\begin{equation}
\tn u - u_h \tn_h \lesssim h^p \| u \|_{H^{p+1}(\Omega)} 
\label{eq:opt_bound}
\end{equation}
\end{thm}
\begin{proof}
We start by splitting the error in two parts using the interpolation operator
\begin{align}\label{eq:errorsplit}
\tn u - u_h \tn_{h} \lesssim  \tn u - \pi_h u \tn_{h}  
+ \tn \pi_h u - u_h \tn_h 
\end{align}
We first consider the first term and show that
\begin{equation}\label{eq:interpol-energy}
 \tn u - \pi_h u \tn_h \lesssim h^p \| u \|_{H^{p+1}(\Omega)}
\end{equation}
To establish~\eqref{eq:interpol-energy}, we note that by virtue of a standard trace inequality (see, \emph{e.g.}, \cite{BrennerScott2008})  
\begin{equation}
\| v \|^2_{T \cap \partial \Omega_D} \lesssim h^{-1} \| v \|^2_T  + h \| \nabla v \|^2_T \qquad v \in H^1(T)
\end{equation}
it holds that 
\begin{equation}\label{eq:part1bound}
\begin{aligned}
\tn v \tn^2_h &= \| \nabla v \|^2_{\Omega}+ \| h^{-\frac{1}{2}} v \|^2_{\partial \Omega_D} + \| h^{\frac{1}{2}} \partial_n v \|^2_{\partial \Omega_D} 
\\
& \lesssim \| \nabla v \|^2_{\Omega_h} + \| h^{-1} v \|^2_{\Omega_h} + \| h \nabla^2 v \|^2_{\Omega_h}
\end{aligned}
\end{equation}
for $v \in V_{h,\bigstar} + H^2(\Omega)$.
Setting $v = u -\pi_h u$ and using (\ref{eq:approx}), we obtain (\ref{eq:interpol-energy}). Equation \eqref{eq:part1bound} applies the definition of $\tn \cdot \tn^2_h$ given in \eqref{eq:energynorm}.

Considering the second term in (\ref{eq:errorsplit}), we note that the definition of $\tn \cdot \tn_{h,\bigstar}$ and the coercivity 
property~(\ref{eq:coer}) imply
\begin{align}\label{eq:energy-d}
\tn  \pi_h  u - u_h \tn_h \lesssim \tn  \pi_h u  - u_h \tn_{h,\bigstar}  \lesssim \sup_{ w_h \in V_{h,\bigstar}\setminus\{0\}}\frac{a_{h,\bigstar}(\pi_h  u - u_h, w_h) }{\tn w_h \tn_{h,\bigstar}}
\end{align}
In the case of the aggregation method, we have $a_{h,\bigstar}(\cdot,\cdot) = a_h(\cdot,\cdot)$ and 
$\tn \cdot \tn_{h,\bigstar} = \tn \cdot \tn_{h}$ and it immediately follows that
\begin{equation}
\begin{aligned}
a_{h,\bigstar}( \pi_h u - u_h,w_h) &= a_h( \pi_h u - u_h ,w_h) 
\\
&= a_h( \pi_h u  ,w_h)  - l_h(w_h) 
\\
&= a_h(\pi_h u - u ,w_h) 
\\
&\lesssim \tn \pi_h u - u \tn_h \tn w_h \tn_h
\\ \label{eq:energy-e}
&\lesssim h^p \| u \|_{H^{p+1}(\Omega)} \tn w_h \tn_h
\end{aligned}
\end{equation}
where we used, consecutively,  the consistency relation (\ref{eq:consistency}) to replace $l_h(w_h)$ by $a_h(u,w_h)$,  
continuity of $a_h$ according to~(\ref{eq:cont}), and the interpolation error estimate~(\ref{eq:interpol-energy}). In the case of ghost-penalty stabilization, it holds that
\begin{equation}
\begin{aligned}
a_{h,\bigstar}( \pi_h u - u_h,w_h) 
&= a_{h,\bigstar}( \pi_h u  ,w_h)  - l_h(w_h) 
\\
&= a_h( \pi_h u  ,w_h)  - l_h(w_h) + s_h(\pi_h u,w_h)  
\\
&= a_h(\pi_h u - u ,w_h) + s_h(\pi_h u, w_h) 
\\
&\lesssim \tn \pi_h u - u \tn_h \tn w_h \tn_h + \| \pi_h u \|_{s_h} \| w_h \|_{s_h}
\\
&\lesssim \Big( \tn \pi_h u - u \tn^2_h  + \| \pi_h u \|^2_{s_h} \Big)^{1/2} \tn w_h \tn_{h,\bigstar}
\\ \label{eq:energy-f}
& \lesssim h^p \| u \|_{H^{p+1}(\Omega)} \tn w_h \tn_{h,\bigstar}
\end{aligned}
\end{equation}
where we used the interpolation error estimate~(\ref{eq:interpol-energy}) and 
the weak consistency condition~(\ref{eq:approxsh}) on $s_h(\cdot,\cdot)$. 
From (\ref{eq:energy-d})--(\ref{eq:energy-f}) it follows that the second term in~\eqref{eq:errorsplit} is also bounded
in accordance with~\eqref{eq:opt_bound}.
\end{proof}

\clearpage

\subsubsection*{Condition number estimate} 
Under certain (non-restrictive) conditions, both aggregation and ghost-penalty stabilization engender stiffness matrices that exhibit similar condition numbers as conventional boundary-fitted finite element methods. To analyze the condition numbers corresponding to both methods, we consider the basis $\{\phi_i\}_{i=1}^N$ in $V_{h,\bigstar}$, so that any $v_h\in{}V_{h,\bigstar}$ can be expanded as
\begin{align}
v_h = \sum_{i=1}^N y_i \phi_i
\end{align}
where $\mathbf{y} \in \mathbb{R}^N$ is the coefficient vector. It is to be noted that with ghost-penalty stabilization the basis $V_{h,\bigstar} = \Vh$ is a standard nodal basis, while with aggregation $V_{h,\bigstar} = \Vhag$ is the basis that is modified at the boundary in such a way that the support of each basis function has a sufficiently large intersection with $\Omega$, more precisely $|\text{supp}(\phi_i) \cap \Omega| \sim h^d$. The stiffness matrix 
$\mathbf{A}\in \mathbb{R}^{N\times N}$ is defined as
\begin{equation}\label{eq:stiffmat}
\mathbf{z}^T \mathbf{A} \mathbf{y} = a_{h,\bigstar} (w_h,v_h) \qquad w_h,v_h \in V_{h,\bigstar}
\end{equation}
and for the considered symmetric bilinear forms, the condition number of the stiffness matrix $\mathbf{A}$ corresponds to
\begin{equation}
\label{eq:kappaA}
\kappa\left( \mathbf{A} \right) = \frac{\lambda_{\max}}{\lambda_{\min}}
\end{equation}
where $\lambda_{\max}$ and $\lambda_{\min}$ are the maximum and minimum eigenvalues 
of $\mathbf{A}$, respectively.  To estimate the condition number we start from the following basic bounds, which hold by construction for both aggregated and weakly stabilized methods:
\begin{itemize}
\item There is a constant (hidden in the binary relation $\lesssim$) such that 
\begin{align}\label{eq:cond-inv}
\tn v_h \tn_{h,\bigstar} \lesssim h^{-1} \| v_h \|_{\Omega_h} \qquad v_h \in V_{h,\bigstar}
\end{align}
Both for weakly stabilized and aggregated methods, this bound follows from standard inverse inequalities; see \eqref{eq:ghost-inverse} for the inverse estimate of the ghost-penalty form.

\item There is a (hidden) constant such that 
\begin{align}\label{eq:cond-poin}
\| v_h \|_{\Omega_h} \lesssim \tn v_h \tn_{h,\bigstar} \qquad v_h \in V_{h,\bigstar}
\end{align}
This is a Poincar\'e inequality, requiring aggregation of the approximation space or ghost-penalty stabilization of the bilinear form to gain control on the approximate solution outside~$\Omega$.

\item There are constants such that the following equivalence holds:
\begin{align}\label{eq:Rneqv}
h^d \| \mathbf{y} \|^2_2 \sim \| v_h \|^2_{\Omega_h} \qquad v_h \in V_{h,\bigstar}
\end{align}
For weakly stabilized methods we have  $V_{h,\bigstar} = \Vh$ and this is a standard estimate that follows 
from the quasi-uniformity of the mesh and the use of a nodal basis. For the aggregated finite 
element space $V_{h,\bigstar} = \Vhag$, this estimate holds by virtue of the construction of the 
extension operator.
\end{itemize}
\begin{thm} Consider the stiffness matrix in~\eqref{eq:stiffmat}. Assume that the bilinear form $a_{h,\bigstar}(\cdot,\cdot)$ is coercive according to~\eqref{eq:coer} and that conditions~\eqref{eq:cond-inv}--\eqref{eq:Rneqv} hold. The condition number~$\kappa\left( \mathbf{A} \right)$ of the stiffness matrix satisfies the following bound
\begin{equation}
\kappa\left( \mathbf{A} \right) \lesssim h^{-2}
\end{equation}
\end{thm} 

\begin{proof} To estimate the condition number according to~\eqref{eq:kappaA}, we use the following Rayleigh quotient characterization of the eigenvalues to bound the maximum (resp. minimum) eigenvalue from above (resp. below):
\begin{align}
\lambda_{\max} = \max_{\mathbf{y} \in \mathbb{R}^N\setminus \{0\}} 
\frac{\mathbf{y}^T \mathbf{A} \mathbf{y}}{ \| \mathbf{y} \|^2_2} , \qquad 
\lambda_{\min} = \min_{\mathbf{y} \in \mathbb{R}^N\setminus \{0\}} 
\frac{\mathbf{y}^T \mathbf{A} \mathbf{y}}{ \| \mathbf{y} \|^2_2} 
\end{align}
We bound the Rayleigh quotient from above as follows
\vspace{.5mm}
\begin{align}
\frac{\mathbf{y}^T \mathbf{A} \mathbf{y}}{ \| \mathbf{y} \|^2_2} 
\sim \frac{a_{h,\bigstar}(v_h,v_h)}{ h^{-d} \| v_h \|^2_{\Omega_h}} 
= h^d \frac{a_{h,\bigstar}(v_h,v_h)}{\tn v_h \tn^2_{h,\bigstar}} \frac{\tn v_h \tn^2_{h,\bigstar}}{\| v_h \|^2_{\Omega_h}} 
\lesssim h^{d-2}
\end{align}
\vspace{.5mm}
where we used, consecutively, the definition (\ref{eq:stiffmat}) of the stiffness matrix,  the equivalence 
(\ref{eq:Rneqv}), and continuity of $a_{h,\bigstar}(\cdot,\cdot)$ according to~(\ref{eq:cont}) and the inverse inequality (\ref{eq:cond-inv}). 
Conversely, to bound the quotient from below, we infer from the equivalence relation~\eqref{eq:Rneqv}, coercivity condition~(\ref{eq:coer}), and the Poincar\'e inequality~(\ref{eq:cond-poin}) that 
\vspace{.5mm}
\begin{align}
\frac{\mathbf{y}^T \mathbf{A} \mathbf{y}}{ \| \mathbf{y} \|^2_2} 
\gtrsim 
\frac{\tn v_h \tn^2_{h,\bigstar}}{ h^{-d} \| v_h \|^2_{\Omega_h}}
\gtrsim 
h^d
\end{align}
\vspace{.5mm}
Combining the upper and lower bounds, we obtain the desired estimate
\vspace{.5mm}
\begin{equation}
\kappa = \frac{\lambda_{\max}}{\lambda_{\min}} \lesssim \frac{h^{d-2}}{h^d} \lesssim h^{-2}
\end{equation}
\end{proof}
\begin{remark}\label{rem:comparison}\textbf{Comparison to Nitsche's method without stabilization.}
We here briefly discuss which aspects of Section~\ref{sec:stabilityAnalysis} do, and do not, apply when Nitsche's method is applied in an immersed setting without an additional stability-enhancing technique. First and foremost, the intrinsic property \eqref{eq:stab} will not be satisfied. Consequently, the normal gradients on the boundary are not sufficiently controlled by the gradients in the bulk, such that coercivity only holds with locally very large (\emph{i.e.}, \emph{a priori} unbounded) values for the Nitsche parameter $\beta$. Nevertheless, with the Nitsche parameter large enough, in the $\beta$-norm both \eqref{eq:cont} and \eqref{eq:coer} hold (\emph{i.e.}, respectively continuity and coercivity), and only the continuous dependence on the data in the Lax--Milgram lemma can formally not be verified as the right-hand-side operator $l_h(\cdot)$ is not bounded in the $\beta$-norm such that $\tn u_h \tn_\beta$ cannot be \emph{a priori} bounded. The error estimate provided in \eqref{eq:opt_bound} does not hold for unstabilized formulations in combination with the $\beta$-norm, as $\beta \sim h_{T_\Omega}^{-1} \gtrsim h^{-1}$ such that \eqref{eq:part1bound}, and consequently \eqref{eq:interpol-energy}, do not apply. Note that \eqref{eq:interpol-energy} is also employed in the proof of the bound on the second part of the right-hand-side of \eqref{eq:errorsplit}. As discussed in Section~\ref{sec:introSec4}, this conclusion is also obtained in \cite{Prenter2018}, which shows that a best approximation property in the $\beta$-norm can be demonstrated following C\'{e}a's lemma, but that the problem lies in the error estimates because a bound on the error of the best approximation in the $\beta$-norm cannot be provided. As already demonstrated in Section~\ref{sec:conditioningAnalysis}, condition number bounds cannot be provided and tailored preconditioning is required. Related to the analysis presented in this section, both \eqref{eq:cond-inv} and \eqref{eq:cond-poin} are contingent on ghost-penalty stabilization or aggregation. For \eqref{eq:cond-inv} the bound cannot be provided as there is no \emph{a priori} upper bound on $\beta$. Nevertheless, this lack of an upper bound to the eigenvalues is generally not the cause of the conditioning problems of these methods, and this type of large eigenvalues only occurs on highly irregularly shaped cut elements. The main problem with the conditioning is that without additional stabilization, \eqref{eq:cond-poin} does not hold, and as demonstrated in Section~\ref{sec:conditioningAnalysis} it is this lack of a lower bound to the eigenvalues that generally causes the ill-conditioning of unstabilized immersed methods.\end{remark}

\section{Discussion and conclusion}\label{sec:discussion}

In this review, we considered the problems resulting from small cut elements in immersed finite element methods, as well as the solutions that have been developed to counter these problems. This section assesses the current state of the field and presents a discussion about future research directions. Similarities and differences between the different techniques that are discussed in this manuscript are treated in the respective sections, and we do not provide a qualitative comparison here. This is because, first, the presented techniques pursue different goals related to the different problems on small cut elements, \emph{i.e.}, stability and conditioning. Second, the approaches are not mutually exclusive and combinations can be made. In Section~\ref{sec:stabilization} the combination of aggregation and ghost-penalty stabilization is already discussed, and similarly it is possible to combine it with, \emph{e.g.}, the minimal stabilization procedure, which fixes the stability but not the conditioning, with a preconditioning approach.

The field of immersed finite element methods has substantially grown and developed over the past decade. The employed formulations have matured, and are currently based on rigorous theoretical and mathematical foundations. Furthermore, efficient parallel implementations have been developed that enable the application to large problems into the billions of degrees of freedom. Certain aspects are still challenging, however, which particularly manifest in the application to real-world problems. These challenging aspects, and the corresponding unresolved questions, form the basis of future research directions that we believe are important for the further advancement of the field of immersed finite element methods. In this section we discuss these open research questions in a general sense, and indicate it if these are specifically related to one of the techniques that are discussed in this manuscript.
    
One aspect that requires further consideration to advance the field of immersed finite element methods, is the detection of small features in the geometry or the solution. The decoupling of the geometry description from the finite element mesh makes the consideration of adaptive discretizations natural to immersed formulations. In this regard, it is notable that grid refinements in the vicinity of domain boundaries are commonly considered in the $hp$-adaptive finite cell method \cite{Zander2015,Angella2016,Duester2017,Elhaddad2017}. Such refinements are not required along all boundaries, but mainly in regions where small solution features such as large gradients are expected. For efficiency purposes, we therefore consider the further advancement of error estimators and error-estimation-based adaptivity schemes tailored to immersed formulations as presented in \cite{SaiEEA} a desired development. One complication in regard of \emph{a posteriori} error estimates is that the bilinear form is only coercive on discrete spaces, and that the dependence of the bilinear form on the discretization (both on the mesh size and on the order) complicates saturation assumptions. Besides the accuracy considerations related to enriched approximation spaces along small geometrical features, stabilized methods also require a certain level of geometry resolution, which makes local adaptive schemes well suited for the application of these techniques to problems on (scan-based) porous domains, as the refinements reduce the ratio of elements that are cut. Since element aggregation strongly ties degrees of freedom on badly cut elements to well posed degrees of freedom, local refinement procedures can be beneficial in cases where a large portion (or in certain regions even the majority) of the active mesh consists of badly cut elements. This is similar for ghost-penalty stabilization terms on domains with (relative to the mesh) thin structures, which can potentially make the domain artificially stiff in such cases. Note that in these cases the choice of the parameters $\eta^*$ (with aggregation; see Section~\ref{sec:aggregation}) and $\tau$ (with ghost-penalty stabilization; see Section~\ref{sec:ghost}) is particularly critical, and that in many cases problems can be avoided by a adequate parameter choice. Nevertheless, the accurate approximation of problems with very thin features in an unfitted setting remains a difficult problem for any method, and deserves further attention. Schwarz preconditioners are not affected by the matter of mesh resolution, but it should be noted that this only pertains to the solving of the linear system. With a natural boundary condition on the unfitted boundary and a preconditioner to resolve the conditioning problems, a relatively low level of mesh resolvement can therefore lead to reasonable solutions. It should be mentioned that this can still result in artificial couplings as described in \emph{e.g.}, \cite{Verhoosel2015,Prenter2020}, however. With essential boundary conditions on the unfitted boundary, such as with flow problems in porous media, the application of a preconditioner without an additional stabilization technique will lead to inaccurate approximations, as described in Section~\ref{sec:stabilization}.

While this manuscript focuses on single-field problems, mixed formulations, in particular flow problems, are an important application area of immersed finite element methods. The techniques presented in this manuscript are all applicable to such problems, and can retain or precondition inf-sup stable formulations (see, \emph{e.g.}, \cite{Badia2018b} for aggregation, \cite{Massing2014,Hoang2019} for ghost penalty, and \cite{Prenter2019} for preconditioning). An unresolved topic is the application of (pointwise) divergence-conforming or geometry-preserving discretizations for incompressible problems. Compatible pairs of approximation spaces that result in pointwise divergence-free approximations for boundary-fitted discretizations, generally lose this property in an unfitted setting. Furthermore, stabilization techniques complicate the formulation of pointwise divergence-free methods. Ghost-penalty terms interfere with the enforcement of incompressibility in the bilinear form, and it has not been thoroughly investigated how the existing discrete aggregation and extension operators affect the compatibility of structure-preserving pairs of approximation spaces. Besides these comments about $H$(div)-conforming discretizations for the conservation of geometry, similar considerations can be made about $H$(curl)-conforming discretizations when solving the Maxwell equations in electromagnetism.

An aspect that relates to both of the previous topics of mesh adaptivity and flow problems is the consideration of anisotropic grid refinements. To accurately capture boundary layers at a reasonable computational cost, boundary layer meshes are commonly anisotropic with a much finer resolution normal to the boundary. While the aspect of boundary layers in immersed methods has already been considered in multiple studies (see, \emph{e.g.}, \cite{Gerstenberger2010,Xu2016}), present understanding of this aspect is incomplete, and this is a topic that requires further investigation.

While only briefly considered in Section~\ref{sec:introStabilityConditioning}, a topic that is strongly related to small cut elements is explicit time integration for problems in dynamics. This is particularly interesting for crash simulations, as such problems are commonly solved by explicit time steppers \cite{Leidinger2020}. Although tools exist that result in stable time steps that are not affected by the cut elements (\emph{i.e.}, element aggregation, ghost penalty terms in the mass matrix, or lumping of the mass matrix \cite{Leidinger2019,Leidinger2020,burman2020explicit}), the research on this topic is not exhaustive. In particular, a comprehensive analysis of stable time-step sizes with different choices for the mass and stiffness forms (\emph{i.e.}, boundary condition enforcement technique, ghost-penalty stabilization of both forms, lumping of the mass matrix, etc.) does, to the authors' knowledge, not yet exist.

Another open question deals with the choice, or the optimization, of parameters in the presented approaches. This does not only concern, \emph{e.g.}, the parameters $\eta^*$ and $\tau$ that are explicitly contained in element aggregation and the ghost penalty, but also hidden parameters (or choices) in the setup such as the choice of the index blocks with Schwarz preconditioning. While these parameters generally do not affect the asymptotic behavior, well-chosen (or, conversely, poorly-chosen) parameters can significantly affect the performance pre-asymptotically. We therefore consider an optimal choice of (hidden) parameters an important topic for future research. This specifically applies to the dependence of this optimality on the basis-function type and the discretization order and, in particular for discretizations with local refinements or local order elevations, to the option of choosing parameters locally or element-wise (similar to the Nitsche parameter without stabilization). 

Over the past years, many well-developed, well-documented, and user-friendly open source codes have become available that aid the application of immersed finite element methods, a comprehensive list of which is beyond the scope of this work. These toolboxes support all facets of immersed finite element computations, such as pre-processing steps (in particular integration and the application of boundary conditions), the solution process itself, and also post-processing, where often a subtriangulation is employed like this is also done for higher-order boundary-fitted finite elements. Something that is not yet available is a graphical user interface that makes immersed finite element methods accessible also to non-expert users. Besides the developments in open source codes, also commercial software has adopted (aspects of) immersed methods. To provide some examples:
in Ansys Fluent the Immersed Boundary Method (IBM) is available \cite{Mittal2005,Wolfe2009};
Ansys LS-DYNA facilitates the application of trimmed isogeometric analysis in both shell and solid models for static and dynamic analyses \cite{Hartmann2019,Messmer2021};
in Abaqus it is possible to (partially) include and exclude elements from a computation with \emph{element progressive activation} \cite{Favaloro2017,Courter2017}, which is mainly intended for the simulation of additive manufacturing;
and the platform Hyperganic \cite{Hyperganic} employs immersed finite element techniques to perform simulations. The implementation of the stabilization methods discussed in this work can require tools that are not generally available in existing software and codes. For instance, face-based ghost-penalty stabilization requires the computation of jumps in normal gradients up to the order of approximation. For quadratic and higher-order discretizations, this functionality is generally not available, even for software frameworks that include the possibility to perform face loops and compute jumps for standard Discontinuous Galerkin (DG) formulations. The implementation of patch-based ghost-penalty stabilization or aggregation require the implementation of an aggregation routine and the computation of aggregate-wise terms. These terms can be implemented using standard element-wise terms if the finite element software provides functionality to enforce linear constraints. Many codes already provide this functionality as this is essential for local grid refinements and periodic boundary conditions (\emph{e.g.}, deal.ii \cite{Bangerth2007a}, fenics \cite{Alnaes2015}, gridap \cite{Badia2020-gridap}, fempar \cite{Badia2018-fempar}, netgen \cite{Schberl1997}, or dune \cite{Bastian2021}). While (geometric) multigrid routines are intrusive to the finite element software, the implementation of tailored Schwarz preconditioners is exceptionally non-intrusive and only requires the matrix itself and the basis function connectivity. Furthermore, as additive Schwarz preconditioning is a well-established concept also in boundary-fitted finite element methods, these preconditioners are available in many codes. The construction of the index blocks is not generally the same in existing software, however, and it should be noted that the effectivity for immersed methods relies on the specific construction of these blocks. Additionally, advanced procedures in which \emph{e.g.}, index blocks are only devised for basis functions with support on cut elements, require the availability of this information or a routine to extract it. 

By virtue of the stabilization, aggregation, and preconditioning techniques as discussed in this review, immersed finite element methods have developed into reliable and versatile computational analysis instruments. In our assessment, the potential of immersed methods and, in particular, their ability to provide new simulation workflows with minimal user intervention, has not been fully exploited yet. Another recommendation therefore pertains to the exploration of immersed methods in new application fields, such as hyperbolic systems, diffuse interface models, poro-elasticity, computational aeroacoustics, and phase field fracture.

\subsubsection*{Conflict of interest statement}
On behalf of all authors, the corresponding author states that there is no conflict of interest.

\bibliographystyle{unsrt}
\bibliography{references}

\end{document}